\documentclass[11pt]{amsart}

\usepackage{mathrsfs}
\usepackage{amsfonts}
\usepackage{amssymb}
\usepackage{amsxtra}
\usepackage{dsfont}
\usepackage[dvipsnames]{xcolor}
\usepackage[english,polish]{babel}
\usepackage{enumerate}

\usepackage{tikz}
\usetikzlibrary{arrows}

\usepackage[compress, sort]{cite}

\usepackage{newtxmath}
\usepackage{newtxtext}

\usepackage[margin=2.5cm, centering]{geometry}
\usepackage[colorlinks,citecolor=blue,urlcolor=blue,bookmarks=true]{hyperref}
\hypersetup{
pdfpagemode=UseNone,
pdfstartview=FitH,
pdfdisplaydoctitle=true,
pdfborder={0 0 0}, 
pdftitle={Orthogonal polynomials with periodically modulated recurrence coefficients in the Jordan block case},
pdfauthor={Grzegorz Świderski and Bartosz Trojan},
pdflang=en-US
}

\newcommand{\CC}{\mathbb{C}}
\newcommand{\ZZ}{\mathbb{Z}}
\newcommand{\sS}{\mathbb{S}}
\newcommand{\NN}{\mathbb{N}}
\newcommand{\RR}{\mathbb{R}}

\newcommand{\calR}{\mathcal{R}}

\newcommand{\calC}{\mathcal{C}}

\newcommand{\calO}{\mathcal{O}}

\newcommand{\calD}{\mathcal{D}}

\newcommand{\calE}{\mathcal{E}}
\newcommand{\frakB}{\mathfrak{B}}
\newcommand{\frakX}{\mathfrak{X}}
\newcommand{\frakA}{\mathfrak{A}}
\newcommand{\frakp}{\mathfrak{p}}

\newcommand{\scrD}{\mathscr{D}}

\newcommand{\vphi}{\varphi}

\newcommand{\Id}{\operatorname{Id}}

\newcommand{\sign}[1]{\operatorname{sign}({#1})}

\newcommand{\pl}[1]{\foreignlanguage{polish}{#1}}

\newcommand{\abs}[1]{\lvert {#1} \rvert}
\newcommand{\sprod}[2]{\langle {#1}, {#2} \rangle}

\newcommand{\tr}{\operatorname{tr}}

\DeclareMathOperator{\sinc}{sinc}

\newcommand{\sym}{\operatorname{sym}}

\newcommand{\GL}{\operatorname{GL}}
\newcommand{\SL}{\operatorname{SL}}
\newcommand{\discr}{\operatorname{discr}}

\newcommand{\Mat}{\operatorname{Mat}}

\newcommand{\sigmaEss}{\sigma_{\mathrm{ess}}}
\newcommand{\sigmaP}{\sigma_{\mathrm{p}}}

\newcommand{\sigmaAC}{\sigma_{\mathrm{ac}}}
\newcommand{\sigmaS}{\sigma_{\mathrm{sing}}}

\newcommand{\ud}{{\: \rm d}}
\newcommand{\ue}{\textrm{e}}
\newcommand{\supp}{\operatornamewithlimits{supp}}

\newtheorem{theorem}{Theorem}[section]

\newtheorem{proposition}[theorem]{Proposition}
\newtheorem{lemma}[theorem]{Lemma}
\newtheorem{corollary}[theorem]{Corollary}
\newtheorem{claim}[theorem]{Claim}

\theoremstyle{plain}
\newcounter{thm}

\newtheorem{main_theorem}[thm]{Theorem}

\numberwithin{equation}{section}

\theoremstyle{definition}
\newtheorem{example}[theorem]{Example}

\title[Orthogonal polynomials in the Jordan block case]
{Orthogonal polynomials with periodically modulated recurrence coefficients in the Jordan block case}

\author{Grzegorz \'Swiderski}
\address{
	Grzegorz \'Swiderski \\
	Faculty of Mathematics and Computer Science \\
	University of Wroc\l{}aw \\
	pl. Grunwaldzki 2/4 \\
	50-384 Wroc\l{}aw \\
	Poland
}
\email{grzegorz.swiderski@math.uni.wroc.pl}

\author{Bartosz Trojan}
\address{
	\pl{
		Bartosz Trojan\\
		Institute of Mathematics\\
		Polish Academy of Sciences\\
        ul. \'Sniadeckich 8\\
        00-696 Warszawa\\
        Poland}
}
\email{btrojan@impan.pl}

\keywords{Orthogonal polynomials, asymptotics, Tur\'an determinants, Christoffel functions, scaling limits}

\subjclass[2020]{Primary 42C05; Secondary 47B36}

\begin{document}
\selectlanguage{english}

\begin{abstract}
	We study orthogonal polynomials with periodically modulated recurrence coefficients when $0$ lies 
	on the hard edge of the spectrum of the corresponding periodic Jacobi matrix. In particular, we show
	that their orthogonality measure is purely absolutely continuous on a real half-line and purely
	discrete on its complement. Additionally, we provide the constructive formula for the density in terms 
	of Tur\'an determinants. Moreover, we determine the exact asymptotic behavior of the orthogonal
	polynomials. Finally, we study scaling limits of the Christoffel--Darboux kernel.
\end{abstract}

\maketitle

\section{Introduction}
Let $\mu$ be a probability measure on the real line with infinite support such that for every $n \in \NN_0$,
\[
	\text{the moments} \quad \int_{\RR} x^n \ud \mu(x) \quad \text{are finite}.
\]
Let us denote by $L^2(\RR, \mu)$ the Hilbert space of square-integrable functions equipped with the scalar product
\[
	\langle f, g \rangle = \int_\RR f(x) \overline{g(x)} \ud \mu(x).
\]
By performing on the sequence of monomials $(x^n : n \in \NN_0)$ the Gram--Schmidt orthogonalization process 
one obtains the sequence of polynomials $(p_n : n \in \NN_0)$ satisfying
\begin{equation}
	\label{eq:1a}
	\langle p_n, p_m \rangle = \delta_{nm}
\end{equation}
where $\delta_{nm}$ is the Kronecker delta. Moreover, $(p_n : n \in \NN_0)$ satisfies the following recurrence relation
\begin{equation}
	\label{eq:100}
	\begin{aligned} 
	p_0(x) &= 1, \qquad p_1(x) = \frac{x - b_0}{a_0}, \\
	x p_n(x) &= a_n p_{n+1}(x) + b_n p_n(x) + a_{n-1} p_{n-1}(x), \qquad n \geq 1
	\end{aligned}
\end{equation}
where
\[
	a_n = \langle x p_n, p_{n+1} \rangle, \qquad
	b_n = \langle x p_n, p_n \rangle, \qquad n \geq 0.
\]
Notice that for every $n$, $a_n > 0$ and $b_n \in \RR$. The pair $(a_n : n \in \NN_0)$ and $(b_n : n \in \NN_0)$
is called the \emph{Jacobi parameters}. Another central object of this article is the \emph{Christoffel--Darboux}
kernel $K_n$ which is defined as
\begin{equation}
	\label{eq:83}
	K_n(x, y) = \sum_{j=0}^n p_j(x) p_j(y).
\end{equation}
The classical topic in analysis is studying the asymptotic behavior of orthogonal polynomials $(p_n : n \in \NN_0)$
which often leads to computing the asymptotics of the Christoffel--Darboux kernel. To motivate the interest in the Christoffel--Darboux kernel see surveys \cite{Lubinsky2016} and \cite{Simon2008}.

When the starting point is the measure $\mu$ there is a rather good understanding of both 
orthogonal polynomials and the Christoffel--Darboux kernel. In particular, for the measures with compact
support, see e.g. \cite{Zhou2011, Kuijlaars2004, Totik2009, Lubinsky2009, Xu2011}; 
for the measures with the support being the whole real line, see e.g. \cite{Levin2009, Deift1999}; 
for the measures with the support being a half-line, see e.g. 
\cite{Vanlessen2007, Dai2014, Chen2019, Clarkson2018, Xu2014}; and for discrete measures, see e.g. the monograph 
\cite{VanAssche2018}.

Instead of taking the measure $\mu$ as the starting point one can consider polynomials $(p_n : n \in \NN_0)$ satisfying
the three-term recurrence relation \eqref{eq:100} for a given sequences $(a_n : n \in \NN_0)$ and $(b_n : n \in \NN_0)$
such that $a_n > 0$ and $b_n \in \RR$. In view of the Favard's theorem (see, e.g. \cite[Theorem II.3.1]{Chihara1978}),
there is a probability measure $\nu$ such that $(p_n : n \in \NN_0)$ is orthonormal in $L^2(\RR, \nu)$. The measure $\nu$
is unique, if and only if there is exactly one measure with the same moments as $\nu$. In such a case the measure
$\nu$ is called \emph{determinate} and will be denoted by $\mu$. A sufficient condition for the determinacy
of $\nu$ is given by the \emph{Carleman's condition}, that is
\begin{equation}
	\label{eq:37}
	\sum_{n = 0}^\infty \frac{1}{a_n} = \infty
\end{equation}
(see, e.g. \cite[Corollary 6.19]{Schmudgen2017}). 
Let us recall that the orthogonality measure has compact support, if and only if the Jacobi parameters are bounded, namely
\begin{equation} \label{eq:130}
	\sup_{n \in \NN_0} a_n < \infty \quad \text{and} \quad
	\sup_{n \in \NN_0} |b_n| < \infty.
\end{equation}
The corresponding theory is very well developed. In particular it covers the cases of constant, periodic, or almost periodic 
Jacobi parameters and compact perturbations thereof, see, e.g. \cite{Christiansen2010, Simon2003, Simon2010, Remling2011}. 
However, to the best of the authors’ knowledge, the most techniques used in the bounded case, compare the recent monograph 
\cite{Simon2010Book}, cannot be adapted when Jacobi parameters are \emph{unbounded}, that is when the condition \eqref{eq:130}
is violated. 

In this article we are exclusively interested in unbounded Jacobi parameters that belong to the class
of periodically modulated sequences. This class was introduced in \cite{JanasNaboko2002}, and it is systematically studied
since then. To be precise, let $N$ be a positive integer. We say that the Jacobi parameters $(a_n : n \in \NN_0)$ and
$(b_n : n \in \NN_0)$ are \emph{$N$-periodically modulated} if there are two $N$-periodic sequences $(\alpha_n : n \in \ZZ)$
and $(\beta_n : n \in \ZZ)$ of positive and real numbers, respectively, such that
\begin{enumerate}[a)]
	\item
	$\begin{aligned}[b]
	\lim_{n \to \infty} a_n = \infty
	\end{aligned},$
	\item
	$\begin{aligned}[b]
	\lim_{n \to \infty} \bigg| \frac{a_{n-1}}{a_n} - \frac{\alpha_{n-1}}{\alpha_n} \bigg| = 0
	\end{aligned},$
	\item
	$\begin{aligned}[b]
	\lim_{n \to \infty} \bigg| \frac{b_n}{a_n} - \frac{\beta_n}{\alpha_n} \bigg| = 0
	\end{aligned}.$
\end{enumerate}
It turns out that the properties of the measure $\mu$ corresponding to $N$-periodically modulated
Jacobi parameters depend on the matrix $\frakX_0(0)$ where
\begin{equation} \label{eq:113}
	\frakX_0(x) = 
	\prod_{j=0}^{N-1}
	\begin{pmatrix}
		0 & 1 \\
		-\frac{\alpha_{j-1}}{\alpha_i} & \frac{x - \beta_j}{\alpha_j}
	\end{pmatrix}.
\end{equation}
More specifically, one can distinguish four cases:
\begin{enumerate}[I.]
\item \label{perMod:I}
if $|\tr \frakX_0(0)|<2$, then, under some regularity assumptions on Jacobi parameters, the measure $\mu$ is purely absolutely continuous on $\RR$ with positive continuous density, see 
\cite{JanasNaboko2002, PeriodicI, PeriodicII, SwiderskiTrojan2019, JanasNaboko2001}; 

\item \label{perMod:II} 
if $|\tr \frakX_0(0)|=2$, then we have two subcases: 
\begin{enumerate}[a)]
\item  \label{perMod:IIa}
if $\frakX_0(0)$ is diagonalizable, then, under some regularity assumptions on Jacobi parameters,  there is a compact interval $I \subset \RR$ such that the measure $\mu$ is purely absolutely continuous on $\RR \setminus I$ with positive continuous density and it is purely discrete in the interior of $I$, see \cite{Dombrowski2004, Dombrowski2009, DombrowskiJanasMoszynskiEtAl2004, DombrowskiPedersen2002a, 
DombrowskiPedersen2002, JanasMoszynski2003, JanasNabokoStolz2004, Sahbani2016, Janas2012, PeriodicII, PeriodicIII,
Discrete};

\item \label{perMod:IIb}
if $\frakX_0(0)$ is \emph{not} diagonalizable, then only certain \emph{examples} have been studied, see 
\cite{Damanik2007, Janas2001, Janas2009, Motyka2015, Naboko2009, Naboko2010, Simonov2007, DombrowskiPedersen1995,
Naboko2019, Yafaev2020a, Pchelintseva2008, Motyka2014, Motyka2015}. Then usually the measure $\mu$ is purely absolutely
continuous on a real half-line and discrete on its complement;
\end{enumerate}

\item \label{perMod:III}
if $|\tr \frakX_0(0)|>2$, then, under some regularity assumptions on Jacobi parameters, the measure $\mu$ is purely
discrete with the support having no finite accumulation points, see \cite{Discrete, JanasNaboko2002, HintonLewis1978,
Szwarc2002}; 
\end{enumerate}
One can describe these four cases geometrically. Specifically, we have 
\[
	(\tr \frakX_0)^{-1} \big( (-2, 2) \big) = \bigcup_{j=1}^N I_j
\]
where $I_j$ are disjoint open non-empty bounded intervals whose closures might touch each other. Let us denote
\begin{equation} \label{eq:1}
	I_j = (x_{2j-1}, x_{2j}) \qquad (j=1,2, \ldots, N),
\end{equation}
where the sequence $(x_k : k = 1, 2, \ldots, 2N)$ is increasing. Then we are in the case \ref{perMod:I} if $0$ belongs to some interval \eqref{eq:1}, in the case \ref{perMod:IIa} if $0$ lies on the boundary of exactly two intervals, in the case \ref{perMod:IIb} if $0$ lies on the boundary of exactly one interval and in \ref{perMod:III} in the remaining cases. An example for $N=4$ is presented in Figure~\ref{img:1}.
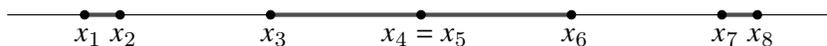
\begin{figure}[h!]
	\centering
	\begin{tikzpicture}
		\draw[->, thin, black] (-5.5,0) -- (5.5, 0);  
		\draw[-,ultra thick, black!70] (-4.472135954,0) -- (-4,0);
		\draw[-,ultra thick, black!70] (-2,0) -- (2,0);
		\draw[-,ultra thick, black!70] (4,0) -- (4.472135954,0);
		\filldraw[black] (0,0) circle (1.6pt);
		\draw (0.05,-0.3) node {$x_4=x_5$};
		\filldraw[black] (-4.472135954,0) circle (1.6pt);
		\draw (-4.422135954,-0.3) node {$x_1$};
		\filldraw[black] (-4,0) circle (1.6pt);
		\draw (-3.95,-0.3) node {$x_2$};		
		\filldraw[black] (-2,0) circle (1.6pt);
		\draw (-1.95,-0.3) node {$x_3$};
		\filldraw[black] (2,0) circle (1.6pt);
		\draw (2.05,-0.3) node {$x_6$};
		\filldraw[black] (4,0) circle (1.6pt);
		\draw (4.05,-0.3) node {$x_7$};
		\filldraw[black] (4.472135954,0) circle (1.6pt);
		\draw (4.522135954,-0.3) node {$x_8$};
	\end{tikzpicture}
	\caption{An example for $N=4$. If $0 = x_4$, then we are in the case \ref{perMod:IIa}, while $0 \in \{x_1, x_2, x_3, x_6, x_7, x_8 \}$ corresponds to the case \ref{perMod:IIb}.}
	\label{img:1}
\end{figure}

In this article we focus on the case \ref{perMod:IIb}. Since nowadays there is a rather good
understanding of the remaining cases, our results complete the study of the basic properties of periodic 
modulations.

Before we go any further let us recall a few definitions. We say that a sequence of real numbers $(x_n : n \in \NN)$
belongs to $\calD_1^N$, if
\[
	\sum_{n=1}^\infty |x_{n+N} - x_n| < \infty.
\]
If $N = 1$ we usually drop the superscript from the notation. For a matrix
\[
	X = 
	\begin{pmatrix}
		x_{1, 1} & x_{1, 2} \\
		x_{2, 1} & x_{2, 2}
	\end{pmatrix}
\]
we set $[X]_{i, j} = x_{i, j}$. Let
\[
	\sinc(x) = 
	\begin{cases}
		\frac{\sin(x)}{x} & \text{if } x \neq 0,\\
		1 & \text{otherwise.}
	\end{cases}
\]
In our first result we identify the location of the discrete part of the measure $\mu$, see Theorem~\ref{thm:8} for
the proof.
\begin{main_theorem} 
	\label{thm:A}
	Suppose that Jacobi parameters $(a_n : n \in \NN_0)$ and $(b_n : n \in \NN_0)$ are $N$-periodically modulated
	and such that $\frakX_0(0)$ is \emph{not} diagonalizable. Assume further that
	\[
		\bigg(\frac{\alpha_{n-1}}{\alpha_n} a_n - a_{n-1} : n \in \NN \bigg),
		\bigg(\frac{\beta_n}{\alpha_n} a_n - b_n : n \in \NN\bigg),
		\bigg(\frac{1}{\sqrt{a_n}} : n \in \NN\bigg) \in \calD_1^N.	
	\]
	Then there is an explicit polynomial $\tau$ of degree $1$ (see \eqref{eq:61b}) such that
	the measure $\mu$ restricted to
	\[
		\Lambda_+ = \tau^{-1} \big( (0, \infty) \big),
	\]
	is purely discrete. More precisely\footnote{For a set $X \subset \RR$ by $X'$ we denote the set of its accumulation points.}, $(\supp \mu)' \cap \Lambda_+ = \emptyset$.
\end{main_theorem}

In the next theorem we study convergence of $N$-shifted Tur\'an determinants. We prove 
that they are related to the density of the measure $\mu$. In this manner, we constructively
prove that $\mu$ is absolutely continuous on the set
\[
	\Lambda_- = \tau^{-1} \big( (-\infty, 0) \big).
\]
Let us recall that $N$-shifted Tur\'an determinant is defined as
\[
	\scrD_n(x) =
	\det
	\begin{pmatrix}
		p_{n+N-1}(x) & p_{n-1}(x) \\
		p_{n+N}(x) & p_n(x)
	\end{pmatrix}
	=
	p_{n}(x) p_{n+N-1}(x) - p_{n-1}(x) p_{n+N}(x).
\]
The approach to proving absolute continuity of the measure $\mu$ with compact support by means of 
Tur\'an determinants has been started in \cite{Nevai1979} and later developed in \cite{Nevai1983, Nevai1987,
GeronimoVanAssche1991}, see also the survey \cite{Nevai1992}. In \cite{PeriodicII, PeriodicIII, SwiderskiTrojan2019} an extension to measures
with unbounded support has been accomplished. For the proof of the following theorem see Theorems~\ref{thm:5}
and \ref{thm:4}.
\begin{main_theorem} 
	\label{thm:B}
	Suppose that the hypotheses of Theorem~\ref{thm:A} is satisfied. Let $i \in \{0, 1, \ldots, N-1 \}$. Then
	the limit
	\begin{equation} 
		\label{eq:95c}
		g_i(x) = \lim_{\stackrel{n \to \infty}{n \equiv i \bmod N}} 
		a_{n+N-1}^{3/2} \big| \scrD_n(x) \big|
	\end{equation}
	exists for any $x \in \Lambda_-$, and defines a continuous positive function. Moreover, the convergence 
	is locally uniform. If 
	\begin{equation}
		\label{eq:95a}
		\lim_{n \to \infty} \big( a_{n+N} - a_n \big) = 0,
	\end{equation}
	then the measure $\mu$ is purely absolutely continuous on $\Lambda_-$ with the density
	\begin{equation}
		\label{eq:95b}
		\mu'(x) = \frac{\sqrt{\alpha_{i-1} |\tau(x)|}}{\pi g_i(x)}, \qquad x \in \Lambda_-.
	\end{equation}
\end{main_theorem}
Let us note that the exponent of $a_{n+N-1}$ in front of the Tur\'an determinant 
in \eqref{eq:95c} equals $\frac{3}{2}$. In the cases \ref{perMod:I} and \ref{perMod:IIa} the exponent equals
$1$ and $2$, respectively (see \cite[Theorem B]{SwiderskiTrojan2019} and \cite[Theorem D]{PeriodicIII}).

However, there are classical orthogonal polynomials with Jacobi parameters which do not satisfy \eqref{eq:95a},
for example, the sequence of Laguerre polynomials for which
\[
	a_n = \sqrt{(n+1)(n + 1 + \lambda)}, \qquad
	b_n = 2n+1+\lambda
\]
where $\lambda>-1$, thus
\[
	\lim_{n \to \infty} ( a_{n+1} - a_n ) = 1.
\]
Let us mention that in the sequel \cite{jordan2} we show extensions of Theorems \ref{thm:A} and \ref{thm:B}
for a larger class of Jacobi parameters without condition \eqref{eq:95a} covering the example above.

We also study asymptotic behavior of the orthogonal polynomials in the form similar to 
\cite{GeronimoVanAssche1991, Totik1985, SwiderskiTrojan2019, ChristoffelII}, see also the survey \cite{VanAssche1990}.
For the proof of the next theorem see Theorem~\ref{thm:6}.
\begin{main_theorem} 
	\label{thm:C}
	Suppose that the hypotheses of Theorem~\ref{thm:A} and \eqref{eq:95a} are satisfied. 
	Let $i \in \{0, 1, \ldots, N-1 \}$. Then for any compact set $K \subset \Lambda_-$, there are a 
	continuous real-valued function $\chi$ and $j_0 \geq 1$ such that for all $j > j_0$,
	\begin{equation}
		\label{eq:96}
		\sqrt[4]{a_{(j+1)N+i-1}} p_{jN+i}(x)
		=
		\sqrt{\frac{\big| [\frakX_i(0)]_{2,1} \big|}{ \pi \mu'(x) \sqrt{\alpha_{i-1} |\tau(x)|}}}
		\sin\Big(\sum_{k=j_0}^{j-1} \theta_k(x) + \chi(x)\Big)
		+ o_K(1)
	\end{equation}
	where $\theta_k : K \to (0, \pi)$ are certain continuous functions, and $o_K(1)$ tends to $0$ uniformly on $K$.
\end{main_theorem}
Let us note that in \eqref{eq:96} the exponent of $a_{(j+1)N+i-1}$ equals $\frac{1}{4}$ and it is different than 
in the cases \ref{perMod:I} and \ref{perMod:IIa} where it equals $\frac{1}{2}$
(see \cite[Theorem C]{SwiderskiTrojan2019} and \cite[Theorem C]{ChristoffelII}).

Finally, we prove scaling limits of the Christoffel--Darboux kernel in the form analogous to \cite{ChristoffelI, ChristoffelII}. 
\begin{main_theorem}
	\label{thm:D}
	Under the hypotheses of Theorem~\ref{thm:C} we have
	\begin{equation} \label{eq:153}
		\lim_{n \to \infty} 
		\frac{1}{\rho_n} K_n \Big( x+\frac{u}{\rho_n}, x + \frac{v}{\rho_n} \Big) 
		=
		\frac{\upsilon(x)}{\mu'(x)} \sinc \big( (u-v) \pi \upsilon(x) \big)
	\end{equation}
	locally uniformly with respect to $(x, u, v) \in \Lambda_- \times \RR^2$ where
	\[
		\rho_n = \sum_{j=0}^n \sqrt{\frac{\alpha_j}{a_j}},
		\quad
		\text{and}\quad
		\upsilon(x) = \frac{|\tr \frakX_0'(0)|}{2 \pi N \sqrt{|\tau(x)|}}.
	\]
\end{main_theorem}
For the proof of Theorem \ref{thm:D}, see Theorem \ref{thm:7}. The definition of $\rho_n$ reflects the unusual asymptotic
behavior of the orthogonal polynomials. Indeed, in the cases \ref{perMod:I} and \ref{perMod:IIa} we have
\[
	\rho_n = \sum_{j=0}^n \frac{\alpha_j}{a_j}
\] 
(see, \cite[Theorem C]{ChristoffelI} and \cite[Theorem D]{ChristoffelII}). Let us note that the 
definitions of $\upsilon$ are different in all of the three cases.

The study of sequences of the form \eqref{eq:153} originates in Random Matrix Theory. They describes local fluctuations of eigenvalues of large random Hermitian matrices whose distribution is invariant under unitary transformations (the so-called \emph{unitary invariant ensembles}), see e.g. the recent monograph \cite{Anderson2010} and the survey \cite{Lubinsky2016} for more details. In fact, the natural framework of the limits \eqref{eq:153} is the subclass of determinantal point processes, called \emph{orthogonal polynomial ensembles}, which contains many models considered in physics, statistical mechanics, probability theory and combinatorics, see the survey \cite{Konig2005} for more details. For such processes the correlation functions can be expressed in terms of determinants of the Christoffel--Darboux kernels. Then the limits \eqref{eq:153} imply that around any point $x \in \Lambda_-$ the fluctuations at the scale $\rho_n$ are weakly convergent to the sine process, see e.g. \cite[Proposition 3.10]{Shirai2003}. Let us mention that we can apply Theorem \ref{thm:D} to measures studied in \cite{Vanlessen2007} using Riemann--Hilbert techniques, 
which results in the new and previously non-considered regime, namely when $x$ stays in a compact subset of $(0, \infty)$. 

Apart from the application to Random Matrix Theory, we can present a standard consequence of Theorem \ref{thm:D} that is
description of the spacing of zeros of orthogonal polynomials. To be more precise, for any $x_0 \in \supp(\mu)$ let us denote by 
$x^{(n)}_j(x_0)$ the zeros of $p_n$ labeled so that
\[
	\ldots < x_{-1}^{(n)}(x_0) < x_0 \leq x_{0}^{(n)}(x_0) < x_{1}^{(n)}(x_0) < \ldots.
\]
Then Freud--Levin theorem (see \cite{Levin2008a} and \cite[Theorem 23.1]{Simon2008}) stems that \eqref{eq:153} implies that for any $x_0 \in \Lambda_-$ and any $j \in \ZZ$ we have the following spacing of zeros
\[
	\lim_{n \to \infty} \rho_n \Big( x_{j+1}^{(n)}(x_0) - x_{j}^{(n)}(x_0) \Big) =
	\frac{1}{\upsilon(x_0)}.
\]
Observe that the limit does not involve the value of $\mu'$ at all.

Next, let us briefly describe a class of \emph{bounded} Jacobi parameters which is closely related to periodically modulated
Jacobi parameters considered in this article. Let $N$ be a positive integer. We say that the Jacobi parameters 
$(a_n : n \in \NN_0)$ and $(b_n : n \in \NN_0)$ are \emph{$N$-asymptotically periodic} if there are two $N$-periodic sequences  
$(\alpha_n : n \in \NN_0)$ and $(\beta_n : n \in \NN_0)$ of positive and real numbers, respectively, such that
\begin{equation} \label{eq:127}
	\lim_{n \to \infty} |a_n - \alpha_n| = 0 \quad \text{and} \quad
	\lim_{n \to \infty} |b_n - \beta_n| = 0.
\end{equation}
A brief description of the current state of the art for the class of asymptotically periodic Jacobi parameters can be found in
\cite[Section 7.1]{SwiderskiTrojan2019}. In particular, we always have $(\supp \mu)' = (\tr \frakX_0)^{-1} \big( [-2,2] \big)$
where $\frakX_0$ is the matrix defined in \eqref{eq:113}. Moreover, we can distinguish three subsets of the real line:
\begin{enumerate}[A.]
\item \label{per:A}
	$(\tr \frakX_0)^{-1} \big( (-2,2) \big)$: Under some regularity assumptions on Jacobi parameters, e.g. 
	$(a_n : n \in \NN_0), (b_n : n \in \NN_0) \in \calD_1^N$, the measure $\mu$ is absolutely continuous with positive density
	on this set, see \cite[Theorem 6]{GeronimoVanAssche1991};
\item \label{per:B}
	$(\tr \frakX_0)^{-1} \big( \{ -2,2 \} \big)$: The set contains at most $2N$ points. In fact, this set is not sufficiently
	well-understood, see e.g. \cite{Lukic2019};
\item \label{per:C}
	$\RR \setminus (\tr \frakX_0)^{-1} \big( [-2,2] \big)$: The measure $\mu$ restricted to this set is always discrete.
\end{enumerate}
Observe that \ref{per:A}, \ref{per:B} and \ref{per:C} is an analogue of \ref{perMod:I}, \ref{perMod:II} and \ref{perMod:III}, 
respectively. However, there is one crucial difference: for periodic modulations everything depends only on the properties of 
$\frakX_0(0)$ whereas for asymptotically periodic case all of the sets \ref{per:A}--\ref{per:C} are always present. The proof
of discreteness of $\mu$ in the set \ref{per:C} is an easy consequence of the Weyl perturbation theorem applied to the
associated Jacobi matrix and it uses \eqref{eq:127} only. In comparison, for unbounded Jacobi parameters usually it is not
possible to use this technique, see \cite{Kupin2018}. Instead to prove discreteness of the measure $\mu$ one uses the asymptotic
analysis, see \cite{Discrete}. We implement this idea in the proof of Theorem~\ref{thm:A}. The analogues of Theorems \ref{thm:B}
and \ref{thm:C} in the set~\ref{per:A} are proven in \cite{GeronimoVanAssche1991} by means of discrete Green functions. 
Recently, in \cite{SwiderskiTrojan2019} we found an alternative approach based on diagonalization of transfer matrices. We extend
this idea in Theorems \ref{thm:B} and \ref{thm:C}. Finally, the analogue of Theorem~\ref{thm:D} in the set \ref{per:A}
has been proven in \cite[Theorem 2.2]{Totik2009} with a help of logarithmic potential theory. Since this technique is not
available for unbounded Jacobi parameters, in \cite{ChristoffelI, ChristoffelII} we took a different approach exploiting
asymptotics of orthogonal polynomials. In Theorem~\ref{thm:D} we adapt this idea to the current setup.

Let us comment the relation of our results to the available literature. In Theorems~\ref{thm:A}--\ref{thm:D}
we consider a wide class of $N$-periodically modulated Jacobi parameters satisfying regularity conditions expressed
in terms of the total variation of certain sequences. 
In the case $N = 1$, the most general class of Jacobi parameters has been studied in the articles
\cite{Motyka2014,Motyka2015} where $\calD_1$-type condition has been combined with $\ell^1$-type condition.
Under certain additional hypotheses, the author obtained a weaker variant of Theorem~\ref{thm:A} as well as asymptotics of
generalized eigenvectors and absolute continuity of the measure $\mu$. However, there are no analogues of
Theorems~\ref{thm:B}, \ref{thm:C} nor \ref{thm:D}. In a recent preprint \cite{Naboko2019} the authors proved a variant of
Theorem~\ref{thm:C} for $\ell^1$-type perturbations of the sequences $a_n=(n+1)^\gamma$, $b_n = -2(n+1)^\gamma$
for $\gamma \in (0,1)$. Since in this context $\calD_1$-type conditions do not cover $\ell^1$-type perturbations, 
in Section \ref{sec:9} we generalize Theorems~\ref{thm:A}--\ref{thm:D} to $\ell^1$-type perturbations of sequences
satisfying $\calD_1$-type conditions. In particular, our class of Jacobi parameters properly contains those investigated
in \cite{Motyka2014, Motyka2015, Naboko2019}, see Section~\ref{sec:10:1} for details. Moreover, the way we deal with
$\ell^1$-type perturbations may have applications beyond the current setup. We also believe that $\ell^1$-type
perturbations are rather straightforward to obtain provided that one has a good understanding of the unperturbed sequences.
For this reason we consider $\calD_1$-type regularity as genuinely more natural. 
Lastly, in the case $N > 1$, only specific \emph{examples} have been studied for $N = 2$ giving variants of Theorem 
\ref{thm:A} and the absolute continuity of the measure $\mu$ with a help of subordinacy theory, see Section
\ref{sec:10:2} for details.

We briefly outline the proofs: In this article we adapt techniques that were successful in the generic case
\ref{perMod:I} as well as in the soft edge regime \ref{perMod:IIa}. However before adapting them we have to introduce
a proper modification to the recurrence system to obtain a sequence of transfer matrices which is uniformly
diagonalizable. This is the main novelty of the paper. We call it \emph{shifted conjugation.} The resulting transfer
matrices have a form similar to that appearing in the soft edge regime but with $a_n$ replaced by $\sqrt{a_n}$, 
see Section~\ref{sec:3} for details. Our method is simpler than discrete variants of Wentzel--Kramers--Brillouin
approximation which is the standard technique used in the case of Jordan block \ref{perMod:IIb} (see \cite{Motyka2014} and the references therein). 

Going back to the description of the proofs, for any compact set $K \subset \Lambda_+$, to the conjugated system we apply
the recently developed Levinson's type theorem, see \cite{Discrete}, to produce a family of generalized eigenvectors 
(see Section~\ref{sec:2b} for the definition) $(u_n(x) : n \in \NN_0)$, $x \in K$, such that
\[
	\sum_{n = 0}^\infty \sup_{x \in K}{|u_n(x)|^2} < \infty.
\]
Using the arguments as in \cite{Silva2007}, we deduce that the measure $\mu$ restricted to $\Lambda_+$ is purely
atomic and all accumulation points of its support are on the boundary of $\Lambda_+$, see Theorem \ref{thm:8}.
To study the convergence of $N$-shifted Tur\'an determinants, first we show that the corresponding objects defined for
the conjugated system multiplied by the correcting factor are close to Tur\'an determinants for the original system.
Then we prove that they constitute a uniform Cauchy sequence, see Theorem \ref{thm:5}. To identify the limit we adapt
the approximation procedure used in \cite{SwiderskiTrojan2019}, which is inspired by \cite{AptekarevGeronimo2016} and 
further developed in \cite{PeriodicII,ChristoffelII}, see Theorem \ref{thm:4}. We observe that the conjugated system
satisfies uniform diagonalization, thus motivated by the techniques developed in \cite{ChristoffelII}, we manage to
describe the asymptotic behavior of the generalized eigenvectors, see Theorem \ref{thm:3}.
However, by this method we cannot determine the factor $|\vphi|$ which is computed in Theorem \ref{thm:6} once again
with the help of the approximation procedure. In the presence of $\ell^1$-perturbation we show that orthogonal
polynomials can be expressed as generalized eigenvectors for unperturbed Jacobi parameters for a certain initial
conditions. In Section \ref{sec:9} we explicitly construct the mapping which describes how to choose the initial 
conditions. It turns out that the shifted conjugation can be performed with matrices constructed for unperturbed system.
All of this allows us to approximate Tur\'an determinants by generalized Tur\'an determinants for unperturbed sequences,
as well as to find the asymptotic behavior of orthogonal polynomials.

The paper is organized as follows: In Section~\ref{sec:2} we fix notation and formulate basic definitions.
Section~\ref{sec:3} is devoted to shifted conjugation. In Section \ref{sec:4} we study the measure $\mu$ restricted
to $\Lambda_+$. The convergence of $N$-shifted generalized Tur\'an determinants is proved in Section~\ref{sec:5}.
In the following section, we study the asymptotic behavior of the orthogonal polynomials. In Section~\ref{sec:7}
we describe the approximation procedure which is used in determining the limit of Tur\'an determinants and the exact
asymptotics of the polynomials. In Section~\ref{sec:8}, we investigate the convergence of the Christoffel--Darboux kernel.
In Section~\ref{sec:9} we show how to extend these results in the presence of the $\ell^1$ perturbation.
Finally, Section~\ref{sec:10} contains several examples illustrating the results of this paper and discuss how they
are related to the available literature.

\subsection*{Notation}
By $\NN$ we denote the set of positive integers and $\NN_0 = \NN \cup \{0\}$. Throughout the whole article, we write 
$A \lesssim B$ if there is an absolute constant $c>0$ such that $A \le cB$. We write $A \asymp B$ if $A \lesssim B$ and
$B \lesssim A$. Moreover, $c$ stands for a positive constant whose value may vary from occurrence to occurrence. For any
compact set $K$, by $o_K(1)$ we denote the class of functions $f_n : K \rightarrow \RR$ such that
$\lim_{n \to \infty} f_n(x) = 0$ uniformly with respect to $x \in K$.

\subsection*{Acknowledgment}
The first author was supported by long term structural funding -- Methusalem grant of the Flemish Government.
This work was completed while the first author was a postdoctoral fellow at KU Leuven.

\section{Preliminaries}
\label{sec:2}

\subsection{Stolz class}
Let $N$ be a positive integer. We say that a sequence $(x_n : n \in \NN)$ of vectors from a normed 
vector space $V$ belongs to $\calD_1^N (V)$, if
\[
	\sum_{n=1}^\infty \| x_{n+N} - x_n \| < \infty.
\]
Let us recall that $\calD_1^N(V)$ is an algebra provided $V$ is a normed algebra.
If $N=1$, then we usually omit the superscript. If $V$ is the real line with Euclidean norm we abbreviate 
$\calD_{1} = \calD_{1}(V)$. Given a compact set $K \subset \CC$ and a normed vector space $R$, we 
denote by $\calD_{1}(K, R)$ the case when $V$ is the space of all continuous mappings from $K$ to 
$R$ equipped with the supremum norm. 

\subsection{Finite matrices}
\label{sec:2a}
By $\Mat(2, \CC)$ and $\Mat(2, \RR)$ we denote the space of $2 \times 2$ matrices with complex and real entries, respectively, equipped with the spectral norm. Next, $\GL(2, \RR)$ and $\SL(2, \RR)$ consist of all matrices from $\Mat(2, \RR)$ which are invertible and of determinant equal $1$, respectively.

Let us recall that symmetrization and the discriminant of a matrix $A \in \Mat(2, \CC)$, are defined as
\[
	\sym(A) = \frac{1}{2} A + \frac{1}{2} A^*, \quad\text{and}\quad
	\discr(A) = (\tr A)^2 - 4 \det A,
\]
respectively. Here $A^*$ denotes the Hermitian transpose of the matrix $A$.

By $\{ e_1, e_2 \}$ we denote the standard orthonormal basis of $\CC^2$, i.e.
\begin{equation} \label{eq:109}
	e_1 =
	\begin{pmatrix}
		1 \\
		0
	\end{pmatrix} \quad \text{and} \quad
	e_2 =
	\begin{pmatrix}
		0 \\
		1
	\end{pmatrix}.
\end{equation}
Let $\sprod{\cdot}{\cdot}$ be the standard Hermitian inner product in $\CC^2$.

For a sequence of square matrices $(C_n : n_0 \leq n \leq n_1)$ we set
\[
	\prod_{k = n_0}^{n_1} C_k = C_{n_1} C_{n_1-1} \cdots C_{n_0}.
\]
Moreover, if $n_0 > n_1$, we set
\[
	\prod_{k = n_0}^{n_1} C_k = \Id.
\]
A matrix $X \in \SL(2, \RR)$ is \emph{non-trivial parabolic} if it is not a multiple of the identity and
$|\tr X| = 2$. Then $X$ is conjugated to
\[
	\varepsilon
	\begin{pmatrix}
		0 & 1 \\
		-1 & 2
	\end{pmatrix}
\]
where $\varepsilon = \sign{\tr X}$. Moreover, if
\[
	X = \varepsilon T 
	\begin{pmatrix}
		0 & 1 \\
		-1 & 2
	\end{pmatrix}
	T^{-1}
\]
then $X$ equals
\[
	X = \frac{\varepsilon}{\det T}
	\begin{pmatrix}
		\det T - (T_{11} + T_{12})(T_{21} + T_{22}) & (T_{11}+T_{12})^2 \\
		-(T_{21} + T_{22})^2 & \det T + (T_{11} + T_{12})(T_{21} + T_{22})
	\end{pmatrix}.
\]
In particular,
\begin{equation}
	\label{eq:12}
	\frac{(T_{11} + T_{12})(T_{21} +T_{22})}{\det T} = 1 - \varepsilon X_{11}	
\end{equation}
and
\begin{equation}
	\label{eq:11}
	\frac{(T_{21} + T_{22})^2}{\det T} = -\varepsilon X_{21}.
\end{equation}

\subsection{Jacobi matrices}
\label{sec:2b}
Given two sequences $a = (a_n : n \in \NN_0)$ and $b = (b_n : n \in \NN_0)$ of positive and real numbers, respectively, by $A$ we define the closure in $\ell^2$ of the operator acting on sequences having finite support by the matrix
\[
	\begin{pmatrix}
		b_0 & a_0 & 0   & 0      &\ldots \\
		a_0 & b_1 & a_1 & 0       & \ldots \\
		0   & a_1 & b_2 & a_2     & \ldots \\
		0   & 0   & a_2 & b_3   &  \\
		\vdots & \vdots & \vdots  &  & \ddots
	\end{pmatrix}.
\]
The operator $A$ is called \emph{Jacobi matrix}. If the Carleman condition \eqref{eq:37}
is satisfied then the operator $A$ is self-adjoint (see e.g. \cite[Corollary 6.19]{Schmudgen2017}).
Let us denote by $E_A$ its spectral resolution of the identity. Then for any Borel subset $B \subset \RR$, we set
\[
	\mu(B) = \langle E_A(B) \delta_0, \delta_0 \rangle_{\ell^2}
\]
where $\delta_0$ is the sequence having $1$ on the $0$th position and $0$ elsewhere. The polynomials $(p_n : n \in \NN_0)$
form an orthonormal basis of $L^2(\RR, \mu)$. By $\sigma(A), \sigmaP(A), \sigmaS(A), \sigmaAC(A)$ and $\sigmaEss(A)$ we denote the spectrum, the point spectrum, the singular spectrum, the absolutely continuous spectrum and the essential spectrum of $A$, respectively.

A sequence $(u_n : n \in \NN_0)$ is a \emph{generalized eigenvector} associated to
$x \in \CC$ and corresponding to $\eta \in \RR^2 \setminus \{0\}$, if the sequence of vectors
\begin{align*}
	\vec{u}_0 &= \eta, \\
	\vec{u}_n &= 
	\begin{pmatrix}
		u_{n-1} \\
		u_n
	\end{pmatrix}, \quad n \geq 1,
\end{align*}
satisfies
\begin{equation} 
	\label{eq:108a}
	\vec{u}_{n+1} = B_n(x) \vec{u}_n, \quad n \geq 0,
\end{equation}
where $B_n$ is the \emph{transfer matrix} defined as
\begin{equation} 
	\label{eq:108}
	\begin{aligned}
	B_0(x) &= 
	\begin{pmatrix}
		0 & 1 \\
		-\frac{1}{a_0} & \frac{x-b_0}{a_0}
	\end{pmatrix} \\
	B_n(x) &= 
	\begin{pmatrix}
		0 & 1 \\
		-\frac{a_{n-1}}{a_n} & \frac{x - b_n}{a_n}
	\end{pmatrix}
	,
	\quad n \geq 1.
	\end{aligned}
\end{equation}
Sometimes we write $(u_n(\eta, x) : n \in \NN_0)$ to indicate the dependence on the parameters.
In particular, the sequence of orthogonal polynomials $(p_n(x) : n \in \NN_0)$ is the generalized eigenvector associated to $x \in \CC$ and corresponding to $\eta = e_2$.

\subsection{Periodic Jacobi parameters}
\label{sec:2c}
By $(\alpha_n : n \in \ZZ)$ and $(\beta_n : n \in \ZZ)$ we denote
$N$-periodic sequences of real and positive numbers, respectively. For each $k \geq 0$, let us define polynomials
$(\mathfrak{p}^{[k]}_n : n \in \NN_0)$ by relations
\[
	\begin{gathered}
		\mathfrak{p}_0^{[k]}(x) = 1, \qquad \mathfrak{p}_1^{[k]}(x) = \frac{x-\beta_k}{\alpha_k}, \\
		\alpha_{n+k-1} \mathfrak{p}^{[k]}_{n-1}(x) + \beta_{n+k} \mathfrak{p}^{[k]}_n(x) 
		+ \alpha_{n+k} \mathfrak{p}^{[k]}_{n+1}(x)
		= x \mathfrak{p}^{[k]}_n(x), \qquad n \geq 1.
	\end{gathered}
\]
Let
\[
	\frakB_n(x) = 
	\begin{pmatrix}
		0 & 1 \\
		-\frac{\alpha_{n-1}}{\alpha_n} & \frac{x - \beta_n}{\alpha_n}
	\end{pmatrix},
	\qquad\text{and}\qquad
	\frakX_n(x) = \prod_{j = n}^{N+n-1} \mathfrak{B}_j(x), \qquad n \in \ZZ.
\]
By $\frakA$ we denote the Jacobi matrix corresponding to 
\begin{equation*}
	\begin{pmatrix}
		\beta_0 & \alpha_0 & 0   & 0      &\ldots \\
		\alpha_0 & \beta_1 & \alpha_1 & 0       & \ldots \\
		0   & \alpha_1 & \beta_2 & \alpha_2     & \ldots \\
		0   & 0   & \alpha_2 & \beta_3   &  \\
		\vdots & \vdots & \vdots  &  & \ddots
	\end{pmatrix}.
\end{equation*}

We start by showing the following identity.
\begin{proposition} 
	\label{prop:3}
	For all $x \in \RR$,
	\[
		\tr \frakX_0'(x) = 
		-\sum_{i=1}^N \frac{[\frakX_i(x)]_{2,1}}{\alpha_{i-1}}.
	\]
\end{proposition}
\begin{proof}
	By the Leibniz's rule
	\begin{align*}
		\frakX_0'(x) 
		&= \big(\frakB_{N-1} \ldots \frakB_1 \frakB_0\big)'(x) \\
		&= 
		\sum_{k=0}^{N-1} \bigg( \prod_{j=k+1}^{N-1} \frakB_j(x) \bigg) 
		\frakB_k'(x) \bigg( \prod_{j=0}^{k-1} \frakB_j(x) \bigg).
	\end{align*}
	Thus by linearity of the trace and its invariance on cyclic permutations
	\begin{align*}
		\tr \frakX_0'(x) &= 
		\sum_{k=0}^{N-1} 
		\tr \bigg\{ \bigg( \prod_{j=k+1}^{N-1} \frakB_j(x) \bigg) 
		\frakB_k'(x) \bigg( \prod_{j=0}^{k-1} \frakB_j(x) \bigg) \bigg\} \\
		&= 
		\sum_{k=0}^{N-1} 
		\tr \bigg\{ \frakB_k'(x) \prod_{j=k+1}^{N+k-1} \frakB_j(x) \bigg\}.
	\end{align*}
	In view of \cite[Proposition 3]{PeriodicIII},
	\begin{align*}
		\frakB_k'(x) \prod_{j=k+1}^{N+k-1} \frakB_j(x) &=
		\frac{1}{\alpha_k}
		\begin{pmatrix}
			0 & 0 \\
			0 & 1
		\end{pmatrix}
		\begin{pmatrix}
			-\frac{\alpha_{k}}{\alpha_{k+1}} \frakp^{[k+2]}_{N-3}(x) & \frakp^{[k+1]}_{N-2}(x) \\
			-\frac{\alpha_{k}}{\alpha_{k+1}} \frakp^{[k+2]}_{N-2}(x) & \frakp^{[k+1]}_{N-1}(x)
		\end{pmatrix} \\
		&=
		\begin{pmatrix}
			0 & 0 \\
			-\frac{1}{\alpha_{k+1}} \frakp_{N-2}^{[k+2]}(x) & \frac{1}{\alpha_k} \frakp_{N-1}^{[k+1]}(x)
		\end{pmatrix},
	\end{align*}
	thus
	\[
		\tr \frakX_0'(x) = 
		\sum_{k=0}^{N-1} \frac{1}{\alpha_k} \frakp_{N-1}^{[k+1]}(x).
	\]
	Since by \cite[Proposition 3]{PeriodicIII}
	\[
		\frac{1}{\alpha_{k-1}} [\frakX_k(x)]_{2,1} =
		-\frac{1}{\alpha_k} \frakp^{[k+1]}_{N-1}(x)
	\]
	we conclude the proof.
\end{proof}

\begin{proposition}
	\label{prop:4}
	If $|\tr \frakX_0(x)| \leq 2$, then
	\begin{equation}
		\label{eq:64}
		\sum_{i=1}^N \frac{|[\frakX_i(x)]_{2,1}|}{\alpha_{i-1}} =
		\bigg| \sum_{i=1}^N \frac{[\frakX_i(x)]_{2,1}}{\alpha_{i-1}} \bigg|.
	\end{equation}
\end{proposition}
\begin{proof}
	Let us first consider a matrix $A \in \SL(2, \RR)$. We have
	\[
		A_{1,1} A_{2,2} - A_{1,2} A_{2,1} = 1,
	\]
	thus
	\[
		A_{1,1}^2 - (\tr A) A_{1, 1} + 1 + A_{1,2}A_{2,1} = 0.
	\]
	Since $A_{1,1} \in \RR$, we have
	\[
		(\tr A)^2 - 4 (1 + A_{1,2} A_{2, 1}) \geq 0,
	\]
	that is
	\[
		- A_{1, 2} A_{2, 1} \geq 1 - \tfrac{1}{4} (\tr A)^2.
	\]
	Taking for $A = \frakX_i(x)$, by \cite[Proposition 3]{PeriodicIII}, we get
	\[
		\frac{\alpha_{i-1}}{\alpha_{i-2}} [\frakX_{i-1}(x)]_{2,1} [\frakX_i(x)]_{2,1} 
		=
		-[\frakX_i(x)]_{1,2} [\frakX_i(x)]_{2,1}
		\geq
		1 - \tfrac{1}{4} (\tr \frakX_i(x))^2 =  1 - \tfrac{1}{4} (\tr \frakX_0(x))^2,
	\]
	which easily leads to \eqref{eq:64} provided that $|\tr \frakX_0(x)| < 2$. If $|\tr \frakX_0(x)| = 2$, 
	we select a sequence $(x_n)$ tending to $x$ and such that $|\tr \frakX_0(x_n)| < 2$ for each $n$. By the continuity
	of $\frakX_i$,
	\begin{align*}
		\sum_{i=1}^N \frac{|[\frakX_i(x)]_{2,1}|}{\alpha_{i-1}}
		&=
		\lim_{n \to \infty}
		\sum_{i=1}^N \frac{|[\frakX_i(x_n)]_{2,1}|}{\alpha_{i-1}} \\
		&=
		\lim_{n \to \infty}
		\bigg| \sum_{i=1}^N \frac{[\frakX_i(x_n)]_{2,1}}{\alpha_{i-1}} \bigg| 
		=
		\bigg| \sum_{i=1}^N \frac{[\frakX_i(x)]_{2,1}}{\alpha_{i-1}} \bigg|
	\end{align*}
	which finishes the proof.
\end{proof}

\subsection{Periodic modulations}
\label{sec:2d}
In this article we are interested in $N$-periodically modulated Jacobi parameters, $N \in \NN$. We say that 
$(a_n : n \in \NN_0)$ and $(b_n : n \in \NN_0)$ are periodically modulated if there are two $N$-periodic
sequences $(\alpha_n : n \in \ZZ)$ and $(\beta_n  : n \in \ZZ)$ of positive and real numbers, respectively, such that
\begin{enumerate}[(a)]
	\item
	$\begin{aligned}[b]
	\lim_{n \to \infty} a_n = \infty
	\end{aligned},$
	\item
	$\begin{aligned}[b]
	\lim_{n \to \infty} \bigg| \frac{a_{n-1}}{a_n} - \frac{\alpha_{n-1}}{\alpha_n} \bigg| = 0
	\end{aligned},$
	\item
	$\begin{aligned}[b]
	\lim_{n \to \infty} \bigg| \frac{b_n}{a_n} - \frac{\beta_n}{\alpha_n} \bigg| = 0
	\end{aligned}.$
\end{enumerate}
We are mostly interested in periodically modulated parameters so that
\begin{equation} 
	\label{eq:40b}
	\bigg(\frac{\alpha_{n-1}}{\alpha_n} a_n - a_{n-1} : n \in \NN \bigg),
	\bigg(\frac{\beta_n}{\alpha_n} a_n - b_n : n \in \NN\bigg),
	\bigg(\frac{1}{\sqrt{a_n}} : n \in \NN\bigg) \in \calD_1^N.
\end{equation}
In view of \eqref{eq:40b}, there are two $N$-periodic sequences $(s_n : n \in \NN_0)$ and $(r_n : n \in \NN_0)$ 
such that
\begin{equation} 
	\label{eq:40c}
	\lim_{n \to \infty} \bigg| \frac{\alpha_{n-1}}{\alpha_n} a_{n} - a_{n-1} - s_n \bigg| = 0 
	\qquad \text{and} \qquad
	\lim_{n \to \infty} \bigg| \frac{\beta_n}{\alpha_n} a_{n} - b_{n} -r_n \bigg| = 0.
\end{equation}
By \cite[Proposition 4]{ChristoffelII}, for each $i \in \{0, 1, \ldots, N-1 \}$,
\begin{equation}
	\label{eq:25}
	\lim_{j \to \infty} \big( a_{(j+1)N+i} - a_{jN+i} \big) = 
	\alpha_i \sum_{k = 0}^{N-1} \frac{s_k}{\alpha_{k-1}}.
\end{equation}
We define the $N$-step transfer matrix by
\[
	X_n = B_{n+N-1} B_{n+N-2} \cdots B_{n+1} B_n,
\]
where $B_n$ is defined in \eqref{eq:108}.
Let us observe that for each $i \in \{0, 1, \ldots, N-1\}$,
\[
	\lim_{j \to \infty} B_{jN+i}(x) = \frakB_i(0)
\]
and
\[
	\lim_{j \to \infty} X_{jN+i}(x) = \frakX_i(0)
\]
locally uniformly with respect to $x \in \CC$. We always assume that the matrix $\frakX_0(0)$ is a non-trivial parabolic element of $\SL(2, \RR)$. Let $T_0$ be a matrix so that
\[
	\frakX_0(0) = \varepsilon 
	T_0 \begin{pmatrix}
		0 & 1 \\
		-1 & 2
	\end{pmatrix}
	T_0^{-1}
\]
where
\begin{equation} 
	\label{eq:61a}
	\varepsilon = \sign{\tr \frakX_0(0)}.
\end{equation}
Since
\[
	\frakX_i(0) = \frakB_{i-1}(0) \cdots \frakB_0(0) \frakX_0(0) \frakB_0^{-1} (0) \cdots \frakB_{i-1}^{-1}(0),
\]
by taking
\[
	T_i = \frakB_{i-1}(0) \cdots \frakB_0(0) T_0,
\]
we obtain
\[
	\frakX_i(0) = \varepsilon 
	T_i 
	\begin{pmatrix}
		0 & 1 \\
		-1 & 2
	\end{pmatrix}
	T_i^{-1}.
\]

\section{The shifted conjugation}
\label{sec:3}
In this section we introduce the shifted conjugation of $N$-step transfer matrix $X_n$ which produces matrices that
are uniformly diagonalizable. First, by the direct computations we can find that for any $T \in \GL(2, \RR)$,
\[
	\begin{pmatrix}
		1 & -1\\
		-1 & 1
	\end{pmatrix}
	T^{-1} 
	\begin{pmatrix}
		1 & 0 \\
		0 & 0
	\end{pmatrix}
	T
	\begin{pmatrix}
		1 & 1 \\
		1 & 1
	\end{pmatrix}
	=
	\frac{(T_{11}+T_{12})(T_{21}+T_{22})}{\det T}
	\begin{pmatrix}
		1 & 1\\
		-1 & -1
	\end{pmatrix}
\]
and
\[
	\begin{pmatrix}
		1 & -1\\
		-1 & 1
	\end{pmatrix}
	T^{-1} 
	\begin{pmatrix}
		0 & 1 \\
		0 & 0
	\end{pmatrix}
	T
	\begin{pmatrix}
		1 & 1 \\
		1 & 1
	\end{pmatrix}
	=
	\frac{(T_{21}+T_{22})^2}{\det T}
	\begin{pmatrix}
		1 & 1\\
		-1 & -1
	\end{pmatrix}.
\]
Let the sequences $(s_{i'}), (r_{i'})$ be defined in \eqref{eq:40c} and the number $\varepsilon$ 
defined in \eqref{eq:61a}. Hence, by \eqref{eq:12} and \eqref{eq:11} we obtain
\begin{align*}
	&
	\sum_{i' = i}^{N+i-1}
	\frac{1}{\alpha_{i'-1}}
	\begin{pmatrix}
		1 & -1 \\
		-1 & 1
	\end{pmatrix}
	T_{i'}^{-1} 
	\begin{pmatrix}
		s_{i'} & x + r_{i'} \\
		0 & 0
	\end{pmatrix}
	T_{i'}
	\begin{pmatrix}
		1 & 1 \\
		1 & 1
	\end{pmatrix} \\
	&\qquad\qquad=
	\begin{pmatrix}
		 1 & 1\\
		 -1 & -1
	\end{pmatrix}
	\sum_{i' = 0}^{N-1} 
	\bigg(
	\frac{s_{i'}}{\alpha_{i'-1}} \Big(1 - \varepsilon [\frakX_{i'}(0)]_{1,1} \Big)
	-
	\frac{x+r_{i'}}{\alpha_{i'-1}} \varepsilon [\frakX_{i'}(0)]_{2,1}
	\bigg) \\
	&\qquad\qquad=
	\begin{pmatrix}
		 1 & 1\\
		 -1 & -1
	\end{pmatrix}
	\tau(x)
\end{align*}
where we have set
\begin{equation} 
	\label{eq:61b}
	\tau(x) = 
	\sum_{i' = 0}^{N-1} 
	\bigg(
	\frac{s_{i'}}{\alpha_{i'-1}} \Big( 1 - \varepsilon [\frakX_{i'}(0)]_{1,1} \Big)
	-
	\frac{x+r_{i'}}{\alpha_{i'-1}} \varepsilon [\frakX_{i'}(0)]_{2,1} 
	\bigg).
\end{equation}
Let us observe that by Proposition \ref{prop:3},
\[
	\tau(x) = 
	\varepsilon \tr \frakX_0'(0) \cdot x +
	\sum_{i' = 0}^{N-1} 
	\bigg(
	\frac{s_{i'}}{\alpha_{i'-1}} \Big( 1 - \varepsilon [\frakX_{i'}(0)]_{1,1} \Big)
	-
	\frac{r_{i'}}{\alpha_{i'-1}} \varepsilon [\frakX_{i'}(0)]_{2,1} 
	\bigg).
\]
Since $\frakX_0(0)$ is a non-trivial parabolic element of $\SL(2, \RR)$, $\tr \frakX_0'(0) \neq 0$. To see this,
let us suppose, contrary to our claim, that $\tr \frakX_0'(0) = 0$. Then by Propositions \ref{prop:3} and 
\ref{prop:4}, for each $i \in \{0, 1, \ldots, N-1\}$,
\[
	[\frakX_i(0)]_{2, 1} = 0.
\]
Hence, by \cite[Proposition 3]{PeriodicIII},
\[
	[\frakX_i(0)]_{1, 2} = 0,
\]
which is impossible. Knowing that $\tr \frakX_0'(0) \neq 0$, we conclude that
\begin{equation} 
	\label{eq:61d}
	x_0 =
	\frac{1}{\varepsilon \tr \frakX_0'(0)} 
	\sum_{i' = 0}^{N-1} 
	\bigg(
	\frac{r_{i'}}{\alpha_{i'-1}} \varepsilon [\frakX_{i'}(0)]_{2,1} 
	-
	\frac{s_{i'}}{\alpha_{i'-1}} \Big( 1 - \varepsilon [\frakX_{i'}(0)]_{1,1} \Big)
	\bigg).
\end{equation}
is the only solution to $\tau(x) = 0$.

Now, let us fix $i \in \{0, 1, \ldots, N-1\}$ and set
\begin{equation} 
	\label{eq:46}
	Z_j = 
	T_{i}
	\begin{pmatrix}
		1 & 1\\
		\ue^{\vartheta_j} & \ue^{-\vartheta_j}
	\end{pmatrix}
\end{equation}
where
\begin{equation} 
	\label{eq:68}
	\vartheta_j(x) =
	\sqrt{\frac{\alpha_{i-1} |\tau(x)|}{a_{(j+1)N+i-1}}}.
\end{equation}
Then
\begin{equation}
	\label{eq:16}
	\frac{\alpha_{i-1}}{a_{(j+1)N+i-1}}
	\sum_{i' = i}^{N+i-1}
	\frac{1}{\alpha_{i'-1}}
	\begin{pmatrix}
		1 & -1 \\
		-1 & 1
	\end{pmatrix}
	T_{i'}^{-1} 
	\begin{pmatrix}
		s_{i'} & x + r_{i'} \\
		0 & 0
	\end{pmatrix}
	T_{i'}
	\begin{pmatrix}
		1 & 1 \\
		1 & 1
	\end{pmatrix}
	=
	\sigma \vartheta_j^2 
	\begin{pmatrix}
		1 & 1 \\
		-1 & -1
	\end{pmatrix}
\end{equation}
where 
\begin{equation}
	\label{eq:61c}
	\sigma(x) = \sign{\tau(x)}.
\end{equation}

Before we proceed let us recall that the set $\calD_1$ is an algebra over $\RR$. Moreover, we have the following
lemma.
\begin{lemma}
	\label{lem:4}
	If $(a_n : n \in \NN_0)$ is a sequence of positive numbers such that
	\begin{enumerate}[(a)]
		\item \label{eq:2}
		$\begin{aligned}[b] 
		\lim_{n \to \infty} a_n = \infty
		\end{aligned},$
		\item \label{eq:5}
		$\begin{aligned}[b] 
		\big(a_{n+1} - a_n : n \in \NN_0 \big) \in \calD_1
		\end{aligned},$
		\item \label{eq:3}
		$\begin{aligned}[b]
		\bigg(\frac{1}{\sqrt{a_n}} : n \in \NN_0\bigg) \in \calD_1
		\end{aligned},$
	\end{enumerate}
	then 
	\begin{equation*}
		\bigg(\sqrt{\frac{a_{n+1}}{a_n}} : n \in \NN\bigg),
		\big( \sqrt{a_{n+1}} - \sqrt{a_n} : n \in \NN \big), \\
		\bigg(a_n\bigg(\frac{1}{\sqrt{a_n}} - \frac{1}{\sqrt{a_{n+1}}}\bigg) : n \in \NN \bigg)
		\in
		\calD_1.
	\end{equation*}
	Moreover,
	\[
		\lim_{n \to \infty} \sqrt{\frac{a_{n+1}}{a_n}} = 1, \qquad
		\lim_{n \to \infty} \big( \sqrt{a_{n+1}} - \sqrt{a_n} \big) = 0, \qquad
		\lim_{n \to \infty} a_n\bigg(\frac{1}{\sqrt{a_n}} - \frac{1}{\sqrt{a_{n+1}}}\bigg) = 0.
	\]
\end{lemma}
\begin{proof}
	We notice that
	\[
		\frac{1}{a_n} = \frac{1}{\sqrt{a_n}} \cdot \frac{1}{\sqrt{a_n}},
	\]
	thus by \eqref{eq:3}
	\begin{equation} 
		\label{eq:26}
		\bigg(\frac{1}{a_n} : n \in \NN\bigg) \in \calD_1.
	\end{equation}
	Since
	\[
		\frac{1}{a_n} \big(a_{n+1} - a_n\big) = \frac{a_{n+1}}{a_n} - 1,
	\]
	by \eqref{eq:26}, \eqref{eq:5} and \eqref{eq:2}, we conclude that
	\begin{equation}
		\label{eq:29}
		\bigg(\sqrt{\frac{a_{n+1}}{a_n}} : n \in \NN\bigg) \in \calD_1.
	\end{equation}
	Next, observe that
	\[
		\sqrt{a_{n+1}} - \sqrt{a_n} = 
		\frac{1}{\sqrt{a_n}} \frac{a_{n+1} - a_n}{\sqrt{\tfrac{a_{n+1}}{a_n}} + 1},
	\]
	hence, by \eqref{eq:3}, \eqref{eq:5} and \eqref{eq:29} it follows that
	\begin{equation} 
		\label{eq:34}
		\big( \sqrt{a_{n+1}} - \sqrt{a_n} : n \in \NN_0 \big) \in \calD_1.
	\end{equation}
	Finally, we have
	\begin{align*}
		a_n \bigg(\frac{1}{\sqrt{a_n}} - \frac{1}{\sqrt{a_{n+1}}}\bigg)
		&= \frac{\sqrt{a_{n+1}} - \sqrt{a_n}}{\sqrt{\tfrac{a_{n+1}}{a_n}}},
	\end{align*}
	which by \eqref{eq:34} and \eqref{eq:29} belongs to $\calD_1$.
\end{proof}

\begin{theorem}
	\label{thm:2}
	Let $N$ be a positive integer and $i \in \{0, 1, \ldots N-1\}$. Suppose that $(a_n : n \in \NN_0)$ and
	$(b_n : n \in \NN_0)$ are $N$-periodically modulated Jacobi parameters such that $\frakX_0(0)$ is 
	a non-trivial parabolic element. If
	\[
		\bigg(\frac{\alpha_{n-1}}{\alpha_n} a_n - a_{n-1} : n \in \NN \bigg),
		\bigg(\frac{\beta_n}{\alpha_n} a_n - b_n : n \in \NN\bigg),
		\bigg(\frac{1}{\sqrt{a_n}} : n \in \NN\bigg) \in \calD_1^N,
	\]
	then for any compact interval $K \subset \RR \setminus \{x_0\}$,
	\begin{equation} 
		\label{eq:15}
		Z_j^{-1} Z_{j+1} = \Id + \vartheta_j Q_j,
	\end{equation}
	where $x_0$ is defined in \eqref{eq:61d}, and $(Q_j)$ is a sequence from 
	$\calD_1\big(K, \Mat(2, \RR)\big)$ convergent uniformly on $K$ to zero.
\end{theorem}
\begin{proof}
	In the proof we denote by $(\delta_j)$ a generic sequence from $\calD_1$ tending to zero which may
	change from line to line.
	
	By a straightforward computation we obtain
	\begin{align*}
		Z_j^{-1} Z_{j+1} 
		&=
		\frac{1}{\det Z_j} 
		\begin{pmatrix}
			\ue^{-\vartheta_j} & -1\\
			-\ue^{\vartheta_j} & 1
		\end{pmatrix}
		\begin{pmatrix}
			1 & 1\\
			\ue^{\vartheta_{j+1}} & \ue^{-\vartheta_{j+1}}
		\end{pmatrix} \\
		&=
		\frac{1}{\ue^{-\vartheta_{j}} - \ue^{\vartheta_{j}}}
		\begin{pmatrix}
			f_j & g_j \\
			\tilde{g}_j & \tilde{f}_j
		\end{pmatrix}
	\end{align*}
	where
	\begin{alignat*}{3}
		&f_j = \ue^{-\vartheta_j} - \ue^{\vartheta_{j+1}}, &\qquad
		&g_j = \ue^{-\vartheta_j} - \ue^{-\vartheta_{j+1}} \\
		&\tilde{g}_j = - \ue^{\vartheta_j} +  \ue^{\vartheta_{j+1}} &
		&\tilde{f}_j = - \ue^{\vartheta_j} +  \ue^{-\vartheta_{j+1}}.
	\end{alignat*}
	Since
	\[
		(a_{(j+1)N+i} - a_{jN+i} : j \in \NN_0), \bigg(\frac{1}{\sqrt{a_{jN+i}}} : j \in \NN_0\bigg)
		\in \calD_1, 
	\]
	by Lemma \ref{lem:4},
	\[
		\bigg(\frac{a_{jN}}{a_{jN+i}} : j \in \NN_0 \bigg) \in \calD_1,
	\]
	and
	\[
		\vartheta_{j+1} = \vartheta_j + \frac{1}{a_{jN}} \delta_j.
	\]
	Moreover,
	\[
		\ue^{\vartheta_{j+1}} = 1 + \vartheta_{j+1} + \frac{1}{2} \vartheta_{j+1}^2 + \frac{1}{a_{jN}} \delta_j,
	\]
	and
	\[
		\ue^{-\vartheta_j} = 1 - \vartheta_j + \frac{1}{2} \vartheta_j^2 + \frac{1}{a_{jN}} \delta_j,
	\]
	Hence,
	\begin{align*}
		f_j 
		&=1 - \vartheta_j + \frac{1}{2} \vartheta_j^2 
		-\bigg(1 + \vartheta_{j+1} + \frac{1}{2} \vartheta_{j+1}^2\bigg)
		+ \frac{1}{a_{jN}} \delta_j\\
		&=
		-2 \vartheta_j + \frac{1}{a_{jN}} \delta_j.
	\end{align*}
	Since $\frac{x}{\sinh (x)}$ is an even $\calC^2(\RR)$ function, we have
	\[
		\frac{\vartheta_{j}}{\sinh( \vartheta_{j} )} = 1 + \frac{1}{\sqrt{a_{jN}}} \delta_j.
	\]
	Therefore,
	\begin{align*}
		\frac{1}{\ue^{-\vartheta_{j}} - \ue^{\vartheta_{j}}} f_j
		&=
		\frac{f_j}{-2\vartheta_{j}} \frac{\vartheta_j}{\sinh(\vartheta_{j})} \\
		&=
		\bigg(1 + \frac{1}{\sqrt{a_{jN}}} \delta_j \bigg)
		\bigg(1 + \frac{1}{\sqrt{a_{jN}}}\delta_j \bigg) \\
		&=
		1 + \frac{1}{\sqrt{a_{jN}}} \delta_j.
	\end{align*}
	Analogously, we treat $g_j$. Namely, we write
	\begin{align*}
		g_j &= 
		1 - \vartheta_j + \frac{1}{2} \vartheta_j^2 
		-
		\bigg(1 - \vartheta_{j+1} + \frac{1}{2}\vartheta_{j+1}^2\bigg)
		+
		\frac{1}{a_{jN}} \delta_j \\
		&=
		\frac{1}{a_{jN}} \delta_j.
	\end{align*}
	Hence,
	\begin{align*}
		\frac{1}{\ue^{-\vartheta_{j}} - \ue^{\vartheta_{j}}} g_j
		&=
		\frac{1}{\sqrt{a_{jN}}} \delta_j. 
	\end{align*}
	Similarly, we can find that
	\begin{alignat*}{2}
		\frac{1}{\ue^{-\vartheta_{j}} - \ue^{\vartheta_{j}}} \tilde{f}_j
		&=
		1 + \frac{1}{\sqrt{a_{jN}}} \delta_j,\\
		\frac{1}{\ue^{-\vartheta_{j}} - \ue^{\vartheta_{j}}} \tilde{g}_j
		&=
		\frac{1}{\sqrt{a_{jN}}} \delta_j.
	\end{alignat*}
	Hence,
	\[
		Z_{j}^{-1} Z_{j+1} = \Id + \vartheta_j Q_j
	\]
	where $(Q_j)$ is a sequence from $\calD_1\big(K, \Mat(2, \RR)\big)$ for any compact interval
	$K \subset \RR \setminus \{ x_0 \}$ convergent to the zero matrix proving the formula \eqref{eq:15}. 
\end{proof}

\begin{theorem} 
	\label{thm:1}
	Let $N$ be a positive integer and $i \in \{0, 1, \ldots N-1\}$. Suppose that $(a_n : n \in \NN_0)$ and
	$(b_n : n \in \NN_0)$ are $N$-periodically modulated Jacobi parameters such that $\frakX_0(0)$ is 
	a non-trivial parabolic element. If
	\[
		\bigg(\frac{\alpha_{n-1}}{\alpha_n} a_n - a_{n-1} : n \in \NN \bigg),
		\bigg(\frac{\beta_n}{\alpha_n} a_n - b_n : n \in \NN\bigg),
		\bigg(\frac{1}{\sqrt{a_n}} : n \in \NN\bigg) \in \calD_1^N,
	\]
	then for any compact interval $K \subset \RR \setminus \{ x_0 \}$,
	\[
		Z_{j+1}^{-1} X_{jN+i} Z_{j} = \varepsilon \big( \Id + \vartheta_j R_j \big)
	\]
	where $\varepsilon$ and $x_0$ are defined in \eqref{eq:61a} and \eqref{eq:61d}, respectively and 
	$(R_j)$ is a sequence from $\calD_1\big(K, \Mat(2, \RR)\big)$ convergent uniformly on $K$ to
	\[
		\calR_i = 
		\frac{1}{2}
		\begin{pmatrix}
			1 + \sigma & -1 + \sigma  \\
			1 - \sigma & -1 - \sigma 
		\end{pmatrix}
	\]
	where $\sigma$ is defined in \eqref{eq:61c}. In particular, $\discr \calR_i = 4 \sigma$.
\end{theorem}
\begin{proof}
	In the following argument, we denote by $(\delta_j)$ and $(\calE_j)$ generic sequences tending to zero
	from $\calD_1$ and $\calD_1\big(K, \Mat(2, \RR)\big)$, respectively, which may change from line to line.

	Since
	\begin{align*}
		\frac{a_{n-1}}{a_n} &= \frac{\alpha_{n-1}}{\alpha_n} -
		\frac{1}{a_n} \bigg( \frac{\alpha_{n-1}}{\alpha_n} a_n - a_{n-1}\bigg), \\
	\intertext{and}
		\frac{b_n}{a_n} &= 
		\frac{\beta_n}{\alpha_n} - \frac{1}{a_n} \bigg( \frac{\beta_n}{\alpha_n} a_n - b_n\bigg),
	\end{align*}
	by \eqref{eq:40b} and \eqref{eq:40c}, for each $i' \in \{0, 1, \ldots, N-1\}$, we obtain
	\begin{align*}
		\frac{a_{jN+i'-1}}{a_{jN+i'}} &= \frac{\alpha_{i'-1}}{\alpha_{i'}} 
		- \frac{s_{i'}}{a_{jN+i'}} + \frac{1}{a_{jN+i'}} \delta_j,\\
	\intertext{and}
		\frac{b_{jN+i'}}{a_{jN+i'}} &= \frac{\beta_{i'}}{\alpha_{i'}} - \frac{r_{i'}}{a_{jN+i'}} 
		+ \frac{1}{a_{jN+i'}} \delta_j.
	\end{align*}
	We also have
	\begin{align*}
		\frac{1}{a_{jN+i'}}
		&= \frac{1}{a_{jN+i'-1}} \frac{a_{jN+i'-1}}{a_{jN+i'}}  \\
		&= \frac{1}{a_{jN+i'-1}} \bigg(\frac{\alpha_{i'-1}}{\alpha_{i'}} - \frac{s_{i'}}{a_{jN+i'}} 
		+ \frac{1}{a_{jN+i'}} \delta_j \bigg)\\
		&= \frac{1}{a_{jN+i'-1}} \frac{\alpha_{i'-1}}{\alpha_{i'}} + \frac{1}{a_{jN+i'-1}} \delta_j,
	\end{align*}
	thus
	\[
		\frac{1}{a_{jN+i'}} = \frac{1}{a_{jN}} \frac{\alpha_0}{\alpha_{i'}} + \frac{1}{a_{jN}} \delta_j.
	\]
	Therefore,
	\begin{align*}
		\sum_{i' = 0}^{N-1} \frac{1}{a_{jN+i'}} \frac{\alpha_{i'}}{\alpha_{i'-1}} s_{i'}
		&=
		\sum_{i' = 0}^{N-1} \frac{1}{a_{jN}} \frac{\alpha_0}{\alpha_{i'-1}} s_{i'} + \frac{1}{a_{jN}} \delta_j \\
		&=
		\alpha_0 \sum_{i' = 0}^{N-1} \frac{s_{i'}}{\alpha_{i'-1}} +
		\frac{1}{a_{jN}} \delta_j.
	\end{align*}
	Next, we write
	\begin{align*}
		B_{jN+i'}
		&=
		\begin{pmatrix}
			0 & 1 \\
			-\frac{\alpha_{i'-1}}{\alpha_{i'}} + \frac{s_{i'}}{a_{jN+i'}} + \frac{1}{a_{jN+i'}} \delta_j
			& \frac{x+r_{i'}}{a_{jN+i'}} - \frac{\beta_{i'}}{\alpha_{i'}} + \frac{1}{a_{jN+i'}} \delta_j
		\end{pmatrix} \\
		&=
		\begin{pmatrix}
			0 & 1 \\
			-\frac{\alpha_{i'-1}}{\alpha_{i'}} & - \frac{\beta_{i'}}{\alpha_{i'}}
		\end{pmatrix}
		+
		\frac{1}{a_{jN+i'}}
		\begin{pmatrix}
			0 & 0 \\
			s_{i'} & x+r_{i'}
		\end{pmatrix}
		+
		\frac{1}{a_{jN+i'}} \calE_j \\
		&=
		\begin{pmatrix}
			0 & 1 \\
			-\frac{\alpha_{i'-1}}{\alpha_{i'}} & - \frac{\beta_{i'}}{\alpha_{i'}}
		\end{pmatrix}
		\left\{
		\Id
		+
		\frac{1}{a_{jN+i'}}
		\frac{\alpha_j}{\alpha_{i'-1}}
		\begin{pmatrix}
			-\frac{\beta_{i'}}{\alpha_{i'}} & -1 \\
			\frac{\alpha_{i'-1}}{\alpha_{i'}} & 0
		\end{pmatrix}
		\begin{pmatrix}
			0 & 0 \\
			s_{i'} & x + r_{i'}
		\end{pmatrix}
		+
		\frac{1}{a_{jN+i'}} \calE_j \right\} \\
		&=
		\begin{pmatrix}
			0 & 1 \\
			-\frac{\alpha_{i'-1}}{\alpha_{i'}} & - \frac{\beta_{i'}}{\alpha_{i'}}
		\end{pmatrix}
		\left\{
		\Id - 
		\frac{1}{a_{jN+i'}} \frac{\alpha_{i'}}{\alpha_{i'-1}}
		\begin{pmatrix}
			s_{i'} & x+r_{i'} \\
			0 & 0
		\end{pmatrix}
		+
		\frac{1}{a_{jN}} \calE_j
		\right\} \\
		&=
		\begin{pmatrix}
			0 & 1 \\
			-\frac{\alpha_{i'-1}}{\alpha_{i'}} & - \frac{\beta_{i'}}{\alpha_{i'}}
		\end{pmatrix}
		\left\{
		\Id - 
		\frac{1}{a_{(j+1)N+i-1}} \frac{\alpha_{i-1}}{\alpha_{i'-1}}
		\begin{pmatrix}
			s_{i'} & x+r_{i'} \\
			0 & 0
		\end{pmatrix}
		+
		\frac{1}{a_{jN}} \calE_j
		\right\}
	\end{align*}
	where we have used that
	\[
		\frac{\alpha_{i'}}{a_{jN+{i'}}} = \frac{\alpha_{i-1}}{a_{(j+1)N+i-1}} + \frac{1}{a_{jN}} \delta_j.
	\]
	Next, we compute
	\begin{align*}
		X_{jN+i}
		&=
		B_{jN+i+N-1} \cdots B_{jN+i+1} B_{jN+i} \\
		&=
		\frakX_{i}(0)
		\Bigg\{
		\Id
		-
		\frac{\alpha_{i-1}}{a_{(j+1)N+i-1}}
		\sum_{i' = i}^{N+i-1} 
		\frac{1}{\alpha_{i'-1}}
		\Big(
		\frakB_{i'-1}(0) \cdots  \frakB_i(0)\Big)^{-1}
		\begin{pmatrix}
			s_{i'} & x + r_{i'} \\
			0 & 0
		\end{pmatrix}
		\Big(\frakB_{i'-1}(0) \cdots \frakB_i(0)\Big) \\
		&\phantom{\frakX_i(0)\Bigg\{\Id}+
		\frac{1}{a_{jN}} \calE_j
		\Bigg\}.
	\end{align*}
	Thus,
	\begin{align*}
		&Z_{j+1}^{-1} X_{jN+i} Z_j \\
		&\qquad=
		Z_{j+1}^{-1} \mathfrak{X}_{i}(0) Z_j
		\Bigg\{
		\Id
		-
		\frac{\alpha_{i-1}}{a_{(j+1)N+i-1}}
		\sum_{i' = i}^{N+i-1} \frac{1}{\alpha_{i'-1}}
		\begin{pmatrix}
			1 & 1 \\
			\ue^{\vartheta_j} & \ue^{-\vartheta_j}
		\end{pmatrix}^{-1}
		T_{i'}^{-1}
		\begin{pmatrix}
			s_{i'}& x + r_{i'} \\
			0 & 0
		\end{pmatrix}
		T_{i'}
		\begin{pmatrix}
			1 & 1 \\
			\ue^{\vartheta_j} & \ue^{-\vartheta_j}
		\end{pmatrix} \\
		&\qquad\phantom{=Z_{j+1}^{-1} \mathfrak{X}_{i}(0) Z_{j} \Bigg\{}+
		\frac{1}{\sqrt{a_{jN}}} \calE_j
		\Bigg\}.
	\end{align*}
	To find the asymptotics of the first factor, we write
	\[
		Z_{j+1}^{-1} \mathfrak{X}_{i}(0) Z_j
		=
		\frac{\varepsilon}{\ue^{-\vartheta_{j+1}} - \ue^{\vartheta_{j+1}}}
		\begin{pmatrix}
			f_j & g_j \\
			\tilde{g}_j & \tilde{f}_j
		\end{pmatrix}
	\]
	where
	\begin{alignat*}{3}
		f_j &= \ue^{{\vartheta_{j}} - \vartheta_{j+1}} + 1 - 2\ue^{\vartheta_{j}}, &\qquad
		g_j &= \ue^{{-\vartheta_{j}} - \vartheta_{j+1}} + 1 - 2\ue^{-\vartheta_{j}} \\
		\tilde{g}_j &= -\ue^{\vartheta_{j}+\vartheta_{j+1}} -1 +2\ue^{\vartheta_{j}}, &\qquad
		\tilde{f}_j &= -\ue^{-\vartheta_{j}+\vartheta_{j+1}} -1 + 2 \ue^{-\vartheta_{j}}.
	\end{alignat*}
	Since
	\[
		\ue^{{\vartheta_{j}} - \vartheta_{j+1}} = 1 + \frac{1}{a_{jN}} \delta_j,
		\qquad\text{and}\qquad
		\ue^{\vartheta_{j}} = 1 + \vartheta_j + \frac{1}{2} \vartheta_j^2 + \frac{1}{a_{jN}} \delta_j,
	\]
	we get
	\begin{align*}
		f_n 
		&= 1 + 1 - 2 \bigg(1 + \vartheta_{j} + \frac{1}{2}\vartheta_{j}^2 \bigg)+ \frac{1}{a_{jN}} \delta_j \\
		&= -2 \vartheta_{j} - \vartheta_{j}^2 + \frac{1}{a_{jN}} \delta_j.
	\end{align*}
	Thus
	\begin{align}
		\nonumber
		\frac{1}{\ue^{-\vartheta_{j+1}} - \ue^{\vartheta_{j+1}}} f_j
		&=
		\frac{f_j}{-2\vartheta_{j}} \frac{\vartheta_{j}}{\sinh \vartheta_{j}} \\
		\label{eq:14}
		&=
		1 + \frac{1}{2} \vartheta_{j} + \frac{1}{\sqrt{a_{jN}}} \delta_j.
	\end{align}
	Analogously, we can find that
	\[
		\tilde{f}_j = -2 \vartheta_{j} + \vartheta_{j}^2 + \frac{1}{a_{jN}} \delta_j,
	\]
	and
	\begin{equation}
		\label{eq:18}
		\frac{1}{\ue^{-\vartheta_{j+1}} - \ue^{\vartheta_{j+1}}} \tilde{f}_{j} 
		= 1 -\frac{1}{2} \vartheta_{j}  + \frac{1}{\sqrt{a_{jN}}} \delta_j.
	\end{equation}
	Next, we write
	\begin{align*}
		g_j &= 1 - 2\vartheta_{j} + 2 \vartheta_{j}^2 + 1 - 2 \bigg(1 -\vartheta_{j} + \frac{1}{2} \vartheta_{j}^2\bigg) 
		+ \frac{1}{a_{jN}} \delta_j \\
		&= \vartheta_{j}^2 + \frac{1}{a_{jN}} \delta_j,
	\end{align*}
	thus
	\begin{equation}
		\label{eq:19}
		\frac{1}{\ue^{-\vartheta_{j+1}} - \ue^{\vartheta_{j+1}}} g_j 
		= -\frac{1}{2} \vartheta_{j} + \frac{1}{\sqrt{a_{jN}}} \delta_j.
	\end{equation}
	Similarly, we get
	\[
		\tilde{g}_j = -\vartheta_{j}^2 + \frac{1}{a_{jN}} \delta_j,
	\]
	and so
	\begin{equation}
		\label{eq:27}
		\frac{1}{\ue^{-\vartheta_{j+1}} - \ue^{\vartheta_{j+1}}} \tilde{g}_j = 
		 \frac{1}{2} \vartheta_{j} + \frac{1}{\sqrt{a_{jN}}} \delta_j.
	\end{equation}
	Consequently, by \eqref{eq:14}--\eqref{eq:27} we obtain
	\[
		Z_{j+1}^{-1} \mathfrak{X}_{i}(0) Z_{j}
		=
		\varepsilon 
		\left\{
		\Id + \frac{1}{2} \vartheta_{j} 
		\begin{pmatrix}
			1 & -1 \\
			1 & -1
		\end{pmatrix} + \frac{1}{\sqrt{a_{jN}}} \calE_j
		\right\}.
	\]
	Since
	\[
		\ue^{\vartheta_{j}} = 1 + \delta_j, \qquad
		\frac{\vartheta_{j}}{\sinh \vartheta_{j}} = 1 + \delta_j,
	\]
	for each $i' \in \{0, 1, \ldots, N-1\}$, we have
	\begin{align*}
		&
		\begin{pmatrix}
			1 & 1 \\
			\ue^{\vartheta_j} & \ue^{-\vartheta_j}
		\end{pmatrix}^{-1}
		T_{i'}^{-1}
		\begin{pmatrix}
			s_{i'} & x + r_{i'} \\
			0 & 0
		\end{pmatrix}
		T_{i'}
		\begin{pmatrix}
			1 & 1 \\
			\ue^{\vartheta_j} & \ue^{-\vartheta_j}
		\end{pmatrix}  \\
		&\qquad=
		-
		\frac{1}{2\vartheta_j}
		\begin{pmatrix}
			1 & -1\\
			-1 & 1
		\end{pmatrix}
		T_{i'}^{-1}
		\begin{pmatrix}
			s_{i'} & x + r_{i'} \\
			0 & 0
		\end{pmatrix}
		T_{i'}
		\begin{pmatrix}
			1 & 1 \\
			1 & 1
		\end{pmatrix} 
		+
		\sqrt{a_{jN}} \calE_j.
	\end{align*}
	Hence, by \eqref{eq:16}
	\begin{align*}
		&
		\frac{\alpha_{i-1}}{a_{(j+1)N+i-1}}
		\sum_{i' = i}^{N+i-1}
		\frac{1}{\alpha_{i'-1}}
		\begin{pmatrix}
			1 & 1 \\
			\ue^{\vartheta_j} & \ue^{-\vartheta_j}
		\end{pmatrix}^{-1}
		T_{i'}^{-1}
		\begin{pmatrix}
			s_{i'} & x + r_{i'} \\
			0 & 0
		\end{pmatrix}
		T_{i'}
		\begin{pmatrix}
			1 & 1 \\
			\ue^{\vartheta_j} & \ue^{-\vartheta_j}
		\end{pmatrix} \\
		&\qquad\qquad=
		\frac{\sigma}{2} \vartheta_j 
		\begin{pmatrix}
			1 & 1 \\
			-1 & -1
		\end{pmatrix}
		+
		\frac{1}{\sqrt{a_{jN}}} \calE_j.
	\end{align*}
	Finally, we get
	\begin{align*}
		Z_{j+1}^{-1} X_{jN+i} Z_j
		&=
		\varepsilon
		\left\{
		\Id
		+
		\frac{1}{2}
		\vartheta_{j}
		\begin{pmatrix}
			1 + \sigma  & -1 + \sigma \\
			1 - \sigma  & -1 - \sigma 
		\end{pmatrix}
		+
		\frac{1}{\sqrt{a_{jN}}} \calE_j
		\right\}
	\end{align*}
	which finishes the proof.
\end{proof}

\begin{corollary}
	\label{cor:2}
	Let the hypotheses of Theorem \ref{thm:1} be satisfied. Then
	\begin{align*} 
		\lim_{j \to \infty} a_{(j+1)N+i-1} \discr \big( X_{jN+i}\big) &= 
		\alpha_{i-1} |\tau| \discr(\calR_i) \\
		&=
		4 \alpha_{i-1} \tau
	\end{align*}
	locally uniformly on $\RR \setminus \{ x_0 \}$, where $\tau$ and $x_0$ are defined 
	in \eqref{eq:61b} and \eqref{eq:61d}, respectively.
\end{corollary}
\begin{proof}
	Since 
	\[
		Z_j^{-1} X_{jN+i} Z_j = \big( Z_j^{-1} Z_{j+1} \big) \big( Z_{j+1}^{-1} X_{jN+i} Z_j \big),
	\]
	by Theorems \ref{thm:2} and \ref{thm:1}, we obtain
	\[
		\varepsilon Z_j^{-1} X_{jN+i} Z_j 
		= 
		\big( \Id + \vartheta_j Q_j \big) \big(\Id + \vartheta_j R_j \big) 
		= 
		\Id + \vartheta_j R_j + \vartheta_j Q_j + \vartheta_j^2 Q_j R_j.
	\]
	Thus
	\[
		\discr \big( \vartheta_j^{-1} X_{jN+i} \big) = \discr \big( R_j + Q_j + \vartheta_j Q_j R_j \big),
	\]
	and consequently,
	\[
		\lim_{j \to \infty} \discr ( \vartheta_j^{-1} X_{jN+i}) 
		= \discr (\calR_i) 
		= 4 \sign{\tau}.
	\]
	Since
	\begin{align*}
		\discr \big( \sqrt{a_{(j+1)N+i-1}} X_{jN+i} \big) 
		&=
		\discr \big( \sqrt{\alpha_{i-1} |\tau|} \vartheta_j^{-1} X_{jN+i} \big) \\
		&=
		\alpha_{i-1} |\tau| \discr \big( \vartheta_j^{-1} X_{jN+i} \big)
	\end{align*}
	the conclusion follows.
\end{proof}

\section{Essential spectrum}
\label{sec:4}
In this section we start the analysis of the measure $\mu$. To do so, we shall use the Jacobi 
matrix associated to the sequence $(a_n : n \in \NN_0)$ and $(b_n : n \in \NN_0)$, see Section~\ref{sec:2c} for details. 
From \eqref{eq:25} we can easily deduce that the Carleman's condition \eqref{eq:37} is satisfied and consequently 
the operator $A$ is self-adjoint. Moreover, the measure $\mu$ is the spectral measure of $A$. We set
\begin{equation}
	\label{eq:56}
	\Lambda_- = \tau^{-1} \big( (-\infty, 0)  \big), \quad\text{and}\quad
	\Lambda_+ = \tau^{-1} \big( (0, \infty) \big)
\end{equation}
where $\tau$ is given by \eqref{eq:61b}. In Theorem \ref{thm:8} we prove that $\sigmaEss(A)$ is contained in
$\RR \setminus \Lambda_+$ which implies that the measure $\mu$ restricted to $\Lambda_+$ is purely atomic and all accumulation
points of its support are on the boundary of $\Lambda_+$.

\begin{theorem}
	\label{thm:8}
	Let $N$ be a positive integer. Let $A$ be a Jacobi matrix with $N$-periodically
	modulated entries so that $\frakX_0(0)$ is a non-trivial parabolic element. If
	\[
		\bigg(\frac{\alpha_{n-1}}{\alpha_n} a_n - a_{n-1} : n \in \NN \bigg),
		\bigg(\frac{\beta_n}{\alpha_n} a_n - b_n : n \in \NN\bigg),
		\bigg(\frac{1}{\sqrt{a_n}} : n \in \NN\bigg) \in \calD_1^N,
	\]
	then
	\[
		\sigmaEss(A) \cap \Lambda_+ = \emptyset.
	\]
\end{theorem}
\begin{proof}
	Let $K$ be a compact interval contained in $\Lambda_+$ with non-empty interior and $i \in \{0, 1, \ldots, N-1\}$. We set
	\[
		Y_j = Z_{j+1}^{-1} X_{jN+i} Z_j
	\]
	where $Z_j$ is the matrix defined in \eqref{eq:46}. In view of Theorem~\ref{thm:1}, we have
	\begin{equation}
		\label{eq:10}
		Y_j = \varepsilon \big( \Id + \vartheta_j R_j \big)
	\end{equation}
	where $(R_j : j \in \NN_0)$ is a sequence from $\calD_1\big(K, \Mat(2, \RR)\big)$ convergent to
	\[
		\calR_i = 
		\begin{pmatrix}
			1 & 0\\
			0 & -1
		\end{pmatrix}
	\]
	uniformly on $K$. Since
	\begin{equation}
		\label{eq:73}
		\big\{ x \in \RR : \discr \calR_i(x) > 0 \big\} = \Lambda_+,
	\end{equation}
	there are $j_0 \geq 1$ and $\delta > 0$, so that for all $j \geq j_0$, and all $x \in K$,
	\begin{equation} 
		\label{eq:17}
		\discr R_j(x) \geq \delta.
	\end{equation}
	In particular, the matrix $R_j(x)$ has two eigenvalues
	\[
		\xi_j^+ = \frac{\tr R_j(x) + \varepsilon \sqrt{\discr R_j(x)}}{2},
		\qquad\text{and}\qquad
		\xi_j^- = \frac{\tr R_j(x) - \varepsilon \sqrt{\discr R_j(x)}}{2}.
	\]
	By \eqref{eq:10}, for each $x \in K$ and $j \geq j_0$, the matrix $Y_j(x)$ has two eigenvalues
	\[
		\lambda^+_j(x) = \varepsilon \big( 1 + \vartheta_j(x) \xi_j^+(x) \big),
		\qquad\text{and}\qquad
		\lambda^-_j(x) = \varepsilon \big( 1 + \vartheta_j(x) \xi_j^-(x) \big).
	\]
	In view of \eqref{eq:17} and Theorem~\ref{thm:1}, we can apply \cite[Theorem 4.4]{Discrete} to the system
	\[
		\Psi_{j+1} = Y_j \Psi_j.
	\]
	Therefore, there is $(\Psi_j^- : j \geq j_0)$, so that 
	\[
		\sup_{x \in K}{\bigg\|\frac{\Psi_j^-(x)}{\prod_{k = j_0}^{j-1} \lambda^-_k(x)} - e_2 \bigg\|} =0
	\]
	(cf. \eqref{eq:109}).
	Then the sequence $\Phi_j^- = Z_j \Psi_j^-$ satisfies
	\[
		\Phi_{j+1} = X_{jN+i} \Phi_j
	\]
	for $j \geq j_0$. We set
	\[
		\phi_1 = B_1^{-1} \cdots B^{-1}_{j_0} \Phi_{j_0}^-,
	\]
	and
	\begin{equation}
		\label{eq:22}
		\phi_{n+1} = B_n \phi_n,
	\end{equation}
	for $n > 1$. Then, for $jN+i' > j_0N+i$ with $i' \in \{0, 1, \ldots, N-1\}$, we get
	\[
		\phi_{jN+i'} = 
		\begin{cases}
			B_{jN+i'}^{-1} B_{jN+i'+1}^{-1} \cdots B_{jN+i-1}^{-1} \Phi_j^-
			&\text{if } i' \in \{0, 1, \ldots, i-1\}, \\
			\Phi_j^- & \text{if } i ' = i,\\\
			B_{jN+i'-1} B_{jN+i'-2} \cdots B_{jN+i} \Phi_j^- &
			\text{if } i' \in \{i+1, \ldots, N-1\}.
		\end{cases}
	\]
	Since for $i' \in \{0, 1, \ldots, i-1\}$,
	\[
		\lim_{j \to \infty} B_{jN+i'}^{-1} B_{jN+i'+1}^{-1} \cdots B_{jN+i-1}^{-1}
		=
		\frakB_{i'}^{-1}(0) \frakB_{i'+1}^{-1}(0) \cdots \frakB_{i-1}^{-1}(0),
	\]
	and
	\[
		\lim_{j \to \infty} Z_j e_2 = T_i (e_1 + e_2),
	\]
	we obtain
	\begin{equation}
		\label{eq:23}
		\lim_{j \to \infty}
		\sup_K{
		\bigg\|
		\frac{\phi_{jN+i'}}{\prod_{k = j_0}^{j-1} \lambda^-_k} - 
		T_{i'} (e_1 + e_2)
		\bigg\|
		} = 0.
	\end{equation}
	Analogously, we can show that \eqref{eq:23} holds true also for $i' \in \{i+1, \ldots, N-1\}$.

	Let us recall that a non-zero sequence $(u_n(x) : n \in \NN_0)$ is generalized eigenvector 
	associated with $x \in \RR$, if it satisfies \eqref{eq:108a}.

	Since $(\phi_n : j \in \NN)$ satisfies \eqref{eq:22}, the sequence $(u_n(x) : n \in \NN_0)$ defined as
	\[
		u_n(x) = 
		\begin{cases}
			\langle \phi_1(x), e_1 \rangle & \text{if } n = 0, \\
			\langle \phi_n(x), e_2 \rangle & \text{if } n \geq 1,
		\end{cases}
	\]
	is a generalized eigenvector associated to $x \in K$, provided that $(u_0, u_1) \neq 0$ on $K$. Suppose on the
	contrary that there is $x \in K$ such that $\phi_1(x) = 0$. Hence, $\phi_n(x) = 0$ for all $n \in \NN$, thus
	by \eqref{eq:23} we must have $T_0(e_1 + e_2) = 0$ which is impossible since $T_0$ is invertible.

	Next, let us observe that, by \eqref{eq:23}, for each $i' \in \{0, 1, \ldots, N-1\}$, $j > j_0$, and $x \in K$,
	\begin{equation}
		\label{eq:24}
		|u_{jN+i'}(x)|
		\leq
		c 
		\prod_{k = j_0}^{j-1} |\lambda_k^-(x)|.
	\end{equation}
	Since $(R_j : j \in \NN)$ converges to $\calR_i$ uniformly on $K$, and
	\[
		\lim_{n \to \infty} a_n = \infty,
	\]
	there is $j_1 \geq j_0$, such that for $j \geq j_1$,
	\[
		\abs{\vartheta_j} \big( |\tr R_j(x)| + \sqrt{\discr R_j(x)} \big) \leq 1.
	\]
	Therefore, for $j \geq j_1$,
	\[
		|\lambda_j^-(x)|
		=
		1 +
		\vartheta_j \frac{ \tr R_j(x) - \sqrt{\discr R_j(x)} }{2}.
	\]
	Next by the Stolz--Ces\'aro theorem and \eqref{eq:25}, we get
	\begin{align*}
		\lim_{j \to \infty} \frac{\sqrt{a_{(j+1)N+i-1}}}{j}
		&=
		\lim_{j \to \infty} \big(\sqrt{a_{(j+1)N+i-1}} - \sqrt{a_{jN+i-1}}\big) \\
		&=
		\lim_{j \to \infty} \frac{a_{(j+1)N+i-1} - a_{jN+i-1}}{\sqrt{a_{(j+1)N+i-1}} + \sqrt{a_{jN+i-1}}} 
		= 0.
	\end{align*}
	Since $\tr \calR_i = 0$, there is $j_2 \geq j_1$ such that for all $j \geq j_2$ and $x \in K$,
	\[
		j \vartheta_j \frac{\tr R_j(x) - \sqrt{\discr R_j(x)}} {2} \leq -1,
	\]
	and thus
	\[
		\sup_{x \in K} {|\lambda_j^-(x)|} \leq 1 - \frac{1}{j}.
	\]
	Consequently, by \eqref{eq:24}, there is $c' > 0$ such that for all $i' \in \{0, 1, \ldots, N-1\}$ and
	$j \geq j_2$,
	\[
		\sup_{x \in K}{|u_{jN+i'}(x)|} 
		\leq c \prod_{k = j_0}^{j-1} \bigg(1 - \frac{1}{k} \bigg) \leq \frac{c'}{j},
	\]
	hence
	\[
		\sum_{n = 0}^\infty \sup_{x \in K}{|u_n(x)|^2} < \infty.
	\]
	Now, by the proof of \cite[Theorem 5.3]{Silva2007} we conclude that $\sigmaEss(A) \cap K = \emptyset$. Since
	$K$ was arbitrary compact subinterval of $\Lambda_+$ the theorem follows.
\end{proof}

\section{Generalized Tur\'an determinants}
\label{sec:5}
In this section we study behavior of $N$-shifted generalized Tur\'an determinants on $\Lambda_-$. The good understanding
of them allows us to deduce that the measure $\mu$ restricted to $\Lambda_-$ is absolutely continuous, see
Theorem \ref{thm:4} for details. Let us recall that $N$-shifted generalized Tur\'an determinant $S_n(\eta, x)$
where $\eta \in \RR^2 \setminus \{0\}$ and $x \in \RR$, is defined as
\begin{equation}
	\label{eq:116}
	S_n(\eta, x) = a_{n+N-1}^{3/2}
	\big\langle E
	\vec{u}_{n+N},
	\vec{u}_n
	\big\rangle
\end{equation}
where $(u_n : n \in \NN_0)$ is a generalized eigenvector associated to $x$ and corresponding to $\eta$, and
\[
	E =
	\begin{pmatrix}
		0 & -1 \\
		1 & 0
	\end{pmatrix}.
\]
\begin{theorem}
	\label{thm:5}
	Let $N$ be a positive integer and $i \in \{0, 1, \ldots, N-1\}$. Let $(a_n : n \in \NN_0)$ and $(b_n : n \in \NN_0)$
	be $N$-periodically modulated Jacobi parameters so that $\frakX_0(0)$ is a non-trivial 
	parabolic element.
	If
	\[
		\bigg(\frac{\alpha_{n-1}}{\alpha_n} a_n - a_{n-1} : n \in \NN \bigg),
		\bigg(\frac{\beta_n}{\alpha_n} a_n - b_n : n \in \NN\bigg),
		\bigg(\frac{1}{\sqrt{a_n}} : n \in \NN\bigg) \in \calD_1^N,
	\]
	then the sequence $(|S_{nN+i}| : n \in \NN)$ converges locally uniformly on\footnote{By $\sS^1$ we denote the unit
	sphere in $\RR^2$.} $\sS^1 \times \Lambda_-$ to a positive continuous function.
\end{theorem}
\begin{proof}
	We start by describing the uniform diagonalization under the assumptions of the theorem. For matrices defined in
	\eqref{eq:46}, we set
	\begin{equation}
		\label{eq:35}
		Y_j = Z_{j+1}^{-1} X_{jN+i} Z_j,  
	\end{equation}
	and
	\begin{equation}
		\label{eq:30}
		\vec{v}_j(x) = Z_j^{-1}(x) \vec{u}_{jN+i}(x).
	\end{equation}
	Then
	\begin{equation}
		\label{eq:20}
	 	\vec{v}_{j+1} = Y_j \vec{v}_j.
	\end{equation}
	Fix a compact subset $K \subset \Lambda_-$. By Theorem \ref{thm:1}, we have
	\begin{equation}
		\label{eq:40}
		Y_j = \varepsilon \big(\Id + \vartheta_j R_j\big)
	\end{equation}
	where $(R_j : j \in \NN_0)$ is a sequence from $\calD_1\big(K, \Mat(2, \RR)\big)$ convergent to
	\begin{equation}
		\label{eq:33}
		\calR_i = 
		\begin{pmatrix}
			0 & -1 \\
			1 & 0
		\end{pmatrix}
	\end{equation}
	uniformly on $K$. Since
	\[
		\big\{ x \in \RR : \discr \calR_i(x) < 0 \big\} = \Lambda_-
	\]
	there are $\delta > 0$ and $j_0 \geq 1$ such that for all $j \geq j_0$ and $x \in K$,
	\[
		\discr R_j(x) \leq -\delta,
		\qquad\text{and}\qquad
		[R_j(x)]_{1, 2} < -\delta.
	\]
	Thus $R_j(x)$ has two eigenvalues $\xi_j(x)$ and $\overline{\xi_j(x)}$ where
	\begin{equation} 
		\label{eq:41}
		\xi_j(x) = \frac{\tr R_j(x) + i \sqrt{-\discr R_j(x)}}{2}.
	\end{equation}
	Moreover,
	\[
		R_j = C_j
		\begin{pmatrix}
			\xi_j & 0 \\
			0 & \overline{\xi}_j
		\end{pmatrix}
		C_j^{-1}
	\]
	where
	\[
		C_j = 
		\begin{pmatrix}
			1 & 1 \\
			\frac{\xi_j - [R_j]_{1,1}}{[R_j]_{1, 2}} & \frac{\overline{\xi_j} - [R_j]_{1,1}}{[R_j]_{1, 2}}
		\end{pmatrix}.
	\]
	In view of \eqref{eq:40}, $Y_j(x)$ has two eigenvalues $\lambda_j(x)$ and $\overline{\lambda_j(x)}$ where
	\begin{equation}
		\label{eq:65}
		\lambda_j(x) = \varepsilon \big( 1 + \vartheta_j(x) \xi_j(x) \big).
	\end{equation}
	Moreover,
	\begin{equation} 
		\label{eq:69}
		Y_j = C_j D_j C_j^{-1}
	\end{equation}
	where
	\begin{equation} 
		\label{eq:32}
		D_j = 
		\begin{pmatrix}
			\lambda_j & 0 \\
			0 & \overline{\lambda_j}
		\end{pmatrix}.
	\end{equation}
	Let us observe that, by Theorem \ref{thm:1}, both $(C_j: j \geq j_0)$ and $(D_j : j \geq j_0)$ belong to
	$\calD_1\big(K, \Mat(2, \CC)\big)$. By \eqref{eq:33}, we have
	\begin{equation}
		\label{eq:36}
		\lim_{j \to \infty} C_j = C_\infty = 
		\begin{pmatrix}
			1 & 1 \\
			-i & i
		\end{pmatrix}
	\end{equation}
	uniformly on $K$.

	Before we embark on the proof of the theorem, we show the following claim.
	\begin{claim}
		\label{clm:2}
		There is $c > 0$ so that for all $j \geq j_0$,
		\[
			\|\vec{v}_j\|
			\leq 
			c \Big(\prod_{k = j_0}^{j-1} \|D_k\| \Big) \|\vec{v}_{j_0}\|
		\]
		uniformly on $K$.
	\end{claim}
	Using \eqref{eq:32}, we have
	\[
		\vec{v}_j = Y_{j-1} \cdots Y_{j_0} \vec{v}_{j_0},
	\]
	thus
	\[
		\|\vec{v}_j\| \leq \|Y_{j-1} \cdots Y_{j_0} \| \|\vec{v}_{j_0}\|.
	\]
	Next, we write
	\begin{align*}
		Y_{j-1} Y_{j-2} \cdots Y_{j_0}
		= C_{j-1} \big(D_{j-1} C_{j-1}^{-1} C_{j-2}\big) \big(D_{j-2} C_{j-2}^{-1} C_{j-3}\big)
		\cdots \big(D_{j_0} C_{j_0}^{-1} C_{j_0-1}\big) C_{j_0-1}^{-1},
	\end{align*}
	and so
	\begin{align*}
		\big\|
		Y_{j-1} Y_{j-2} \cdots Y_{j_0}
		\big\|
		&\leq
		\|C_{j-1}\| 
		\Big\|
		\big(D_{j-1} C_{j-1}^{-1} C_{j-2}\big) \big(D_{j-2} C_{j-2}^{-1} C_{j-3}\big)
		\cdots \big(D_{j_0} C_{j_0}^{-1} C_{j_0-1}\big)
		\Big\|
		\big\|C_{j_0-1}^{-1}\big\|\\
		&\leq
		c \prod_{k = j_0}^{j-1} \|D_k\|,
	\end{align*}
	where the last estimate is the consequence of \cite[Proposition 1]{SwiderskiTrojan2019} and \eqref{eq:36},
	proving the claim.

	Now, let us define
	\begin{equation}
		\label{eq:66}
		\tilde{S}_j = a_{(j+1)N+i-1}^{3/2} (\det Z_j) \langle E \vec{v}_{j+1}, \vec{v}_j \rangle.
	\end{equation}
	Our next step is to show that $(\tilde{S}_j : j \geq j_0)$ is asymptotically close to $(S_{jN+i} : j \geq j_0)$.
	\begin{claim}
		\label{clm:6}
		We have
		\[
			\lim_{j \to \infty} \big| S_{jN+i} - \tilde{S}_j \big|= 0
		\]
		uniformly on $\sS^1 \times K$.
	\end{claim}
	For the proof we write
	\begin{align*}
		S_{jN+i} 
		&= a_{(j+1)N+i-1}^{3/2} \langle E \vec{u}_{(j+1)N+i}, \vec{u}_{jN+i} \rangle \\
		&= a_{(j+1)N+i-1}^{3/2} \langle Z_j^* E Z_{j+1} \vec{v}_{j+1}, \vec{v}_j \rangle \\
		&= a_{(j+1)N+i-1}^{3/2} (\det Z_j) \langle E Z_j^{-1} Z_{j+1} \vec{v}_{j+1}, \vec{v}_j \rangle
	\end{align*}
	where we have used that for any $Y \in \GL(2, \RR)$, 
	\[
		(Y^{-1})^* E = \frac{1}{\det Y} E Y.
	\]
	Now, by Theorem~\ref{thm:2}
	\begin{align*}
		S_{jN+i} - \tilde{S}_j &=
		a_{(j+1)N+i-1}^{3/2} (\det Z_j) 
		\big\langle E (Z_j^{-1} Z_{j+1} - \Id) \vec{v}_{j+1}, \vec{v}_j \big\rangle \\
		&= a_{(j+1)N+i-1}^{3/2} (\det Z_j) \vartheta_j 
		\big\langle E Q_j \vec{v}_{j+1}, \vec{v}_j \big\rangle.
	\end{align*}
	Observe that by \eqref{eq:32} and \eqref{eq:69}
	\[
		\|D_k\|^2 = \abs{\lambda_k}^2 = \lambda_k \overline{\lambda_k} = \det Y_k.
	\]
	Therefore, by \eqref{eq:35}, 
	\[
		\prod_{k = j_0}^{j-1} \|D_k\|^2 = \frac{\det Z_{j_0}}{\det Z_j} \frac{a_{j_0 N+i-1}}{a_{jN+i-1}}.
	\]
	Next, in view of Claim \ref{clm:2}, for $j \geq j_0$,
	\[
		\|\vec{v}_j\|^2 \lesssim \prod_{k = j_0}^{j-1} \|D_k\|^2 \lesssim \frac{1} {a_{jN+i-1} |\det Z_j|}.
	\]
	Hence, 
	\[
		\big|a_{(j+1)N+i-1}^{3/2} (\det Z_j) \vartheta_j \langle E Q_j \vec{v}_{j+1}, \vec{v}_j \rangle\big|
		\lesssim
		a_{(j+1)N+i-1}^{3/2} \vartheta_j |\det Z_j| \cdot\|Q_j\| \cdot \|\vec{v}_j\|^2,
	\]
	which is bounded by a constant multiple of $\|Q_j\|$, and the claim follows by Theorem \ref{thm:2}.

	Next we show that the sequence $(\tilde{S}_j : j \geq j_0)$ converges uniformly on $\sS^1 \times K$ to a positive
	continuous function. By \eqref{eq:66} and \eqref{eq:20}, we have
	\begin{align*}
		\tilde{S}_j 
		&= a_{(j+1)N+i-1}^{3/2} (\det Z_j) \langle E \vec{v}_{j+1}, Y_j^{-1} \vec{v}_{j+1} \rangle \\
		&= a_{(j+1)N+i-1}^{3/2} (\det Z_j) \langle (Y_j^{-1})^* E \vec{v}_{j+1}, \vec{v}_{j+1} \rangle \\
		&= a_{(j+1)N+i-1}^{3/2} (\det Z_j) (\det Y_j^{-1}) \langle E Y_j \vec{v}_{j+1}, \vec{v}_{j+1} \rangle,
	\end{align*}
	and since
	\[
		\det Y_j = \det \big( Z_{j+1}^{-1} X_{jN+i} Z_j \big),
	\]
	we obtain
	\begin{equation}
		\label{eq:43}
		\tilde{S}_j = a_{(j+1)N+i-1}^{3/2} (\det Z_{j+1}) \frac{a_{(j+1)N+i-1}}{a_{jN+i-1}} 
		\langle E Y_j \vec{v}_{j+1}, \vec{v}_{j+1} \rangle.
	\end{equation}
	By \eqref{eq:66} and \eqref{eq:20} we have
	\[
		\tilde{S}_{j+1} = a_{(j+2)N+i-1}^{3/2} (\det Z_j) 
		\langle E Y_{j+1} \vec{v}_{j+1}, \vec{v}_{j+1} \rangle.
	\]
	Therefore, by Theorem \ref{thm:1}
	\begin{align*}
		\tilde{S}_{j+1} - \tilde{S}_j
		=
		\varepsilon
		a_{(j+1)N+i-1}^{3/2} (\det Z_{j+1}) 
		\left\langle
			E W_j \vec{v}_{j+1}, \vec{v}_{j+1}
		\right\rangle
	\end{align*}
	where
	\[
		W_j =
		\sqrt{\frac{a_{(j+2)N+i-1}}{a_{(j+1)N+i-1}}} \frac{a_{(j+2)N+i-1}}{a_{(j+1)N+i-1}} \vartheta_{j+1} 
		R_{j+1} 
		-
		\frac{a_{(j+1)N+i-1}}{a_{jN+i-1}} \vartheta_j R_j.
	\]
	Since
	\[
		\sqrt{\frac{a_{(j+2)N+i-1}}{a_{(j+1)N+i-1}}} \vartheta_{j+1} = \vartheta_j
	\]
	we have
	\begin{align*}
		W_j = \vartheta_j
		\bigg(  
		\frac{a_{(j+2)N+i-1}}{a_{(j+1)N+i-1}} R_{j+1} - \frac{a_{(j+1)N+i-1}}{a_{jN+i-1}} R_j \bigg),
	\end{align*}
	and so
	\[
		\| W_j \|
		\lesssim 
		\vartheta_j
		\Big( 
		\Big| \Delta \Big( \frac{a_{jN+i-1}}{a_{jN+i}} \Big) \Big| +
		\big\| \Delta R_j \big\|
		\Big)
	\]
	where for a sequence $(x_n : n \in \NN)$ we have set
	\[
		\Delta x_n = x_{n+1} - x_n.
	\]
	On the other hand, by \eqref{eq:43},
	\[
		\tilde{S}_j = 
		\varepsilon a_{(j+1)N+i-1}^{3/2} \frac{a_{(j+1)N+i-1}}{a_{jN+i-1}} (\det Z_{j+1}) \vartheta_j
		\big\langle 
			E R_j \vec{v}_{j+1}, \vec{v}_{j+1} 
		\big\rangle,
	\]
	and since 
	\begin{equation}
		\label{eq:28}
		\lim_{j\to \infty} \sym (E R_j) =
		\sym(E \calR_i) = -\Id,
	\end{equation}
	we get
	\[
		|\tilde{S}_j| \gtrsim 
		a_{(j+1)N+i-1}^{3/2} \frac{a_{(j+1)N+i-1}}{a_{jN+i-1}} \vartheta_j
		|\det Z_{j+1}| \cdot
		\|\vec{v}_{j+1} \|^2.
	\]
	Consequently, we arrive at
	\[
		|\tilde{S}_{j+1} - \tilde{S}_j| \lesssim 
		\Big( 
		\Big| \Delta \Big( \frac{a_{jN+i-1}}{a_{jN+i}} \Big) \Big| +
		\| \Delta R_j \|
		\Big) |\tilde{S}_j|.
	\]
	Since $\tilde{S}_j \neq 0$ on $K$, we get
	\[
		\sum_{j = j_0}^\infty
		\sup_{\eta \in \sS^1} \sup_{x \in K} {
		\bigg|
		\frac{\abs{\tilde{S}_{j+1}(\eta, x)}}{\abs{\tilde{S}_j(\eta, x)}} - 1 
		\bigg|}
		\lesssim
		\sum_{j = j_0}^\infty
		\Big| \Delta \Big( \frac{a_{jN+i-1}}{a_{jN+i}} \Big) \Big| + 
		\sup_{x \in K}{\| \Delta R_j(x) \|},
	\]
	which implies that the product
	\[
		\prod_{k = j_0}^\infty \bigg(1 + \frac{\abs{\tilde{S}_{k+1}} - \abs{\tilde{S}_k}}{\abs{\tilde{S}_k}}\bigg)
	\]
	converges uniformly on $\sS^1 \times K$ to a positive continuous function. Because
	\[
		\bigg| \frac{\tilde{S}_j}{\tilde{S}_{j_0}} \bigg|
		=
		\prod_{k = j_0}^{j-1} \bigg(1 + \frac{\abs{\tilde{S}_{k+1}} - \abs{\tilde{S}_k}}{\abs{\tilde{S}_k}}\bigg),
	\]
	the same holds true for the sequence $(\tilde{S}_j : j \geq j_0)$. In view of Claim \ref{clm:6},
	the proof is completed.
\end{proof}
From now on, if $(a_n : n \in \NN_0)$ and $(b_n : n \in \NN_0)$ are $N$-modulated Jacobi parameters 
satisfying \eqref{eq:40b} and \eqref{eq:40c}, for fixed $i \in \{0, 1, \ldots, N-1\}$ and a compact subset of
$K \subset \Lambda_-$ we use the diagonalization constructed at the beginning of the proof of Theorem \ref{thm:5}.
\begin{corollary}
	\label{cor:3}
	Let the hypotheses of Theorem~\ref{thm:5} be satisfied. For each compact subset $K \subset \Lambda_-$,
	there is a constant $c>0$, such that for every generalized eigenvector $u$ associated with $x \in K$,
	\[
		\sup_{j \in \NN_0}{
		\sqrt{a_{(j+1)N+i-1}} \big( u_{jN+i-1}^2 + u_{jN+i}^2 \big)} \leq c (u_0^2 + u_1^2).
	\]
\end{corollary}
\begin{proof}
	Without loss of generality we assume that $u_0^2 + u_1^2 = 1$. Let us fix a compact subset $K \subset \Lambda_-$.
	By Theorem~\ref{thm:1} we have
	\begin{align*}
		\tilde{S}_j 
		&= 
		a_{(j+1)N+i-1}^{3/2} (\det Z_j)
		\varepsilon \langle E (\Id + \vartheta_j R_j) \vec{v}_j, \vec{v}_j \rangle \\
		&=
		\varepsilon a_{(j+1)N+i-1}^{3/2} (\det Z_j) \vartheta_j 
		\langle E R_j \vec{v}_j, \vec{v}_j \rangle.
	\end{align*}
	thus in view of \eqref{eq:28},
	\[
		|\tilde{S}_j| \gtrsim
		a_{(j+1)N+i-1}^{3/2} \vartheta_j |\det(Z_j)| \cdot \| \vec{v}_j \|^2
	\]
	on $\sS^1 \times K$. Since $(|\tilde{S}_j| : j \in \NN_0)$ is uniformly bounded on $\sS^1 \times K$, and
	\[
		|\det Z_j | \vartheta_j \gtrsim a_{(j+1)N+i-1}^{-1},
	\]
	by \eqref{eq:30} we conclude that
	\begin{align*}
		\| \vec{u}_{jN+i} \|^2 
		&\leq \|Z_j\|^2 \| \vec{v}_j \|^2 \\
		&\leq \|Z_j\|^2 \frac{1}{\sqrt{a_{(j+1)N+i-1}}}.
	\end{align*}
	Because $(Z_j)$ is uniformly bounded on $K$, the proof is complete.
\end{proof}

\section{Asymptotics of the generalized eigenvectors}
\label{sec:6}
In this section we study the asymptotic behavior of generalized eigenvectors. We prove the following theorem.
\begin{theorem} 
	\label{thm:3}
	Let $N$ be a positive integer and $i \in \{0, 1, \ldots, N-1\}$. Let $(a_n : n \in \NN_0)$ and $(b_n : n \in \NN_0)$
	be $N$-periodically modulated Jacobi parameters so that $\frakX_0(0)$ is a non-trivial parabolic element.
	If
	\[
		\bigg(\frac{\alpha_{n-1}}{\alpha_n} a_n - a_{n-1} : n \in \NN \bigg),
		\bigg(\frac{\beta_n}{\alpha_n} a_n - b_n : n \in \NN\bigg),
		\bigg(\frac{1}{\sqrt{a_n}} : n \in \NN\bigg) \in \calD_1^N,
	\]
	then for each compact subset $K \subset \Lambda_-$ there are $j_0 \in \NN$, and a continuous function 
	$\vphi: \sS^1 \times K \rightarrow \CC$ such that every generalized eigenvector $(u_n : n \in \NN_0)$ 
	\[
		\lim_{j \to \infty}
		\sup_{\eta \in \sS^1} \sup_{x \in K}
		\bigg|
		\frac{\sqrt{a_{(j+1)N+i-1}}}{\prod_{k = j_0}^{j-1} \lambda_k(x)}
		\Big(
		u_{(j+1)N+i}(\eta, x) - \overline{\lambda_j(x)} u_{jN+i}(\eta, x)
		\Big)
		-
		\vphi(\eta, x)
		\bigg|
		=0.
	\]
	Moreover,
	\[ 
		\frac{u_{jN+i}(\eta, x)}{\prod_{k = j_0}^{j-1} |\lambda_k(x)|}
		=
		\frac{\abs{\vphi(\eta, x)}}{\sqrt{\alpha_{i-1} \big|\tau(x) \big|}}
		\sin\Big(\sum_{k=j_0}^{j-1} \theta_k(x) + \arg \vphi(\eta, x)\Big)
		+
		E_j(\eta, x)
	\]
	where
	\[
		\theta_k(x) = \arccos \bigg(\frac{\tr Y_k(x)}{2 \sqrt{\det Y_k(x)}} \bigg),
	\]
	and
	\[
		\sup_{\eta \in \sS^1} \sup_{x \in K} |E_j(\eta, x)|
		\leq
		c \sum_{k = j}^\infty
		\sup_{x \in K} \Big( \big\|\Delta C_k(x) \big\| + \big\|\Delta R_k(x) \big\|\Big).
	\]
\end{theorem}
\begin{proof}
	We use the diagonalization constructed at the beginning of the proof of Theorem \ref{thm:5} as well as the
	notation introduced there. For $j > j_0$, we set
	\begin{equation} 
		\label{eq:145}
		\phi_j = \frac{u_{(j+1)N+i} - \overline{\lambda_j} u_{jN+i}}{\prod_{k = j_0}^{j-1} \lambda_k}.
	\end{equation}
	Observe that there is $c > 0$ so that for all $j \in \NN_0$ and $x \in K$,
	\begin{equation}
		\label{eq:21}
		\left\|
		Z_j^t e_2 - 
		\begin{pmatrix}
			1 & 1 \\
			1 & 1
		\end{pmatrix}
		T_i^t e_2
		\right\|
		\leq
		c \vartheta_j.
	\end{equation}
	We are going to show that the sequence $(\sqrt{a_{(j+1)N+i-1}} \phi_j : j > j_0)$ converges uniformly on $K$. Let
	\[
		\vec{q}_j = Z_j^{-1} \vec{u}_{jN+i}.
	\]
	By \eqref{eq:32} we have $\|D_j\| = |\lambda_j|$. Hence,  by Claim \ref{clm:2}, we get
	\begin{align*}
		\big| u_{(j+1)N+i} - \langle \vec{q}_{j+1}, Z_j^t e_2 \rangle \big|
		&=
		\big|
		\big\langle \vec{q}_{j+1}, \big(Z_{j+1}^t - Z_j^t\big) e_2 \big\rangle
		\big| \\
		&\lesssim
		\big\|\vec{q}_{j+1} \big\| \cdot \big|\vartheta_{j+1} - \vartheta_j\big| \\
		&\lesssim
		\Big( \prod_{k = j_0}^{j-1} \abs{\lambda_k}\Big) \big|\vartheta_{j+1} - \vartheta_j\big|.
	\end{align*}
	Therefore,
	\[
		\lim_{j \to \infty}
		\sqrt{a_{(j+1)N+i-1}} \frac{\big| u_{(j+1)N+i} - \langle \vec{q}_{j+1}, Z_j^t e_2 \rangle \big|}
		{\prod_{k = j_0}^{j-1} \abs{\lambda_k}}
		=0
	\]
	uniformly on $K$. Next, by \eqref{eq:69}, we can write
	\begin{align*}
		\big(Y_j - \overline{\lambda_j} \Id \big) \vec{q}_j
		=
		C_j
		\begin{pmatrix}
			\lambda_j - \overline{\lambda_j} & 0 \\
			0 & 0
		\end{pmatrix}
		C_j^{-1}
		\vec{q}_j, 
	\end{align*}
	therefore, by \eqref{eq:21}, we get
	\begin{align*}
		\bigg|
		\bigg\langle \big(Y_j - \overline{\lambda_j} \Id \big) \vec{q}_j, Z_j^t e_2 \bigg\rangle
		-
		\bigg\langle \big(Y_j - \overline{\lambda_j} \Id \big) \vec{q}_j, 
			\begin{pmatrix}
				1 & 1 \\
				1 & 1
			\end{pmatrix}
			T_i^t e_2
		\bigg\rangle
		\bigg|
		&\leq
		\vartheta_j 
		\big|\lambda_j - \overline{\lambda_j}\big| \cdot \big\|\vec{q}_j \big\| \\
		&\lesssim
		\vartheta_j^2 \Big(\prod_{k = j_0}^{j-1} \abs{\lambda_k}\Big)
	\end{align*}
	where in the last estimate we have used
	\begin{equation}
		\label{eq:67}
		\lambda_j - \overline{\lambda_j} = i \vartheta_j \sqrt{-\discr R_j}
	\end{equation}
	which is a consequence of \eqref{eq:65} and \eqref{eq:41}. Hence, it is enough to show that the sequence 
	$\big(\sqrt{a_{(j+1)N+i-1}} \tilde{\phi}_j : j > j_0 \big)$ where
	\[
		\tilde{\phi}_j = 
		\frac{\bigg\langle \big(Y_j - \overline{\lambda_j} \Id \big) \vec{q}_j, 
			\begin{pmatrix}
				1 & 1 \\
				1 & 1
			\end{pmatrix}
			T_i^t
			e_2
		\bigg\rangle}
		{\prod_{k = j_0}^{j-1} \lambda_k}
	\]
	converges uniformly on $K$. For the proof, we write
	\begin{align*}
		&\frac{\sqrt{a_{(j+1)N+i-1}}}{\lambda_{j-1}}
		(Y_j - \overline{\lambda_j}\Id) Y_{j-1}
		-
		\sqrt{a_{jN+i-1}} (Y_{j-1} - \overline{\lambda_{j-1}}\Id) \\
		&\qquad=
		\frac{\sqrt{a_{(j+1)N+i-1}}}{\lambda_{j-1}}
		\Big(
		C_j \big(D_j - \overline{\lambda_j}\Id\big) C_j^{-1} Y_{j-1}
		-
		C_j \big(D_j - \overline{\lambda_j}\Id\big) D_{j-1} C_{j-1}^{-1}\Big) \\
		&\qquad\phantom{=}
		-
		\sqrt{a_{jN+i-1}}
		\Big(
		C_{j-1} \big(D_{j-1} - \overline{\lambda_{j-1}}\Id\big) C_{j-1}^{-1}
		-
		C_j \big(D_{j-1} - \overline{\lambda_{j-1}}\Id\big) C_{j-1}^{-1}
		\Big) \\
		&\qquad\phantom{=}
		+
		C_j
		\bigg(
		\frac{\sqrt{a_{(j+1)N+i-1}}}{\lambda_{j-1}}
		\big(D_j - \overline{\lambda_j}\Id\big) D_{j-1} 
		-
		\sqrt{a_{jN+i-1}}
		\big(D_{j-1} - \overline{\lambda_{j-1}}\Id\big)\bigg) C_{j-1}^{-1}.
	\end{align*}
	The first two terms are estimated as follows,
	\begin{align*}
		&
		\frac{\sqrt{a_{(j+1)N+i-1}}}{\abs{\lambda_{j-1}}}
		\left\|
		C_j \big(D_j - \overline{\lambda_j}\Id\big) C_j^{-1} Y_{j-1}
		-
		C_j \big(D_j - \overline{\lambda_j}\Id\big) D_{j-1} C_{j-1}^{-1}
		\right\| \\
		&\qquad\qquad\lesssim
		\sqrt{a_{(j+1)N+i-1}}
		\big\|D_j - \overline{\lambda_j}\Id\big\|
		\cdot
		\|\Delta C_{j-1}\| \\
		&\qquad\qquad\lesssim
		\|\Delta C_{j-1}\|,
	\end{align*}
	and
	\begin{align*}
		&
		\sqrt{a_{jN+i-1}}
		\left\|
		C_{j-1} \big(D_{j-1} - \overline{\lambda_{j-1}}\Id\big) C_{j-1}^{-1}
		-
		C_j \big(D_{j-1} - \overline{\lambda_{j-1}}\Id\big) C_{j-1}^{-1}
		\right\| \\
		&\qquad\qquad\lesssim
		\sqrt{a_{jN+i-1}}
		\big\|D_{j-1} - \overline{\lambda_{j-1}}\Id\big\| \cdot \big\|\Delta C_{j-1} \big\| \\
		&\qquad\qquad\lesssim
		\big\|\Delta C_{j-1} \big\|.
	\end{align*}
	Next, by \eqref{eq:67} and \eqref{eq:68}, we write
	\begin{align*}
		&
		\frac{\sqrt{a_{(j+1)N+i-1}}}{\lambda_{j-1}}
		\big(D_j - \overline{\lambda_j}\Id\big) D_{j-1}
		-
		\sqrt{a_{jN+i-1}}
		\big(D_{j-1} - \overline{\lambda_{j-1}}\Id\big) \\
		&\qquad\qquad=
		\Big(\sqrt{-\discr R_j} - \sqrt{-\discr{R_{j-1}}}\Big)
		\begin{pmatrix}
			i \sqrt{\alpha_{i-1} |\tau(x)|} & 0 \\
			0 & 0
		\end{pmatrix},
	\end{align*}
	thus for the last term we get
	\begin{align*}
		\left\|
		C_j
		\bigg(
		\frac{\sqrt{a_{(j+1)N+i-1}}}{\lambda_{j-1}}
		\big(D_j - \overline{\lambda_j} \Id \big) D_{j-1} 
		-
		\sqrt{a_{jN+i-1}}
		\big(D_{j-1} - \overline{\lambda_{j-1}}\Id\big)\bigg) C_{j-1}^{-1}
		\right\| \lesssim
		\|\Delta R_{j-1} \|.
	\end{align*}
	Therefore, by Claim \ref{clm:2}, we obtain
	\begin{align*}
		\big|
		\sqrt{a_{(j+1)N+i-1}} \tilde{\phi}_j - \sqrt{a_{jN+i-1}} \tilde{\phi}_{j-1}
		\big|
		\lesssim
		\big\|\Delta C_{j-1} \big\| + \big\|\Delta R_{j-1} \big\|.
	\end{align*}
	Consequently, the sequence $\big( \sqrt{a_{(j+1)N+i-1}} \tilde{\phi}_j : j > j_0 \big)$ converges uniformly on 
	$\sS^1 \times K$. Hence, there is a function $\vphi: \sS^1 \times K \rightarrow \RR$, so that
	\begin{equation} 
		\label{eq:146}
		\vphi = \lim_{j \to \infty} \sqrt{a_{(j+1)N+i-1}} \phi_j
	\end{equation}
	uniformly on $\sS^1 \times K$. In particular, we get
	\begin{align*}
		\lim_{j \to \infty}
		\sup_{\eta \in \sS^1}
		\sup_{x \in K}{
		\bigg|
		\sqrt{a_{(j+1)N+i-1}} 
		\frac{u_{(j+1)N+i}(\eta, x) - \overline{\lambda_j(x)} u_{jN+i}(\eta, x)}{\prod_{k = j_0}^{j-1} |\lambda_k(x)|}
		-
		\vphi(\eta, x) \prod_{k = j_0}^{j-1} \frac{\lambda_k(x)}{|\lambda_k(x)|}
		\bigg|}
		=0.
	\end{align*}
	Since $u_n(\eta, x) \in \RR$, by taking imaginary part we conclude that
	\begin{align*}
		\lim_{j \to \infty}
		\sup_{\eta \in \sS^1}
		\sup_{x \in K}
		\bigg|
		&
		\sqrt{a_{(j+1)N+i-1}} \vartheta_j(x) \sqrt{-\discr R_j(x)}
		\frac{u_{jN+i}(\eta, x)}{\prod_{k = j_0}^{j-1} |\lambda_k(x)|} \\
		&\qquad-
		2 \abs{\vphi(\eta, x)} 
		\sin\Big(\sum_{k = j_0}^{j-1} \theta_k(x) + \arg \vphi(\eta, x) \Big)
		\bigg|
		=0
	\end{align*}
	where we have also used that
	\[
		\Im(\lambda_j(x)) = \frac{1}{2} \vartheta_j \sqrt{-\discr(R_j(x))}.
	\]
	Lastly, observe that
	\[
		\bigg|
		\frac{1}{\sqrt{-\discr R_j(x)}} - \frac{1}{2}\bigg| 
		\lesssim
		\sum_{k = j}^\infty \big\|\Delta R_k(x)\big\|,
	\]
	which together with
	\[
		\sqrt{a_{(j+1)N+i-1}} \vartheta_j(x) = \sqrt{\alpha_{i-1} |\tau(x)|},
	\]
	completes the proof.
\end{proof}

\section{Approximation procedure}
\label{sec:7}
In this section we describe the approximation procedure which allows us to show that the measure $\mu$ is absolutely
continuous on $\Lambda_-$ as well as to find its density, see Theorem \ref{thm:4}. We can also identify the function
$\vphi$ in Theorem \ref{thm:3}, see Theorem \ref{thm:6} for details.

Let $(a_n : n \in \NN_0)$ and $(b_n : n \in \NN_0)$ be $N$-periodically modulated Jacobi parameters. For a given
$L \in \NN$, we consider the truncated sequences $(a^L_n : n \in \NN_0)$ and $(b^L_n : n \in \NN_0)$ that are defined as
\begin{subequations}
	\begin{equation}
	\label{eq:42a}
	a^L_n = 
	\begin{cases}
		a_n & \text{if } 0 \leq n < L+N, \\
		a_{L+i} & \text{if } L+N \leq n, \text{ and } n-L \equiv i \bmod N,
	\end{cases}
	\end{equation}
	and
	\begin{equation}
	\label{eq:42b}
	b^L_n =
	\begin{cases}
		b_n & \text{if } 0 \leq n < L+N, \\
		b_{L+i} & \text{if } L+N \leq n, \text{ and } n-L \equiv i \bmod N,
	\end{cases}
	\end{equation}
\end{subequations}
where $i \in \{0, 1, \ldots, N-1\}$. Let
\[
	X_n^L(x) = \prod_{j=n}^{n+N-1}
	\begin{pmatrix}
		0 & 1\\
		-\frac{a_{j-1}^L}{a_j^L} & \frac{x - b_j^L}{a_j^L}
	\end{pmatrix}.
\]
By $(p^L_n : n \in \NN_0)$ we denote the sequence of orthogonal polynomials corresponding to the 
sequences $a^L$ and $b^L$. Let $\mu_L$ be their orthonormalizing measure.
\begin{lemma}
	\label{lem:2}
	Let $(L_j : j \in \NN)$ be an increasing sequence of positive integers.
	Let $K$ be a compact subset of $\RR$. Suppose that
	\[
		\sup_{j \in \NN} {\sup_{x \in K} \| X_{L_j}(x) \|} < \infty.
	\]
	If
	\[
		\lim_{j \to \infty} a_{L_j-1} = \infty \quad \text{and} \quad
		\lim_{j \to \infty} \big( a_{L_j+N-1} - a_{L_j-1} \big) = 0,
	\]
	then
	\begin{equation}
		\label{eq:74}
		\lim_{j \to \infty} 
		a_{L_j+N-1} \cdot \sup_{x \in K}{\big\| X_{L_j+N}^{L_j}(x) - X_{L_j}(x) \big\|} 
		= 0.
	\end{equation}
	Moreover,
	\begin{align}
		\label{eq:48}
		&\lim_{j \to \infty} a_{L_j+N-1} \cdot \sup_{x \in K} \Big| \det X^{L_j}_{L_j+N}(x) - \det X_{L_j}(x) \Big| = 0, \\
		\label{eq:49}
		&\lim_{j \to \infty} a_{L_j+N-1} \cdot \sup_{x \in K} \Big| \discr X^{L_j}_{L_j+N}(x) - \discr X_{L_j}(x) \Big| = 0.
	\end{align}
\end{lemma}
\begin{proof}
	Let $L \in \{ L_j : j \in \NN \}$.
	By \cite[Corollary 4]{SwiderskiTrojan2019}
	\[
		\big\| X_{L+N}^L(x) - X_L(x) \big\| \leq \big\| X_L(x) \big\|\cdot
		\bigg| \frac{a_{L+N-1}}{a_{L-1}} - 1 \bigg|,
	\]
	which easily leads to \eqref{eq:74}. Next, we write
	\[
		a_{L+N-1} \Big( \det X^L_{L+N} - \det X_L \Big) = a_{L+N-1} - a_{L-1},
	\]
	proving \eqref{eq:48}. Lastly,
	\[
		\Big( \tr X^L_{L+N}(x) \Big)^2 - \Big( \tr X_{L}(x) \Big)^2 =
		\tr \Big( X_{L+N}^L(x) - X_L(x) \Big) \tr \Big( X^L_{L+N}(x) + X_{L}(x) \Big)
	\]
	thus by \eqref{eq:74} and \eqref{eq:48}, we conclude \eqref{eq:49}.
\end{proof}

\begin{lemma} 
	\label{lem:3}
	Let $N$ be a positive integer. Suppose that $(a_n : n \in \NN_0)$ and $(b_n : n \in \NN_0)$ are $N$-periodically
	modulated Jacobi parameters such that 
	\[
		\lim_{n \to \infty} \bigg| \frac{\alpha_{n-1}}{\alpha_n} a_{n} - a_{n-1} - s_n \bigg| = 0 
		\qquad \text{and} \qquad
		\lim_{n \to \infty} \bigg| \frac{\beta_n}{\alpha_n} a_{n} - b_{n} -r_n \bigg| = 0
	\]
	for certain $N$-periodic sequences $(s_n : n \in \ZZ)$ and $(r_n : n \in \ZZ)$. Then for every compact 
	subset $K \subset \CC$,
	\[
		\lim_{L \to \infty} a_{L+N-1} \cdot \sup_{x \in K} \big\| X_{L+N} - X_L \big\| = 0.
	\]
	Moreover,
	\begin{align}
		\label{eq:60}
		&\lim_{L \to \infty} a_{L+N-1} \cdot \sup_{x \in K} \Big| \det X_{L+N}(x) - \det X_L(x) \Big| = 0, \\
		\label{eq:72}
		&\lim_{L \to \infty} a_{L+N-1} \cdot \sup_{x \in K} \Big| \discr X_{L+N}(x) - \discr X_L(x) \Big| = 0.
	\end{align}
\end{lemma}
\begin{proof}
	We notice that
	\[
		X_{L+N} - X_L =
		\sum_{k=0}^{N-1} 
		\bigg( \prod_{j=k+1}^{N-1} B_{L+N+j} \bigg) 
		\big( B_{L+N+k} - B_{L+k} \big)
		\bigg( \prod_{j=k+1}^{N-1} B_{L+j} \bigg),
	\]
	thus
	\begin{equation}
		\label{eq:88}
		a_{L+N-1} \| X_{L+N} - X_L \| 
		\leq
		\sum_{k=0}^{N-1} \frac{a_{L+N-1}}{a_{L+k}}
		\bigg( \prod_{j=k+1}^{N-1} \| B_{L+N+j} \| \bigg) 
		a_{L+k} \big\| B_{L+N+k} - B_{L+k} \big\|
		\bigg( \prod_{j=k+1}^{N-1} \| B_{L+j} \| \bigg).
	\end{equation}
	Next, we compute
	\begin{equation}
		\label{eq:87}
		a_{L+k} \big( B_{L+N+k}(x) - B_{L+k}(x) \big) =
		\begin{pmatrix}
			0 & 0 \\
			a_{L+k-1} -\frac{a_{L+k+N-1}}{a_{L+k+N}} a_{L+k} &
			x \Big( \frac{a_{L+k}}{a_{L+k+N}} - 1 \Big) +
			b_{L+k} - \frac{b_{L+k+N}}{a_{L+k+N}} a_{L+k}
		\end{pmatrix}.
	\end{equation}
	Since $(a_n : n \in \NN_0)$ and $(b_n : n \in \NN_0)$ are $N$-periodically modulated
	\begin{equation}
		\label{eq:97}
		\lim_{L \to \infty} \frac{a_{L+k}}{a_{L+k+N}} = 1.
	\end{equation}
	By $N$-periodicity of $(\alpha_n : n \in \ZZ)$, we have
	\[
		a_{L+k-1} -\frac{a_{L+k+N-1}}{a_{L+k+N}} a_{L+k} = 
		a_{L+k-1} - \frac{\alpha_{L+k-1}}{\alpha_{L+k}} a_{L+k}
		+ \bigg( \frac{\alpha_{L+k+N-1}}{\alpha_{L+k+N}} a_{L+k+N} - a_{L+k+N-1} \bigg) 
		\frac{a_{L+k}}{a_{L+k+N}},
	\]
	hence, $N$-periodicity of $(s_n : n \in \ZZ)$ and \eqref{eq:97} leads to
	\begin{equation}
		\label{eq:53}
		\lim_{L \to \infty} \bigg( a_{L+k-1} -\frac{a_{L+k+N-1}}{a_{L+k+N}} a_{L+k} \bigg) =
		\lim_{L \to \infty} \bigg( -s_{L+k} + s_{L+k+N} \frac{a_{L+k}}{a_{L+k+N}} \bigg) = 0.
	\end{equation}
	Similarly, we write
	\[
		b_{L+k} - \frac{b_{L+k+N}}{a_{L+k+N}} a_{L+k} =
		b_{L+k} - \frac{\beta_{L+k}}{\alpha_{L+k}} a_{L+k} +
		\bigg( \frac{\beta_{L+k+N}}{\alpha_{L+k+N}} a_{L+k+N} - b_{L+k+N} \bigg) \frac{a_{L+k}}{a_{L+k+N}},
	\]
	and by $N$-periodicity of $(r_n : n \in \ZZ)$ and \eqref{eq:97},
	\begin{equation}
		\label{eq:54}
		\lim_{L \to \infty} \bigg( b_{L+k} - \frac{b_{L+k+N}}{a_{L+k+N}} a_{L+k} \bigg) =
		\lim_{L \to \infty} \bigg( -r_{L+k} + r_{L+k+N} \frac{a_{L+k}}{a_{L+k+N}} \bigg) = 0.
	\end{equation}
	Consequently, by inserting \eqref{eq:97}--\eqref{eq:54} into \eqref{eq:87}, we get
	\[
		\lim_{L \to \infty} a_{L+k} \cdot \sup_{x \in K} \| B_{L+k+N}(x) - B_{L+k}(x) \| = 0.
	\]
	Hence, by \eqref{eq:88} we obtain
	\[
		\lim_{L \to \infty} a_{L+N-1} \cdot \sup_{x \in K} \| X_{L+N}(x) - X_L(x) \| = 0.
	\]
	The proofs of \eqref{eq:60} and \eqref{eq:72} are analogous to the proof of Lemma \ref{lem:2}.
\end{proof}

\subsection{Tur\'an determinants}
Let us recall that $N$-shifted Tur\'an determinants are defined as
\begin{equation}
	\label{eq:135}
	\scrD_n(x) = p_n(x) p_{n+N-1}(x) - p_{n-1}(x) p_{n+N}(x),
\end{equation}
which in terms of the notation introduced in Section \ref{sec:5} takes a form
\begin{equation}
	\label{eq:44}
	\scrD_n(x) 
	= 
	a_{n+N-1}^{-3/2} S_n(e_2, x) 
	=
	\big\langle
	E \vec{p}_{n+N}(x), \vec{p}_n(x)
	\big\rangle
\end{equation}
where
\[
	\vec{p}_n(x) =
	\begin{pmatrix}
		p_{n-1}(x) \\
		p_n(x)
	\end{pmatrix}, \quad n \geq 1.
\]
Let us denote by $(\scrD_n^L : n \in \NN)$ the sequence \eqref{eq:135} associated to the polynomials
$(p_n^L : n \geq 0)$, namely
\[
	\scrD_n^L(x) =
	\big\langle
	E
	\vec{p}^L_{n+N}(x),
	\vec{p}^L_{n}(x)
	\big\rangle
\]
where
\[
	\vec{p}^L_n(x) =
	\begin{pmatrix}
		p^L_{n-1}(x) \\
		p^L_n(x)
	\end{pmatrix}, \quad n \geq 1.
\]
\begin{lemma} 
	\label{lem:1} 
	For all $k \in \NN, x \in \RR$ and $L \in \NN$, we have
	\begin{equation}
		\label{eq:45}
		\scrD_{L+kN}^L(x) = \scrD_{L+N}^L(x).
	\end{equation}
	Let $(L_j : j \in \NN)$ be an increasing sequence of positive integers. Suppose that for a compact set $K \subset \RR$,
	\[	
		\sup_{j \in \NN} \sup_{x \in K} \sqrt{a_{L_j+N-1}} 
		\big( p_{L_j+N-1}^2(x) + p_{L_j+N}^2(x) \big) < \infty,
	\]
	and
	\[
		\sup_{j \in \NN} \sup_{x \in K} \| X_{L_j}(x) \| < \infty.
	\]
	If
	\[
		\lim_{j \to \infty} a_{L_j-1} = \infty \quad \text{and} \quad
		\lim_{j \to \infty} \big( a_{L_j+N-1} - a_{L_j-1} \big) = 0,
	\]
	then
	\begin{equation}
		\label{eq:57}
		\lim_{j \to \infty} 
		a_{L_j+N-1}^{3/2} 
		\cdot \sup_{x \in K}{\Big| \scrD^{L_j}_{L_j+N}(x) - \scrD_{L_j}(x) \Big|} = 0.
	\end{equation}
\end{lemma}
\begin{proof}
	Since
	\begin{align*}
		\scrD_n &=
		\big\langle 
		E \vec{p}_{n+N}, X_n^{-1} \vec{p}_{n+N}
		\big\rangle \\
		&=
		\big\langle 
		(X_n^{-1})^* E \vec{p}_{n+N}, \vec{p}_{n+N}
		\big\rangle \\
		&=
		(\det X_n)^{-1}
		\big\langle 
		E X_n \vec{p}_{n+N}, \vec{p}_{n+N}
		\big\rangle,
	\end{align*}
	we have
	\begin{equation}
		\label{eq:55}
		\scrD_{L+(k+1)N}^L - \scrD_{L+kN}^L = 
		\Big\langle 
		E \Big( X_{L+(k+1)N}^L - \big(\det X_{L+kN}^L\big)^{-1} X_n \Big) \vec{p}^L_{L+(k+1)N}, \vec{p}_{L+(k+1)N}^L
		\Big\rangle.
	\end{equation}
	By \eqref{eq:42a} and \eqref{eq:42b}, for $k \in \NN$, we have
	\[
		 X_{L+(k+1)N}^L = X_{L+kN}^L,
	\]
	and
	\[
		\det X_{L+kN}^L = 1,
	\]
	thus \eqref{eq:45} can be deduce from \eqref{eq:55}. 

	Let $L \in \{L_j : j \in \NN\}$. To prove \eqref{eq:57}, we observe that \cite[Proposition 5]{SwiderskiTrojan2019} implies that
	\[
		\Big| \scrD^L_{L+N}(x) - \scrD_L(x) \Big| \leq 
		\big\| X_L(x) \big\| \cdot \bigg| \frac{a_{L+N-1}}{a_{L-1}} - 1 \bigg|
		\Big( p_{L+N-1}^2(x) + p_{L+N}^2(x) \Big).
	\]
	Therefore, for a certain constant $c>0$,
	\begin{align*}
		a_{L+N-1}^{3/2} \cdot \sup_{x \in K} \Big| \scrD^L_{L+N}(x) - \scrD_L(x) \Big| 
		&\leq 
		c a_{L+N-1} \bigg| \frac{a_{L+N-1}}{a_{L-1}} - 1 \bigg| \\
		&= c \frac{a_{L+N-1}}{a_{L-1}} \big| a_{L+N-1} - a_{L-1} \big|,
	\end{align*}
	which concludes the proof.
\end{proof}

\begin{theorem} 
	\label{thm:4}
	Let $N$ be a positive integer. Let $(a_n : n \in \NN_0)$ and $(b_n : n \in \NN_0)$ be $N$-periodically modulated
	Jacobi parameters. Suppose that there are $(L_j : j \in \NN)$ an increasing sequence of positive integers and $K$
	a compact interval with non-empty interior contained in 
	\[
		\Big\{ 
		x \in \RR : \lim_{j \to \infty} a_{L_j+N-1} \cdot \discr X_{L_j}(x) \text{ exists and is negative} 
		\Big\}
	\]
	such that
	\[
			\sup_{j \in \NN} \sup_{x \in K} \| X_{L_j}(x) \| < \infty,
	\]
	and
	\[
			\sup_{j \in \NN} \sup_{x \in K} \sqrt{a_{L_j+N-1}} \big( p_{L_j+N-1}^2(x) + p_{L_j+N}^2(x) \big) < \infty.
	\]
	Assume that there is a function $g : K \to (0, \infty)$ such that
	\[
		\lim_{j \to \infty} \sup_{x \in K} \Big| a_{L_j+N-1}^{3/2} \big| \scrD_{L_j}(x) \big| - g(x) \Big| = 0.
	\]
	If
	\[
		\lim_{j \to \infty} a_{L_j-1} = \infty, \quad\text{and}\quad
		\lim_{j \to \infty} \big( a_{L_j+N-1} - a_{L_j-1} \big) = 0,
	\]
	then each $\nu$, a weak accumulation point of the sequence $(\mu_{L_j} : j \in \NN)$, is absolutely continuous on $K$
	with the density
	\[
		\nu'(x) = \frac{\sqrt{-h(x)}}{2 \pi g(x)} \qquad x \in K
	\]
	where
	\begin{equation}
		\label{eq:38}
		h(x) = \lim_{j \to \infty} a_{L_j+N-1} \cdot \discr X_{L_j}(x) \qquad x \in K.
	\end{equation}
\end{theorem}
\begin{proof}
	By Lemma~\ref{lem:2}, there are $\delta > 0$ and $j_0 > 0$ so that for
	$j \geq j_0$,
	\[
		a_{L_j+N-1} \cdot \discr X_{L_j+N}^{L_j} < -\delta.
	\]
	Therefore, in view of \eqref{eq:45}, \cite[Theorem 3]{PeriodicII} implies that the measure $\mu_{L_j}$, $j \geq j_0$,
	is purely absolutely continuous on $K$ with the density
	\[
		\mu_{L_j}'(x) =
		\frac{\sqrt{-a_{L_j+N-1} \cdot \discr \big(X^{L_j}_{L_j+N}(x)\big)}}{2 \pi g_j(x)}
	\]
	where
	\[
		g_j(x) = a_{L_j+N-1}^{3/2} \big| \scrD^{L_j}_{L_j+N}(x) \big|.
	\]
	Since $\discr X_{L_j}(x)$ is a polynomial of degree at most $2 N$, the convergence in \eqref{eq:38} is uniform with
	respect to $x \in K$. Hence, by Lemma~\ref{lem:2}, we have
	\[
		\lim_{j \to \infty} a_{L_j+N-1} \cdot \discr \big( X^{L_j}_{L_j+N}(x) \big) = h(x)
	\]
	uniformly with respect to $x \in K$. Next, by Lemma \ref{lem:1},
	\[
		\lim_{j \to \infty} g_j(x) = 
		g(x)
	\]
	uniformly with respect to $x \in K$. Since $g$ is positive continuous function on $K$,
	\[
		\lim_{j \to \infty} \mu_{L_j}'(x) = \frac{\sqrt{-h(x)}}{2 \pi g(x)}
	\]
	uniformly with respect to $x \in K$. Now, the theorem follows by \cite[Proposition 4]{SwiderskiTrojan2019}.
\end{proof}

\begin{corollary} 
	\label{cor:1}
	Suppose that the hypotheses of Theorem \ref{thm:4} are satisfied. Then
	\[
		\lim_{j \to \infty} \mu'_{L_j}(x) = \nu'(x)
	\]
	uniformly with respect to $x \in K$.
\end{corollary}

\subsection{Asymptotics of the polynomials}
In this section we study the asymptotic behavior of the orthogonal polynomials $(p_n : n \in \NN_0)$ corresponding to
$N$-periodically modulated Jacobi parameters $(a_n : n \in \NN_0)$ and $(b_n : n \in \NN_0)$.
Let us recall that the polynomials $(p_n : n \in \NN_0)$ satisfy 
\begin{align*}
	p_0(x) &= 1, \qquad p_1(x) = \frac{x- b_0}{a_0}, \\
	x p_n(x) &= a_n p_{n+1}(x) + b_n p_n(x) + a_{n-1} p_{n-1}(x), \qquad n \geq 1.
\end{align*}
In view of \eqref{eq:25}, the Carleman's
condition \eqref{eq:37} is satisfied, thus the measure $\mu$ is the unique orthogonality measure for
$(p_n : n \in \NN_0)$. 
\begin{theorem}
	\label{thm:6}
	Let $N$ be a positive integer and $i \in \{0, 1, \ldots, N-1\}$. Let $(a_n : n \in \NN_0)$ and $(b_n : n \in \NN_0)$
	be $N$-periodically modulated Jacobi parameters such that $\frakX_0(0)$ is a non-trivial parabolic element.
	If	
	\[
		\bigg(\frac{\alpha_{n-1}}{\alpha_n} a_n - a_{n-1} : n \in \NN \bigg),
		\bigg(\frac{\beta_n}{\alpha_n} a_n - b_n : n \in \NN\bigg),
		\bigg(\frac{1}{\sqrt{a_n}} : n \in \NN\bigg) \in \calD_1^N,
	\]
	and
	\[
		\lim_{n \to \infty} \big( a_{n+N} - a_n \big) = 0,
	\]
	then for each compact subset $K \subset \Lambda_-$,
	there are $j_0 \in \NN_0$ and $\chi: K \rightarrow \RR$, so that 
	\begin{equation}
		\label{eq:80}
		\lim_{j \to \infty}
		\sup_{x \in K}{
		\bigg|
		\sqrt[4]{a_{(j+1)N+i-1}} p_{jN+i}(x)
		-
		\sqrt{\frac{\big| [\frakX_i(0)]_{2,1} \big|}{ \pi \mu'(x) \sqrt{\alpha_{i-1} |\tau(x)|}}}
		\sin\Big(\sum_{k=j_0}^{j-1} \theta_k(x) + \chi(x)\Big)\bigg|
		}
		=0.
	\end{equation}
\end{theorem}
\begin{proof}
	Let $K$ be a compact subset of $\Lambda_-$ and set $L_j = j N + i$. By Theorem~\ref{thm:3}, we have
	\begin{equation} 
		\label{eq:144}
		\frac{p_{L_j}(x)}{\prod_{k=j_0}^{j-1} |\lambda_{k}(x)|} =
		\frac{|\vphi(e_2,x)|}{\sqrt{\alpha_{i-1} |\tau(x)|}}
		\sin \Big( \sum_{k=j_0}^{j-1} \theta_k(x) + \arg \vphi(e_2,x) \Big) + o_K(1).
	\end{equation}
	Now, our aim is to identify the value $|\vphi(e_2,x)|$. By \eqref{eq:146},
	\begin{equation}
		\label{eq:151}
		\vphi(e_2,x) = 
		\lim_{j \to \infty} 
		\sqrt{a_{L_j+N-1}} \phi_{L_j}(e_2,x)
	\end{equation}
	where
	\[
		\phi_{L_j}(e_2,x) = 
		\frac{\big\langle \big( X_{L_j}(x) - \overline{\lambda_j(x)} \Id \big) \vec{p}_{L_j}(x), 
		e_2 \big\rangle}{\prod_{k=j_0}^{j-1} \lambda_k(x)}.
	\]
	We introduce the following auxiliary sequence of functions
	\begin{equation}
		\label{eq:147}
		\phi^{L_j}_{m}(x) =
		\frac{\big\langle \big( X_{L_j+N}^{L_j}(x) - \overline{\lambda^{L_j}_{L_j+N}(x)} \Id \big)
		\vec{p}^{L_j}_{L_j+mN}(x), 
		e_2 \big\rangle}{\big( \lambda^{L_j}_{L_j+N}(x) \big)^{m-1} \prod_{k=j_0}^j \lambda_k(x)},
		\qquad x \in K
	\end{equation}
	where $m \in \NN$, and $\lambda^{L_j}_{L_j+N}$ is the eigenvalue of $X^{L_j}_{L_j+N}$ with positive imaginary part.
	Following the same lines of reasoning as \cite[Claim 3]{SwiderskiTrojan2019}, one can show that
	$\phi^{L_j}_m = \phi^{L_j}_1$ for all $m \geq 1$. Next, we observe that
	\[
		\prod_{k=j_0}^{j-1} |\lambda_k(x)|^2 =
		\frac{\det Z_{j_0}}{\det Z_j} \frac{a_{j_0 N + i - 1}}{a_{jN+i-1}}.
	\]
	Since
	\[
		\frac{\det Z_{j_0}}{\det Z_j} = 
		\frac{\sinh \vartheta_{j_0} }{\vartheta_{j_0}} \cdot \frac{\vartheta_j}{\sinh \vartheta_j } 
		\cdot \frac{\vartheta_{j_0}}{\vartheta_j}
	\]
	and
	\[
		\sqrt{a_{(j+1)N+i-1}} \vartheta_j = \sqrt{\alpha_{i-1} |\tau|},
	\]
	we have
	\begin{equation}
		\label{eq:50}
		\lim_{j \to \infty} 
		\sqrt{a_{(j+1)N + i - 1 }} \prod_{k=j_0}^{j} |\lambda_{k}(x)|^2
		=
		\frac{a_{j_0 N + i -1}
		\sinh \vartheta_{j_0}(x) }{\sqrt{\alpha_{i-1} |\tau(x)|}}
	\end{equation}
	uniformly with respect to $x \in K$.
	\begin{claim}
		\label{clm:3}
		\begin{equation} 
			\label{eq:148}
			\lim_{j \to \infty} 
			\sqrt{a_{L_j+N-1}}
			\sup_{x \in K} \big| \phi^{L_j}_1(x) - \phi_{L_j+N}(e_2, x) \big| = 0.
		\end{equation}
	\end{claim}
	For the proof let us observe that
	\[
		\phi^{L_j}_1(x) - \phi_{L_j+N}(e_2, x) =
		\frac{\langle W_j(x) \vec{p}_{L_j+N}(x), e_2 \rangle}{\prod_{k=j_0}^j \lambda_j(x)}
	\]
	where
	\begin{equation}
		\label{eq:150}
		W_j = 
		\Big( X^{L_j}_{L_j+N} - X_{L_j+N} \Big) +
		\Big( \overline{\lambda_{j+1}} - \overline{\lambda^{L_j}_{L_j+N}} \Big) \Id.
	\end{equation}
	Thus,
	\[
		\big| \phi^{L_j}(x) - \phi_{L_j+N}(e_2, x) \big| \leq
		\| W_j(x) \| \frac{\| \vec{p}_{L_j+N}(x) \| }{\prod_{k=j_0}^{j} |\lambda_k(x)|}.
	\]
	By Corollary \ref{cor:3} together with \eqref{eq:50} we obtain
	\[
		\frac{\| \vec{p}_{L_j+N}(x) \| }{\prod_{k=j_0}^{j} |\lambda_k(x)|} \leq c
	\]
	for all $x \in K$ and $j > j_0$. Therefore, in order to prove \eqref{eq:148} it is enough to show that
	\begin{equation} 
		\label{eq:149}
		\lim_{j \to \infty} \sqrt{a_{L_j+N-1}} \sup_{x \in K} \| W_j(x) \| = 0,
	\end{equation}
	which by \eqref{eq:150}, easily follows from
	\begin{align}
		\label{eq:149a}
		&\lim_{j \to \infty} 
		\sqrt{a_{L_j+N-1}} 
		\sup_{x \in K}
		\big\| X^{L_j}_{L_j+N}(x) - X_{L_j+N}(x) \big\| = 0, \\
		\label{eq:149b}
		&\lim_{j \to \infty}
		\sqrt{a_{L_j+N-1}} 
		\sup_{x \in K} 
		\big| \lambda_{j+1}(x) - \lambda^{L_j}_{L_j+N}(x) \big| = 0.
	\end{align}
	To justify \eqref{eq:149a}, we write
	\[
		\big\| X^{L_j}_{L_j+N}(x) - X_{L_j+N}(x) \big\| \leq
		\big\| X^{L_j}_{L_j+N}(x) - X_{L_j}(x) \big\| +
		\big\| X_{L_j}(x) - X_{L_j+N}(x) \big\|,
	\]
	which by Lemmas \ref{lem:2} and \ref{lem:3} implies that
	\begin{equation}
		\label{eq:149a'}
		\lim_{j \to \infty} 
		a_{L_j+N-1}
		\sup_{x \in K}
		\big\| X^{L_j}_{L_j+N}(x) - X_{L_j+N}(x) \big\| = 0.
	\end{equation}
	To prove \eqref{eq:149b}, it is enough to show
	\begin{align}
		\label{eq:51a}
		&\lim_{j \to \infty} 
		\sqrt{a_{L_j+N-1}}
		\sup_K 
		\big| \tr X^{L_j}_{L_j+N} - \tr Y_{j+1} \big| = 0 \\
		\label{eq:51b}
		&\lim_{j \to \infty} 
		a_{L_j+N-1}
		\sup_K 
		\big| \discr X^{L_j}_{L_j+N} - \discr Y_{j+1} \big| = 0.
	\end{align}
	We write
	\begin{equation}
		\label{eq:152}
		X^{L_j}_{L_j+N} - Y_{j+1} =
		\big( X^{L_j}_{L_j+N} - X_{L_j+N} \big) + 
		\big( X_{L_j+N} - Y_{j+1} \big).
	\end{equation}
	By \eqref{eq:35} and Theorem~\ref{thm:2}
	\[
		\tr Y_{j+1} = 
		\tr \big( Z_{j+2}^{-1} Z_{j+1} X_{L_j+N} \big) =
		\tr X_{L_j+N} + \vartheta_{j+1} \cdot \tr(Q_{j+1} X_{L_j+N}).
	\]
	Since $(Q_j)$ uniformly tends to zero, we get
	\begin{equation} 
		\label{eq:51a'}
		\lim_{j \to \infty} 
		\sqrt{a_{L_j+N-1}}
		\sup_K 
		\big| \tr X_{L_j+N} - \tr Y_{j+1} \big| = 0,
	\end{equation}
	which together with \eqref{eq:149a'} leads to \eqref{eq:51a}.
	
	Next, by Theorem~\ref{thm:1}
	\[
		\discr Y_{j+1} = 
		\discr (\varepsilon(\Id + \vartheta_{j+1} R_{j+1})) =
		\vartheta_{j+1}^2 \discr(R_{j+1}).
	\]
	Since $(R_j)$ tends to $\calR$ uniformly and by Corollary \ref{cor:2}, we conclude that 
	\begin{equation}
		\label{eq:52}
		\lim_{j \to \infty} a_{L_j+N-1} \sup_{K} \big| \discr Y_{j+1} - \discr X_{L_j+N} \big| = 0.
	\end{equation}
	Since there is a constant $c>0$ such that for all $A,B \in \Mat(2, \RR)$,
	\begin{equation}
		\label{eq:114}
		|\discr A - \discr B| \leq c \big( \| A \| + \| B \| \big) \| A - B \|,
	\end{equation}
	by \eqref{eq:152} and \eqref{eq:149a'}, we get 
	\[
		\lim_{j \to \infty} a_{L_j+N-1} \sup_{K} \big| \discr X_{L_j+N} - \discr X^{L_j}_{L_j+N} \big| = 0,
	\]
	which together with \eqref{eq:52} implies \eqref{eq:51b}.
	\begin{claim}
		For $x \in K$,
		\begin{equation}
			\label{eq:118}
			|\vphi(e_2,x)|^2
			=
			\frac{1}{a_{j_0 N + i -1} \sinh \vartheta_{j_0}(x)} 
			\cdot
			\frac{| [\frakX_i(0)]_{2,1}| \cdot \alpha_{i-1} |\tau(x)|}{\pi \mu'(x)}.
		\end{equation}
	\end{claim}
	For the proof, by \eqref{eq:151} and Claim~\ref{clm:3} we get
	\[
		|\vphi(e_2,x)|^2 = 
		\lim_{j \to \infty}
		\big| \sqrt{a_{L_j+N-1}} \phi^{L_j}_1(x) \big|^2
	\]
	uniformly with respect to $x \in K$. Next, we observe that by \cite[formula (6.14)]{SwiderskiTrojan2019}
	\begin{align*}
		\bigg| \phi_1^{L_j}(x) \cdot \prod_{k=j_0}^j \lambda_k(x) \bigg|^2 
		&=
		\frac{1}{2 \pi a_{L_j+N-1} \mu'_{L_j}(x)} 
		\Big| \big[ X^{L_j}_{L_j+N}(x) \big]_{2,1} \Big|
		\sqrt{-\discr X^{L_j}_{L_j+N}(x)} \\
		&=
		\frac{1}{2 \pi a_{L_j+N-1}^{3/2} \mu'_{L_j}(x)} 
		\Big| \big[ X^{L_j}_{L_j+N}(x) \big]_{2,1} \Big|
		\sqrt{-a_{L_j+N-1} \discr X^{L_j}_{L_j+N}(x)}.
	\end{align*}
	Lemma~\ref{lem:2} and Corollary~\ref{cor:2} imply
	\[
		\lim_{j \to \infty} a_{L_j+N-1} \discr X^{L_j}_{L_j+N} =
		\lim_{j \to \infty} a_{L_j+N-1} \discr X_{L_j} = 4 \alpha_{i-1} \tau.
	\]
	Hence, by Lemma~\ref{lem:2} and Corollary~\ref{cor:1}
	\[
		\lim_{j \to \infty}
		\sqrt{a_{L_j+N-1}} 
		\bigg| 
		\sqrt{a_{L_j+N-1}} 
		\phi_1^{L_j}(x) \cdot \prod_{k=j_0}^j \lambda_k(x) 
		\bigg|^2 =
		\frac{\big| [\frakX_i(0)]_{2,1} \big|
		\sqrt{\alpha_{i-1} |\tau(x)|}}{\pi \mu'(x)}.
	\]
	Thus, by \eqref{eq:50} the claim follows.

	Now, to finish the proof of the theorem, we put \eqref{eq:118} into \eqref{eq:144} and apply \eqref{eq:50}.
\end{proof}

\section{The Christoffel--Darboux kernel}
\label{sec:8}
In this section we study the convergence of the Christoffel--Darboux kernel defined in \eqref{eq:83} for
$N$-periodically modulated Jacobi parameters $(a_n : n \in \NN_0)$ and $(b_n : n \in \NN_0)$. 

For $i \in \{0, 1, \ldots, N-1 \}$ and $j \in \NN$ we set
\[
	K_{i;j}(x,y) = \sum_{k=0}^j p_{kN+i}(x) p_{kN+i}(y), \qquad x, y \in \RR,
\]
and
\[
	\rho_{i;j} = \sum_{k=1}^j \frac{1}{\sqrt{a_{kN+i}}}.
\]
To describe the limits of $(K_{i; j} : j \in \NN)$, it is useful to define a function
\begin{equation} 
	\label{eq:85}
	\upsilon(x) = 
	\frac{1}{2 \pi N \sqrt{|\tau(x)|}} 
	\sum_{k=0}^{N-1} 
	\frac{|[\frakX_k(0)]_{2,1}|}{\alpha_{k-1}}, 
	\qquad x \in \Lambda_-.
\end{equation}
In view of Propositions \ref{prop:4} and \ref{prop:3}, we have
\begin{equation}
	\label{eq:85a}
	\upsilon(x) = \frac{|\tr \frakX_0'(0)|}{2 \pi N \sqrt{|\tau(x)|}}.
\end{equation}
The following proposition provides yet another way to compute $\upsilon(x)$ for $x \in \Lambda_-$.
\begin{proposition}
	\label{prop:1}
	Let $N$ be a positive integer and $i \in \{0, 1, \ldots N-1\}$. Let $(a_n : n \in \NN_0)$ and $(b_n : n \in \NN_0)$
	be $N$-periodically modulated Jacobi parameters so that $\frakX_0(0)$ is a non-trivial parabolic element. If
	\[
		\bigg(\frac{\alpha_{n-1}}{\alpha_n} a_n - a_{n-1} : n \in \NN \bigg),
		\bigg(\frac{\beta_n}{\alpha_n} a_n - b_n : n \in \NN\bigg),
		\bigg(\frac{1}{\sqrt{a_n}} : n \in \NN\bigg) \in \calD_1^N,
	\]
	then
	\[
		\upsilon(x) = 
		\lim_{j \to \infty}
		\sqrt{\frac{a_{jN+i}}{\alpha_i}} 
		\frac{|\tr X_{jN+i}'(x)|}{\pi N \sqrt{-\discr X_{jN+i}(x)}}, \qquad x \in \Lambda_-.
	\]
\end{proposition}
\begin{proof}
	Let us observe that
	\[
		\sqrt{\frac{a_{jN+i}}{\alpha_i}} 
		\frac{|\tr X_{jN+i}'(x)|}{\pi N \sqrt{-\discr X_{jN+i}(x)}} =
		\frac{\tfrac{a_{jN+i}}{\alpha_i} 
		|\tr X_{jN+i}'(x)|}{\pi N \sqrt{\tfrac{a_{jN+i}}{\alpha_i}} 
		\sqrt{-\discr X_{jN+i}(x)}}.
	\]
	By Corollary \ref{cor:2}, we have
	\[
		\lim_{j \to \infty}
		\sqrt{\frac{a_{jN+i}}{\alpha_i}} \sqrt{-\discr X_{jN+i}(x)}
		=
		2 \sqrt{|\tau(x)|}.
	\]
	In view of \cite[Corollary 3.10]{ChristoffelI},
	\[
		\lim_{j \to \infty} 
		\frac{a_{jN+i}}{\alpha_i} 
		\big| \tr X_{jN+i}'(x) \big| = 
		\big| \tr \frakX_0'(0) \big|,
	\]
	thus
	\[
		\lim_{j \to \infty}
		\sqrt{\frac{a_{jN+i}}{\alpha_i}}
		\frac{|\tr X_{jN+i}'(x)|}{\pi N \sqrt{-\discr X_{jN+i}(x)}}
		=
		\frac{|\tr \frakX_0'(0)|}{2 \pi N \sqrt{|\tau(x)|}},
	\]
	which together with \eqref{eq:85a} completes the proof.
\end{proof}

\begin{proposition} 
	\label{prop:2}
	Let $N$ be a positive integer and $i \in \{0, 1, \ldots, N-1\}$. Let $(a_n : n \in \NN_0)$ and $(b_n : n \in \NN_0)$
	be $N$-periodically modulated Jacobi parameters so that $\frakX_0(0)$ is a non-trivial parabolic element. If
	\[
		\bigg(\frac{\alpha_{n-1}}{\alpha_n} a_n - a_{n-1} : n \in \NN \bigg),
		\bigg(\frac{\beta_n}{\alpha_n} a_n - b_n : n \in \NN\bigg),
		\bigg(\frac{1}{\sqrt{a_n}} : n \in \NN\bigg) \in \calD_1^N,
	\]
	and
	\[
		\lim_{n \to \infty} \big(a_{n+N}-a_n\big) = 0,
	\]
	then
	\begin{align}
		\label{eq:86a}
		&\lim_{j \to \infty} 
		\sqrt{\frac{a_{(j+1)N+i-1}}{\alpha_{i-1}}} \theta_j(x) =
		\sqrt{|\tau(x)|}, \\
		\label{eq:86b}
		&\lim_{j \to \infty} 
		\sqrt{\frac{a_{(j+1)N+i-1}}{\alpha_{i-1}}} |\theta_j'(x)| =
		-N \pi \upsilon(x), \\
		\label{eq:86c}
		&\lim_{j \to \infty}
		\sqrt{\frac{a_{(j+1)N+i-1}}{\alpha_{i-1}}} |\theta_j''(x)|
		=
		\frac{(N \pi \upsilon(x))^2}{2 \sqrt{|\tau(x)|}}
	\end{align}
	locally uniformly with respect to $x \in \Lambda_-$.
\end{proposition}
\begin{proof}
	Let us begin with \eqref{eq:86a}. By Theorem~\ref{thm:1},
	\[
		\lim_{j \to \infty} Y_j = \varepsilon \Id
	\]
	locally uniformly on $\Lambda_-$. In particular,
	\[
		\lim_{j \to \infty} \frac{\tr Y_j(x)}{2 \sqrt{\det Y_j(x)}} = \varepsilon.
	\]
	Since
	\[
		\lim_{t \to 1^-} \frac{\arccos t}{\sqrt{1-t^2}} =1,
	\]
	we obtain
	\[
		\lim_{j \to \infty} \bigg(1-\bigg( \frac{\tr Y_j(x)}{2 \sqrt{\det Y_j(x)}} \bigg)^2  \bigg)^{-1/2} \theta_j(x) 
		= 1.
	\]
	Let us observe that by Theorem~\ref{thm:1}
	\[
		\sqrt{1-\bigg( \frac{\tr Y_j(x)}{2 \sqrt{\det Y_j(x)}} \bigg)^2} = 
		\frac{\sqrt{-\discr Y_j(x)}}{2 \sqrt{\det Y_j(x)}} =
		\vartheta_j(x) \frac{\sqrt{-\discr R_j(x)}}{2 \sqrt{\det Y_j(x)}}.
	\]
	Hence, by \eqref{eq:68}
	\begin{equation}
		\label{eq:105}
		\lim_{j \to \infty} \sqrt{\frac{a_{(j+1)N+i-1}}{\alpha_{i-1}}} 
		\sqrt{1-\bigg( \frac{\tr Y_j(x)}{2 \sqrt{\det Y_j(x)}} \bigg)^2} = \sqrt{|\tau(x)|},
	\end{equation}
	and the proof of \eqref{eq:86a} is complete.

	Next, by the direct computation, we obtain
	\begin{equation}
		\label{eq:102}
		\tr Y_j = \tr X_{jN+i} + \big[T_i^{-1} X_{jN+i} T_i \big]_{1, 2} 
		\frac{\sinh(\vartheta_{j+1} - \vartheta_j)}{\sinh \vartheta_{j+1}}
		+
		\big[T_i^{-1} X_{jN+i} T_i\big]_{2, 2} \bigg(\frac{\sinh \vartheta_j}{\sinh \vartheta_{j+1}} - 1\bigg).
	\end{equation}
	We write
	\[
		\frac{\sinh(\vartheta_{j+1} - \vartheta_j)}{\sinh \vartheta_{j+1}}
		=
		\frac{\vartheta_{j+1} - \vartheta_j}{\vartheta_{j+1}} 
		\cdot 
		\frac{\sinh(\vartheta_{j+1} - \vartheta_j)}{\vartheta_{j+1} - \vartheta_j}
		\cdot
		\frac{\vartheta_{j+1}}{\sinh \vartheta_{j+1}}.
	\]
	Notice that
	\begin{align*}
		\frac{\vartheta_{j+1} - \vartheta_j}{\vartheta_{j+1}}
		&=
		\sqrt{a_{(j+2)N+i-1}}\bigg(\frac{1}{\sqrt{a_{(j+2)N+i-1}}} - \frac{1}{\sqrt{a_{(j+1)N+i-1}}}\bigg)\\
		&=
		\frac{1}{\sqrt{a_{(j+1)N+i-1}}} \big(\sqrt{a_{(j+1)N+i-1}} - \sqrt{a_{(j+2)N+i-1}}\big) \\
		&=
		\frac{1}{\sqrt{a_{(j+1)N+i-1}}} \frac{a_{(j+1)N+i-1} - a_{(j+2)N+i-1}}
		{\sqrt{a_{(j+1)N+i-1}} + \sqrt{a_{(j+2)N+i-1}}}, 
	\end{align*}
	thus
	\[
		\lim_{j \to \infty} a_{(j+1)N+i-1} \frac{\vartheta_{j+1} - \vartheta_j}{\vartheta_{j+1}} = 0.
	\]
	Since the function
	\[
		F(x) = \frac{\sinh \sqrt{x}}{\sqrt{x}}, \quad x > 0,
	\]
	has smooth extension to $\RR$ which attains $1$ at the origin, for $k \in \{0, 1, 2\}$ we have
	\begin{equation}
		\label{eq:101}
		\begin{aligned}
		&\lim_{j \to \infty}
		a_{(j+1)N+i-1} \bigg(\frac{\sinh(\vartheta_{j+1} - \vartheta_j)}{\sinh \vartheta_{j+1}}\bigg)^{(k)} \\
		&\qquad\qquad=
		\lim_{j \to \infty}
		a_{(j+1)N+i-1}
		\frac{\vartheta_{j+1} - \vartheta_j}{\vartheta_{j+1}}
		\bigg(\frac{\sinh(\vartheta_{j+1} - \vartheta_j)}{\vartheta_{j+1} - \vartheta_j} \cdot 
		\frac{\vartheta_{j+1}}{\sinh \vartheta_{j+1}}\bigg)^{(k)}
		=0
		\end{aligned}
	\end{equation}
	locally uniformly on $\Lambda_-$. Similarly, we write
	\[
		\frac{\sinh \vartheta_{j+1}}{\sinh \vartheta_{j}} - 1 = 
		\frac{\cosh\left(\frac{\vartheta_{j+1} + \vartheta_j}{2}\right)}
		{\cosh\left(\frac{\vartheta_{j+1} - \vartheta_j}{2}\right)}
		\frac{\sinh(\vartheta_{j+1} - \vartheta_j)}{\sinh \vartheta_j},
	\]
	and observe that $G(x) = \cosh \sqrt{x}$, $x > 0$, has a smooth extension to $\RR$ and attains $1$ at the
	origin, hence for $k \in \{0, 1, 2\}$,
	\begin{equation}
		\label{eq:99}
		\lim_{j \to \infty} a_{(j+1)N+i-1} \bigg(\frac{\sinh \vartheta_{j+1}}{\sinh \vartheta_{j}} - 1\bigg)^{(k)}
		=0
	\end{equation}
	locally uniformly on $\Lambda_-$. In particular, by \eqref{eq:102} and \cite[Corollary 3.10]{ChristoffelI}
	\begin{equation}
		\label{eq:104}
		\lim_{j \to \infty} \frac{a_{(j+1)N+i-1}}{\alpha_{i-1}} \tr Y_j'(x) = \tr \frakX_0'(0)
	\end{equation}
	locally uniformly with respect to $x \in \Lambda_-$. Now, to prove \eqref{eq:86b}, we write
	\[
		\frac{\tr Y_j}{2 \sqrt{\det Y_j}} = \frac{1}{2} \sqrt{\frac{a_{(j+1)N+i-1}}{a_{jN+i-1}}}
	    \sqrt{\frac{\sinh \vartheta_{j+1}}{\sinh \vartheta_{j}}} \tr Y_j,
	\]
	and so
	\[
		\theta'_j = -\frac{1}{2} 
		\sqrt{\frac{a_{(j+1)N+i-1}}{a_{jN+i-1}}}
		\bigg(1 - \bigg(\frac{\tr Y_j}{2 \sqrt{\det Y_j}}\bigg)^2\bigg)^{-\frac{1}{2}}
		\bigg(\sqrt{\frac{\sinh \vartheta_{j+1}}{\sinh \vartheta_{j}}} \tr Y_j\bigg)'.
	\]
	By \eqref{eq:99} and \eqref{eq:104},
	\begin{equation}
		\label{eq:107}
		\lim_{j \to \infty} 
		\frac{a_{(j+1)N+i-1}}{\alpha_{i-1}}
		\bigg(\sqrt{\frac{\sinh \vartheta_{j+1}}{\sinh \vartheta_{j}}} \tr Y_j \bigg)' = 
		\tr \frakX_0'(0)
	\end{equation}
	which together with \eqref{eq:105} gives
	\[
		\lim_{j \to \infty}
		\sqrt{\frac{a_{(j+1)N+i-1}}{\alpha_{i-1}}}
		\big|\theta'_j(x)\big|
		=
		\frac{\big|\tr \frakX_0'(0)\big|}{2 \sqrt{|\tau(x)|}}
	\]
	locally uniformly with respect to $x \in \Lambda_-$ proving \eqref{eq:86b}.

	Finally, we turn to the proof of \eqref{eq:86c}. We have
	\begin{align*}
		\theta''_{j} 
		&= -\frac{1}{8} 
		\bigg(\frac{a_{(j+1)N+i-1}}{a_{jN+i-1}}\bigg)^{\frac{3}{2}}
		\bigg(1 - \bigg(\frac{\tr Y_j}{2 \sqrt{\det Y_j}}\bigg)^2\bigg)^{-\frac{3}{2}}
		\bigg(\sqrt{\frac{\sinh \vartheta_{j+1}}{\sinh \vartheta_{j}}} \tr Y_j\bigg)
		\bigg\{
		\bigg(\sqrt{\frac{\sinh \vartheta_{j+1}}{\sinh \vartheta_{j}}} \tr Y_j\bigg)'\bigg\}^2 \\
		&\phantom{=}
		-\frac{1}{2}
		\sqrt{\frac{a_{(j+1)N+i-1}}{a_{jN+i-1}}}
		\bigg(1 - \bigg(\frac{\tr Y_j}{2 \sqrt{\det Y_j}}\bigg)^2\bigg)^{-\frac{1}{2}}
		\bigg(\sqrt{\frac{\sinh \vartheta_{j+1}}{\sinh \vartheta_{j}}} \tr Y_j\bigg)''.
	\end{align*}
	By \cite[Corollary 3.10]{ChristoffelI} together with \eqref{eq:102}, \eqref{eq:101} and \eqref{eq:99},
	\begin{equation}
		\label{eq:106}
		\lim_{j \to \infty} a_{jN+i} \tr Y_j'' = 0,
	\end{equation}
	thus
	\[
		\lim_{j \to \infty}
		a_{jN+i} \bigg(\sqrt{\frac{\sinh \vartheta_{j+1}}{\sinh \vartheta_{j}}} \tr Y_j\bigg)'' = 0,
	\]
	locally uniformly on $\Lambda_-$, which together with \eqref{eq:105} and \eqref{eq:107} implies \eqref{eq:86c}.
\end{proof}

\subsection{Universality limits}

\begin{theorem} 
	\label{thm:7}
	Let $N$ be a positive integer. Suppose that $(a_n : n \in \NN_0)$ and $(b_n : n \in \NN_0)$ are $N$-periodically
	modulated Jacobi parameters so that $\frakX_0(0)$ is a non-trivial parabolic element. If
	\[
		\bigg(\frac{\alpha_{n-1}}{\alpha_n} a_n - a_{n-1} : n \in \NN \bigg),
		\bigg(\frac{\beta_n}{\alpha_n} a_n - b_n : n \in \NN\bigg),
		\bigg(\frac{1}{\sqrt{a_n}} : n \in \NN\bigg) \in \calD_1^N,
	\]
	and
	\[
		\lim_{n \to \infty} \big(a_{n+N}-a_n\big) = 0,
	\]
	then
	\[
		\lim_{n \to \infty} 
		\frac{1}{\rho_{n}} K_{n} \bigg(x + \frac{u}{\rho_n}, x + \frac{v}{\rho_n} \bigg) =
		\frac{\upsilon(x)}{\mu'(x)}
		\cdot
		\sinc\big((u-v) \pi \upsilon(x)\big)
	\]
	locally uniformly with respect to $(x, u, v) \in \Lambda_- \times \RR^2$, 
	where $\upsilon$ is defined in \eqref{eq:85} and
	\[
		\rho_n = \sum_{k=0}^n \sqrt{\frac{\alpha_k}{a_k}}.
	\]
\end{theorem}
\begin{proof}
	Let $K$ be a compact interval with non-empty interior contained in $\Lambda_-$, and let $L > 0$. We select a
	compact interval $\tilde{K} \subset \Lambda_-$ containing $K$ in its interior. There is $j_1 > 0$ such that
	for all $x \in K$, $j \geq j_1$, $i \in \{0, 1, \ldots, N-1\}$, and $u \in [-L, L]$,
	\[
		x + \frac{u}{\rho_{j N + i}}, x + \frac{u}{N \alpha_i \rho_{i; j}} \in \tilde{K}.
	\]
	Given $x \in K$ and $u, v \in [-L, L]$, we set
	\begin{align*}
		x_{i; j} &= x + \frac{u}{N \sqrt{\alpha_i} \rho_{i; j}}, \qquad x_{jN+i} = x + \frac{u}{\rho_{jN+i}}, \\
		y_{i; j} &= x + \frac{v}{N \sqrt{\alpha_i} \rho_{i; j}}, \qquad y_{jN+i} = x + \frac{v}{\rho_{jN+i}}.
	\end{align*}
	By Theorem~\ref{thm:6}, there is $j_0 \geq j_1$ such that for all $x, y \in K$, and $k > j_0$,
	\[
		\begin{aligned}
		&\sqrt{a_{(k+1)N+i-1}} p_{kN+i}(x) p_{kN+i}(y) \\
		&\qquad\qquad= \frac{1}{\pi}
		\sqrt{\frac{|[\frakX_i(0)]_{2,1}|}{\mu'(x) \sqrt{\alpha_{i-1} |\tau(x)|}}}
		\sqrt{\frac{|[\frakX_i(0)]_{2,1}|}{\mu'(y) \sqrt{\alpha_{i-1} |\tau(y)|}}}\\
		&\qquad\qquad\phantom{=}\times
			\sin\Big(\sum_{\ell=j_0}^{k-1} \theta_{\ell N + i}(x) + \chi_i(x) \Big)
			\sin\Big(\sum_{\ell=j_0}^{k-1} \theta_{\ell N + i}(y) + \chi_i(y) \Big)
		+
		E_{kN+i}(x, y)
		\end{aligned}
	\]
	where
	\[
		\lim_{k \to \infty}
		\sup_{x, y \in K} |E_{kN+i}(x, y)| = 0.
	\]
	Therefore, we obtain
	\[
		\begin{aligned}
		&
		\sum_{k = j_0+1}^j
		p_{kN+i}(x) p_{kN+i}(y)
		=
		\frac{1}{\pi}
		\frac{|[\frakX_i(0)]_{2,1}|}{\alpha_{i-1}}
		\frac{1}{\sqrt{\mu'(x) \mu'(y) \sqrt{|\tau(x) \tau(y)|}}} \\
		&\qquad\qquad\phantom{=}\times
		\sum_{k = j_0+1}^j
		\sqrt{\frac{\alpha_{i-1}}{a_{(k+1)N+i-1}}}
		\sin\Big(\sum_{\ell=j_0}^{k-1} \theta_{\ell N + i}(x) + \chi_i(x) \Big)
		\sin\Big(\sum_{\ell=j_0}^{k-1} \theta_{\ell N + i}(y) + \chi_i(y) \Big) \\
		&\qquad\qquad\phantom{=}+
		\sum_{k = j_0+1}^j
		\frac{1}{\sqrt{a_{(k+1)N+i-1}}} E_{kN+i}(x, y).
		\end{aligned}
	\]
	Observe that by the Stolz--Ces\'aro theorem,
	\[
		\lim_{j \to \infty} \frac{1}{\rho_{i-1; j}} \sum_{k = j_0+1}^j \frac{1}{\sqrt{a_{(k+1)N+i-1}}} E_{kN+i}(x, y)
		=
		\lim_{j \to \infty} \sqrt{\frac{a_{jN+i-1}}{a_{(j+1)N+i-1}}} E_{jN+i}(x, y) = 0.
	\]
	In view of Proposition \ref{prop:2}, we can apply \cite[Theorem 9]{SwiderskiTrojan2019} with
	\[
		\xi_j(x) = \theta_{jN+i}(x),
		\qquad
		\gamma_j = N \sqrt{\frac{\alpha_{i-1}}{a_{(j+1)N+i-1}}},
		\qquad\text{and}\qquad
		|\psi(x)| = \pi \upsilon(x).
	\]
	Therefore, for any $i' \in \{0, 1, \ldots, N-1\}$, as $j$ tends to infinity
	\begin{align*}
		\frac{1}{N \sqrt{\alpha_{i-1}} \rho_{i-1; j}}
		\sum_{k = j_0+1}^j
		N\sqrt{\frac{\alpha_{i-1}}{a_{(k+1)N+i-1}}}
		&\sin\Big(\sum_{\ell=j_0}^{k-1} \theta_{\ell N + i}(x_{j N+i'}) + \chi_i(x_{j N+i'}) \Big) \\
		&\times
		\sin\Big(\sum_{\ell=j_0}^{k-1} \theta_{\ell N + i}(y_{j N+i'}) + \chi_i(y_{j N+i'}) \Big)
	\end{align*}
	approaches to
	\[
		\frac{1}{2} \sinc\big((v-u) \pi \upsilon(x) \big)
	\]
	uniformly with respect to $x \in K$ and $u, v \in [-L, L]$. Moreover,
	\begin{align*}
		\lim_{j \to \infty}
		\frac{1}{\mu'(x_{jN+i'}) \sqrt{|\tau(x_{jN+i'}|)}}
		&=
		\lim_{j \to \infty}
		\frac{1}{\mu'(y_{jN+i'}) \sqrt{|\tau(y_{jN+i'}|)}} \\
		&=
		\frac{1}{\mu'(x) \sqrt{|\tau(x)|}}.
	\end{align*}
	Hence,
	\begin{equation}
		\label{eq:75}
		\lim_{j \to \infty} 
		\frac{1}{\rho_{i-1; j}}
		K_{i; j}(x_{jN+i'}, y_{jN+i'}) 
		=
		\sinc\big((v-u)\pi \upsilon(x) \big)
		\frac{1}{2 \pi \mu'(x) \sqrt{|\tau(x)|}} 
        \frac{|[\frakX_i(0)]_{2,1}|}{\sqrt{\alpha_{i-1}}}.
	\end{equation}
	Finally, we write
	\[
		K_{jN+i'}(x, y) = \sum_{i = 0}^{N-1} K_{i; j}(x, y) 
		+ \sum_{i = i'+1}^{N-1} \big(K_{i; j-1}(x, y) - K_{i; j}(x, y)\big).
	\]
	Observe that
	\[
		\sup_{x, y \in K}{\big|K_{i; j-1}(x, y) - K_{i; j}(x, y)\big|} =
		\sup_{x, y \in K}{|p_{jN+i}(x) p_{jN+i}(y)|} \leq c.
	\]
	Moreover, by \cite[Proposition 3.7]{ChristoffelI}, for $m, m' \in \NN_0$,
	\[
		\lim_{j \to \infty} \frac{a_{jN+m'}}{a_{jN+m}} = \frac{\alpha_{m'}}{\alpha_m},
	\]
	thus, by the Stolz--Ces\'aro theorem, 
	\begin{align*}
		\lim_{j \to \infty} \frac{\rho_{i-1; j}}{\rho_{jN+i'}} 
		&=
		\lim_{j \to \infty} \frac{\frac{1}{\sqrt{a_{jN+i-1}}}}{\sum_{k = 1}^N \sqrt{\frac{\alpha_{i'+k}}{a_{jN+i'+k}}}} \\
		&=
		\frac{1}{N \sqrt{\alpha_{i-1}}}.
	\end{align*}
	Hence, by \eqref{eq:75} 
	\begin{align*}
		\lim_{j \to \infty} \frac{1}{\rho_{jN+i'}} K_{jN+i'}(x_{jN+i'}, y_{jN+i'})
		&=
		\lim_{j \to \infty} \sum_{i = 0}^{N-1} \frac{1}{\rho_{i-1; j}} K_{jN+i}(x_{jN+i'} , y_{jN+i'}) \cdot 
		\frac{\rho_{i-1; j}}{\rho_{jN+i'}} \\
		&=
		\frac{1}{\mu'(x)}
		\sinc\big((v-u)\pi \upsilon(x) \big)
		\frac{1}{2 N \pi \sqrt{|\tau(x)|}}
		\sum_{i = 0}^{N-1}
		\frac{|[\frakX_i(0)]_{2,1}|}{\alpha_{i-1}}.
	\end{align*}
	Therefore, in view of \eqref{eq:85}, we obtain
	\[
		\lim_{j \to \infty} \frac{1}{\rho_{jN+i'}} K_{jN+i'}(x_{jN+i'}, y_{jN+i'})
		=
		\frac{\upsilon(x)}{\mu'(x)} \cdot \sinc\big((v-u) \pi \upsilon(x) \big),
	\]
	and the theorem follows.
\end{proof}

\subsection{Applications to Ignjatovi\'c's conjecture}
In the following theorem we extend the results from \cite[Section 4.3]{ChristoffelI} and 
\cite[Section 8.1]{ChristoffelII} to the case when $N=1$ and $\frakX_0(0)$ is a non-trivial parabolic element
of $\SL(2, \RR)$. These results are motivated by \cite[Conjecture 1]{Ignjatovic2016}.
\begin{theorem}
	Let $q \in \{-2, 2\}$. Suppose that
	\[
		\big( a_n - a_{n-1} : n \in \NN \big),
		\big(q a_n - b_n : n \in \NN\big),
		\bigg(\frac{1}{\sqrt{a_n}} : n \in \NN\bigg) \in \calD_1
	\]
	and
	\[
		\lim_{n \to \infty} \big( a_n - a_{n-1} \big) = 0, \qquad
		\lim_{n \to \infty} \big( q a_n - b_n \big) = r, \qquad
		\lim_{n \to \infty} a_n = \infty.
	\]
	Then
	\[
		\lim_{n \to \infty} 
		\bigg( \sum_{j=0}^n \frac{1}{\sqrt{a_j}} \bigg)^{-1}
		\sum_{j=0}^n p^2_j(x) = \frac{1}{\pi \mu'(x) \sqrt{|x+r|}}
	\]
	locally uniformly with respect to $x \in \Lambda_-$ where
	\[
		\Lambda_- = 
		\begin{cases}
			(-r, +\infty) & q=2,\\
			(-\infty, -r) & q=-2.
		\end{cases}
	\]
\end{theorem}
\begin{proof}
	Let $N = 1$, $\alpha_n \equiv 1$, and $\beta_n \equiv q$. Then
	\[
		\frakX_0(0) =
		\begin{pmatrix}
			0 & 1 \\
			-1 & -q
		\end{pmatrix}.
	\]
	By \eqref{eq:61b} and \eqref{eq:25},
	\[	
		\tau(x) = (x+r) \sign{-q} = -(x+r) \sign{q}.
	\]
	Hence, the result follows by Theorem~\ref{thm:7}.
\end{proof}

\section{The $\ell^1$-type perturbations}
\label{sec:9}
In this section we show how to obtain the main results of the paper in the presence of $\ell^1$ perturbation.
We start by introducing notation. Let $(\tilde{a}_n : n \in \NN_0)$ and $(\tilde{b}_n : n \in \NN_0)$ be Jacobi parameters
satisfying
\[
	\tilde{a}_n = a_n \big(1 + \xi_n\big), \qquad \tilde{b}_n = b_n \big(1 + \zeta_n\big)
\]
where $(a_n : n \in \NN_0)$ and $(b_n : n \in \NN_0)$ are $N$-periodically modulated Jacobi parameters so that
$\frakX_0(0)$ is a non-trivial parabolic element, satisfying
\[
	\bigg(\frac{\alpha_{n-1}}{\alpha_n} a_n - a_{n-1} : n \in \NN \bigg),
	\bigg(\frac{\beta_n}{\alpha_n} a_n - b_n : n \in \NN\bigg),
	\bigg(\frac{1}{\sqrt{a_n}} : n \in \NN\bigg) \in \calD_1^N,
\]
and 
\[
	\sum_{n = 0}^\infty \sqrt{a_n} (|\xi_n| + |\zeta_n|)< \infty,
\]
for certain real sequences $(\xi_n : n \in \NN_0)$ and $(\zeta_n : n \in \NN_0)$. In this section we add tilde to
objects defined in terms of Jacobi parameters $(\tilde{a}_n)$ and $(\tilde{b}_n)$. 

Let $K$ be a compact subset of $\RR$. By $(\Delta_n)$ we denote any sequence of $2\times2$ matrices such that
\[
	\sum_{n = 0}^\infty \sup_K \|\Delta_n\| < \infty.
\]
We notice that
\begin{equation}
	\label{eq:117a}
	\tilde{B}_n(x) = B_n(x) + a_n^{-1/2} \Delta_n(x)
\end{equation}
where
\begin{align*}
	\tilde{B}_0(x) 
	&= 
	\begin{pmatrix}
		0 & 1 \\
		-\frac{1}{\tilde{a}_0} & \frac{x-\tilde{b}_0}{\tilde{a}_0}
	\end{pmatrix}, \\
	\tilde{B}_n(x) &= 
	\begin{pmatrix}
		0 & 1 \\
		-\frac{\tilde{a}_{n-1}}{\tilde{a}_n} & \frac{x-\tilde{b}_n}{\tilde{a}_n}
	\end{pmatrix}, \quad n \geq 1.
\end{align*}
Moreover, for
\[
	\tilde{X}_n = \tilde{B}_{n+N-1} \tilde{B}_{n+N-2} \cdots \tilde{B}_n,
\]
we have
\begin{align*}
	\tilde{X}_n - X_n
	=
	\sum_{k = n}^{n+N-1} a_k^{-1/2} \tilde{B}_{n+N-1} \cdots \tilde{B}_{k+1} \Delta_k
	B_{k-1} \cdots B_n,
\end{align*}
which together with
\[
	\sup_{n \in \NN_0}{ \sup_{x \in K}{\big(\|B_n(x)\| + \|\tilde{B}_n(x)\|}\big)} < \infty,
\]
implies that
\begin{equation}
	\label{eq:117}
	\tilde{X}_n = X_n + a_n^{-1/2} \Delta_n.
\end{equation}
Next, if $K \subset \Lambda_-$ then by Theorem \ref{thm:3} and \eqref{eq:50}, there is $c > 0$ such that for all $n \in \NN_0$,
\begin{equation}
	\label{eq:110}
	\sup_{K}{\|B_n B_{n-1} \cdots B_0\|} \leq c a_n^{-1/4},
\end{equation}
and since $\det B_n = \frac{a_{n-1}}{a_n}$, we get
\begin{equation}
	\label{eq:103}
	\sup_{K}{\|(B_n B_{n-1} \cdots B_0)^{-1} \|} \leq c a_n^{3/4}.
\end{equation}
Moreover, by \eqref{eq:117a}
\begin{align*}
	\tilde{B}_n \cdots \tilde{B}_1 \tilde{B}_0
	&=
	\tilde{B}_n \cdots \tilde{B}_1 B_0 \Big(\Id + a_0^{-1/2} B_0^{-1} \Delta_0\Big) \\
	&=
	\tilde{B}_n \cdots \tilde{B}_2 B_1 B_0 \Big(\Id + a_1^{-1/2} (B_1 B_0)^{-1} \Delta_1 B_0 \Big) 
	\Big(\Id + a_0^{-1/2} B_0^{-1} \Delta_0\Big) \\
	&=
	B_n \cdots B_1 B_0
	\prod_{j = 0}^n
	\Big(\Id + a_j^{-1/2} (B_j \cdots B_1 B_0)^{-1} \Delta_j (B_{j-1} \cdots B_1 B_0) \Big)
\end{align*}
thus by \eqref{eq:110} and \eqref{eq:103}
\begin{align*}
	\|\tilde{B}_n \cdots \tilde{B}_1 \tilde{B}_0\|
	&\leq
	\|B_n \cdots B_1 B_0\|
	\prod_{j = 0}^{n} \Big(1 + a_j^{-1/2} 
	\|(B_j \cdots B_1 B_0)^{-1} \| \cdot \|B_{j-1} \cdots B_1 B_0\| \cdot \|\Delta_j\|\Big) \\
	&\leq
	\|B_n \cdots B_1 B_0\|
	\prod_{j=0}^{n} \Big(1 + c a_j^{1/4} a_{j-1}^{-1/4} \|\Delta_j\|\Big) \\
	&\leq
	\|B_n \cdots B_1 B_0\|
	\exp\Big(c \sum_{j = 0}^n \|\Delta_j\|\Big),
\end{align*}
and so
\begin{equation}
	\label{eq:98}
	\sup_{K}{\|\tilde{B}_n \cdots \tilde{B}_1 \tilde{B}_0\|}
	\leq c a_n^{-1/4}.
\end{equation}
Next, let us introduce the following sequence of matrices
\begin{equation}
	\label{eq:115}
	M_j = \big(B_j B_{j-1} \cdots B_0 \big)^{-1} \big(\tilde{B}_j \tilde{B}_{j-1} \cdots \tilde{B}_0\big).
\end{equation}
Since
\[
	M_{j+1} - M_j = \big(B_{j+1} B_{j} \cdots B_0 \big)^{-1} \big(\tilde{B}_{j+1} - B_{j+1}\big)
	\big(\tilde{B}_j \tilde{B}_{j-1} \cdots \tilde{B}_0\big),
\]
by \eqref{eq:117a}, \eqref{eq:103} and \eqref{eq:98}, we obtain
\begin{align*}
	\sup_K{ \|M_{j+1} - M_j\| }
	&\leq
	c a_{j+1}^{3/4} a_{j+1}^{-1/2} a_j^{-1/4}
	\sup_K{\|\Delta_{j+1}\|} \\
	&\leq
	c \sup_K{\|\Delta_{j+1}\|}.
\end{align*}
Hence, the sequence of matrices $(M_j)$ converges uniformly on $K \subset \Lambda_-$ to a certain continuous mapping 
$M$, and
\begin{equation}
	\label{eq:112}
	\sup_K{ \big\|M - M_j \big\| } \leq c \sum_{k = j+1}^\infty \sup_K{\|\Delta_k\|}.
\end{equation}
Observe that for each $x \in K$ the matrix $M(x)$ is non-degenerate. Indeed, we have
\begin{align*}
	\det M(x) 
	&= \lim_{j \to \infty} \det M_j(x) \\
	&= \lim_{j \to \infty} \frac{a_j}{\tilde{a}_j} = 1.
\end{align*}
We set
\[
	\eta_n = \frac{M_{n-1} e_2}{\|M_{n-1} e_2\|}, \qquad 
	\eta = \frac{M e_2}{\|M e_2\|}.
\]
Let us denote by $(\tilde{p}_n : n \in \NN_0)$ orthogonal polynomials generated by $(\tilde{a}_n : n \in \NN_0)$,
and $(\tilde{b}_n : n \in \NN_0)$. Let $\tilde{\mu}$ denote their orthonormalizing measure. Notice that for all $n \in \NN$ and $x \in K$, by \eqref{eq:108a} and \eqref{eq:115}, we have
\begin{equation}
	\label{eq:63}
	\vec{u}_n\big(\eta_n(x), x\big) = 
	\frac{1}{\|M_{n-1}(x) e_2\|}
	\begin{pmatrix}
		\tilde{p}_{n-1}(x) \\
		\tilde{p}_{n}(x)
	\end{pmatrix}.
\end{equation}
By Corollary \ref{cor:3}, 
\[
	\sup_{n \in \NN} \sup_{x \in K} 
	\sqrt{a_{n+N-1}} \big\| \vec{u}_{n} \big( \eta_n(x),x \big) \big\|^2 < \infty,
\]
which together with \eqref{eq:63} implies
\begin{equation}
	\label{eq:119}
	\sup_{n \in \NN}\sup_{x \in K}{
	\sqrt{\tilde{a}_{n+N-1}} \big(\tilde{p}_{n-1}^2(x) + \tilde{p}_{n}^2(x) \big)} < \infty.
\end{equation}
We consider the corresponding $N$-shifted Tur\'an determinants,
\[
	\tilde{\scrD}_n(x) = \tilde{p}_n(x) \tilde{p}_{n+N-1} - \tilde{p}_{n-1}(x) \tilde{p}_{n+N}(x).
\]
By \eqref{eq:63} together with Theorem \ref{thm:3}, \eqref{eq:116} and \eqref{eq:117}, we obtain
\begin{align*}
	\big| \tilde{\scrD}_n(x) - a_{n+N-1}^{-3/2} \|M_{n-1}(x) e_2\|^{2} \cdot S_n(\eta_n(x), x) \big| 
	&=
	\big| \tilde{\scrD}_n(x) - \|M_{n-1}(x) e_2\|^{2}
	\sprod{E X_n(x) \vec{u}_n(\eta_n(x), x) }{\vec{u}_n(\eta_n(x), x) } \big| \\
	&\leq
	c a_{n+N-1}^{-1/2} \|X_n(x) - \tilde{X}_n(x)\| \\
	&\leq
	c a_{n+N-1}^{-1} \sup_K{\| \Delta_n \|}.
\end{align*}
Fix $i \in \{0, 1, \ldots, N-1\}$. Since $(a_n)$ is sublinear and $(\sup_K \|\Delta_n\| )$ belongs to $\ell^1$,
for each subsequence there is a further subsequence $(L_j : j \in \NN_0)$, such that
\begin{equation}
	\label{eq:71}
	\sup_K{\|\Delta_{L_j} \|} 
	\leq c a_{L_j+N-1}^{-1},
\end{equation}
thus we can guarantee that $L_j \equiv i \bmod N$. Moreover, if
\[
	\lim_{n \to \infty} (a_{n+N} - a_n) = 0,
\]
we can ensure that 
\[
	\lim_{j \to \infty} (\tilde{a}_{L_j+N-1} - \tilde{a}_{L_j-1}) = 0.
\]
Having chosen subsequence $(L_j : j \in \NN_0)$, we apply Theorem \ref{thm:5} to deduce that the sequence
$(|S_{L_j}| : j \in \NN_0)$ converges uniformly on $\sS^1 \times K$ to a continuous function $|S|$. Consequently,
by \eqref{eq:112},
\[
	\lim_{j \to \infty} \|M_{L_j-1}(x) e_2 \|^2 \cdot |S_{L_j}(\eta_{L_j}(x), x)| = \|M(x) e_2 \|^2 \cdot |S(\eta(x), x)|,
\]
which leads to
\[
	\lim_{j \to \infty}
	\sup_{x \in K}{
	\Big|
	a_{L_j+N-1}^{3/2} \big|\tilde{\scrD}_{L_j}(x)\big|
	-
	\|M(x) e_2 \|^2 \cdot \big|S(\eta(x), x)\big|
	\Big|}= 0.
\]
Let us recall the definition of $\tau$ and $\Lambda_-$ in \eqref{eq:61b} and \eqref{eq:56}, respectively.
In view of Theorem \ref{thm:4}, we obtain the following statement.
\begin{theorem}
	\label{thm:9}
	Let $N$ be a positive integer and $i \in \{0, 1, \ldots, N-1\}$. 
	Let $(\tilde{a}_n : n \in \NN_0)$ and $(\tilde{b}_n : n \in \NN_0)$ be Jacobi parameters such that
	\[
		\tilde{a}_n = a_n (1 + \xi_n), \qquad \tilde{b}_n = b_n (1 + \zeta_n),
	\]
	where $(a_n : n \in \NN_0)$ and $(b_n : n \in \NN_0)$ are $N$-periodically modulated Jacobi parameters so that 
	$\frakX_0(0)$ is a non-trivial parabolic element, satisfying
	\[
		\bigg(\frac{\alpha_{n-1}}{\alpha_n} a_n - a_{n-1} : n \in \NN \bigg),
		\bigg(\frac{\beta_n}{\alpha_n} a_n - b_n : n \in \NN\bigg),
		\bigg(\frac{1}{\sqrt{a_n}} : n \in \NN\bigg) \in \calD_1^N,
	\]
	and
	\[
		\sum_{n = 0}^\infty \sqrt{a_n} (|\xi_n| + |\zeta_n|) < \infty,
	\]
	for certain real sequences $(\xi_n : n \in \NN_0)$ and $(\zeta_n : n \in \NN_0)$.
	Then there is $(L_j : j \in \NN_0)$ an increasing sequence of integers, $L_j \equiv i \bmod N$, such that
	\[
		\tilde{g}_i(x) 
		= \lim_{j \to \infty} \tilde{a}_{L_j+N-1}^{3/2} \big|\tilde{\scrD}_{L_j}(x)\big|, \qquad x \in \Lambda_-
	\]
	where the sequence converges locally uniformly with respect to $x \in \Lambda_-$, defines a continuous positive
	function. If
	\[
		\lim_{j \to \infty} \big(\tilde{a}_{L_j+N-1} - \tilde{a}_{L_j-1}\big) = 0,
	\]
	then the measure $\tilde{\mu}$ is purely absolutely continuous on $\Lambda_-$ with the density
	\[
		\tilde{\mu}'(x) = \frac{\sqrt{\alpha_{i-1} |\tau(x)|}}{\pi \tilde{g}_i(x)},
		\qquad x \in \Lambda_-.
	\]
\end{theorem}
Next, let us observe that if $K \subset \RR \setminus \{x_0\}$ then by \eqref{eq:117} and Theorem \ref{thm:1}, we get
\begin{align*}
	Z_{j+1}^{-1} \tilde{X}_{jN+i} Z_j 
	&= Z_{j+1}^{-1} X_{jN+i} Z_j + a_{jN+i}^{-1/2} Z_{j+1}^{-1} \Delta_{jN+i} Z_j \\
	&= \varepsilon \big(\Id + \vartheta_j R_j\big) + a_{jN+i}^{-1/2} Z_{j+1}^{-1} \Delta_{jN+i} Z_j.
\end{align*}
Since there is $c > 0$ such that for all $j \in \NN$,
\[
	\| Z_{j+1}^{-1} \| \leq c  a_{jN+i}^{1/2},
	\quad\text{and}\quad
	\| Z_j \| \leq c,
\]
by setting $V_j = \varepsilon a_{jN+i}^{-1/2} Z_{j+1}^{-1} \Delta_{jN+i} Z_j$, we get
\begin{equation}
	\label{eq:58}
	Z_{j+1}^{-1} \tilde{X}_{jN+i} Z_j = \varepsilon \big(\Id + \vartheta_j R_j + V_j\big)
\end{equation}
where $(R_j)$ is a sequence from $\calD_1\big(K, \Mat(2, \RR)\big)$ convergent on $K$ to $\calR_i$, and
\begin{equation}
	\label{eq:59}
	\sum_{j = 1}^\infty \sup_K \|V_j\| < \infty.
\end{equation}
Moreover, by \eqref{eq:71} we have
\[
	\sup_K \| V_{L_j} \| \leq c a_{L_j+N-1}^{-1}.
\]
Since \cite[Theorem 4.4]{Discrete} allows perturbation satisfying \eqref{eq:59} we can repeat the proof of
Theorem~\ref{thm:8} to get the following result.
\begin{theorem}
	\label{thm:10}
	Let $N$ be a positive integer. Let $\tilde{A}$ be the Jacobi matrix associated with
	Jacobi parameters $(\tilde{a}_n : n \in \NN_0)$ and $(\tilde{b}_n : n \in \NN_0)$ such that
	\[
		\tilde{a}_n = a_n (1 + \xi_n), \qquad \tilde{b}_n = b_n (1 + \zeta_n),
	\]
	where $(a_n : n \in \NN_0)$ and $(b_n : n \in \NN_0)$ are $N$-periodically modulated Jacobi parameters so that 
	$\frakX_0(0)$ is a non-trivial parabolic element, satisfying
	\[
        \bigg(\frac{\alpha_{n-1}}{\alpha_n} a_n - a_{n-1} : n \in \NN \bigg),
        \bigg(\frac{\beta_n}{\alpha_n} a_n - b_n : n \in \NN\bigg),
        \bigg(\frac{1}{\sqrt{a_n}} : n \in \NN\bigg) \in \calD_1^N,
    \]
    and
	\[
		\sum_{n = 0}^\infty \sqrt{a_n} (|\xi_n| + |\zeta_n|) < \infty,
	\]
	for certain real sequences $(\xi_n : n \in \NN_0)$ and $(\zeta_n : n \in \NN_0)$.
	Then
	\[
		\sigmaEss(\tilde{A}) \cap \Lambda_+ = \emptyset.
	\]
\end{theorem}
Next, we study the asymptotic behavior of polynomials $(\tilde{p}_n : n \in \NN_0)$. Since the Carleman's condition
\eqref{eq:37} is satisfied, the orthonormalizing measure $\tilde{\mu}$ is unique.
\begin{theorem}
	\label{thm:11}
	Let $N$ be a positive integer and $i \in \{0, 1, \ldots, N-1\}$. Let $(\tilde{a}_n : n \in \NN_0)$ and
	$(\tilde{b}_n : n \in \NN_0)$ be Jacobi parameters such that
	\[
		\tilde{a}_n = a_n (1 + \xi_n), \qquad \tilde{b}_n = b_n (1 + \zeta_n),
	\]
	where $(a_n : n \in \NN_0)$ and $(b_n : n \in \NN_0)$ are $N$-periodically modulated Jacobi parameters so that
	$\frakX_0(0)$ is a non-trivial parabolic element, satisfying
	\[
        \bigg(\frac{\alpha_{n-1}}{\alpha_n} a_n - a_{n-1} : n \in \NN \bigg),
	    \bigg(\frac{\beta_n}{\alpha_n} a_n - b_n : n \in \NN\bigg),
	    \bigg(\frac{1}{\sqrt{a_n}} : n \in \NN\bigg) \in \calD_1^N,
	\]
	and
	\[
		\sum_{n = 0}^\infty \sqrt{a_n} (|\xi_n| + |\zeta_n|) < \infty,
	\]
	for certain real sequences $(\xi_n : n \in \NN_0)$ and $(\zeta_n : n \in \NN_0)$.
	If there is $(L_j : j \in \NN_0)$ an increasing sequence of integers, $L_j \equiv i \bmod N$, such that
	\[
		\lim_{j \to \infty} \big(\tilde{a}_{L_j + N-1} -\tilde{a}_{L_j-1}\big) = 0,
	\]
	then for each compact subset $K \subset \Lambda_-$, there are $j_0 \in \NN$ and $\tilde{\chi}: K \rightarrow \RR$,
	so that
	\begin{equation}
		\label{eq:62}
		\lim_{j \to \infty}
		\sup_{x \in K}{
		\bigg|
		\sqrt[4]{\tilde{a}_{(j+1)N+i-1}} \tilde{p}_{jN+i}(x)
		-
		\sqrt{\frac{\big| [\frakX_i(0)]_{2,1} \big|}{ \pi \tilde{\mu}'(x) \sqrt{\alpha_{i-1} |\tau(x)|}}}
		\sin\Big(\sum_{k=j_0}^{j-1} \theta_k(x) + \tilde{\chi}(x)\Big)\bigg|
		}
		=0
	\end{equation}
	where $\theta_k$ are determined in Theorem \ref{thm:3}.
\end{theorem}
\begin{proof}
	Fix a compact set $K \subset \Lambda_-$. In view of \eqref{eq:63}, Theorem \ref{thm:3} implies that
	\begin{align}
		\nonumber
		\frac{\tilde{p}_{jN+i}(x)}{\prod_{k = j_0}^j \lambda_k(x)}
		&=
		\big\|M_{jN+i-1}(x) e_2 \big\|
		\frac{|\vphi(\eta_{jN+i}(x), x)|}{\sqrt{\alpha_{i-1} |\tau(x)|}} 
		\sin
		\Big(
		\sum_{k = j_0}^j \theta_k(x) + \arg \vphi(\eta_{jN+i}(x), x)\Big)
		+o_K(1) \\
		\label{eq:77}
		&=
		\big\| M(x) e_2 \big\|
		\frac{|\vphi(\eta(x), x)|}{\sqrt{\alpha_{i-1} |\tau(x)|}}
		\sin
		\Big(
		\sum_{k = j_0}^j \theta_k(x) + \arg \vphi(\eta_{jN+i}(x), x)\Big)
		+o_K(1)
	\end{align}
	where we have used \eqref{eq:112} and continuity of $\vphi$. Our aim is to compute the function
	$|\vphi(\eta(x), x)|$. To do so, we can work with the subsequence $(L_j : j \in \NN)$. With no loss of generality
	we assume \eqref{eq:71}. The reasoning follows the same method as in Theorem \ref{thm:6}. In view of \eqref{eq:146}
	\[
		\vphi(\eta(x), x) = \lim_{j \to \infty} \sqrt{a_{L_j+N-1}} 
		\phi_{L_j}\big(\eta(x), x\big), \qquad x \in K,
	\]
	where
	\[
		\phi_{L}(\eta(x), x) = 
		\frac{\big\langle \big( X_{L}(x) - \overline{\lambda_{\lfloor L/N\rfloor} (x)} \Id \big) 
		\vec{u}_{L}(\eta(x), x), 
		e_2 \big\rangle}{\prod_{k=j_0}^{\lfloor L/N\rfloor -1} \lambda_k(x)}.
	\]
	Observe that by \eqref{eq:50} and Corollary \ref{cor:3}, we have
	\begin{align*}
		\sqrt{a_{L_j+N-1}}
		\Big|
		\phi_{L_j}(\eta(x), x) - \phi_{L_j}(\eta_{L_j}(x), x)
		\Big|
		\leq
		c \sup_{x \in K} \big\|\eta(x) - \eta_{L_j}(x)\big\|,
	\end{align*}
	thus by \eqref{eq:112}
	\begin{equation}
		\label{eq:70}
		\big\|M(x) e_2 \big\| \vphi(\eta(x), x) 
		=
		\lim_{j \to \infty} \sqrt{a_{L_j+N-1}} \big\|M_{L_j-1}(x) e_2 \big\| \phi_{L_j}\big(\eta_{L_j}(x), x\big).
	\end{equation}
	Observe that by \eqref{eq:63}
	\begin{equation}
		\label{eq:79}
		\big\|M_{L-1}(x) e_2 \big\|\phi_L\big(\eta_L(x), x\big)
		=
		\frac{\big\langle \big( X_L(x) - \overline{\lambda_{\lfloor L/N\rfloor}(x)} \Id \big) 
		\vec{\tilde{p}}_{L}(x), 
		e_2 \big\rangle}{\prod_{k=j_0}^{\lfloor L/N\rfloor-1} \lambda_k(x)}.
	\end{equation}
	For $m \in \NN$ and $L \equiv i \bmod N$, we set
	\begin{equation}
		\label{eq:81}
		\tilde{\phi}^{L}_m(x) 
		=
		\frac{\big\langle \big( \tilde{X}_{L+N}^{L}(x) - \overline{\tilde{\lambda}^{L}_{L+N}(x)} \Id \big)
		\vec{\tilde{p}}^{L}_{L+mN}(x), 
		e_2 \big\rangle}{\big( \tilde{\lambda}^{L}_{L+N}(x) \big)^{m-1} \prod_{k=j_0}^{\lfloor L/N\rfloor} 
		\lambda_k(x)},
		\qquad x \in K,
	\end{equation}
	where $\tilde{\lambda}^L_{L+N}$ is the eigenvalue of $\tilde{X}^L_{L+N}$ with positive imaginary part.
	By the same lines of reasoning as in \cite[Claim 3]{SwiderskiTrojan2019}, one can show that 
	$\tilde{\phi}^{L}_m = \tilde{\phi}^{L}_1$ for all $m \geq 1$. Next, we claim that the following holds true.
	\begin{claim}
		\label{clm:4}
		\begin{equation}
			\label{eq:78}
			\lim_{j \to \infty} \sqrt{\tilde{a}_{L_j + N-1}} 
			\sup_{x \in K}{\Big|
			\tilde{\phi}^{L_j}_1(x) - 
			\big\|M_{L_j+N-1}(x) e_2 \big\|\phi_{L_j+N}\big(\eta_{L_j+N}(x), x\big) \Big|} = 0.
		\end{equation}
	\end{claim}
	By \eqref{eq:79} and \eqref{eq:81}, we have
	\[
		\tilde{\phi}^{L_j}_1(x)
		-
		\big\|M_{L_j+N-1}(x) e_2 \big\|\phi_{L_j+N}\big(\eta_{L_j+N}(x), x\big) 
		=
		\frac{\langle W_j(x) \vec{\tilde{p}}_{L_j+N}(x), e_2 \rangle}{\prod_{k=j_0}^{\lfloor L_j/N\rfloor} \lambda_k(x)}
	\]
	where
	\begin{equation}
		\label{eq:82}
		W_j = 
		\Big( \tilde{X}^{L_j}_{L_j+N} - X_{L_j+N} \Big) +
		\Big( \overline{\lambda_{\lfloor L_j/N \rfloor +1}} - \overline{\tilde{\lambda}^{L_j}_{L_j+N}} \Big) \Id.
	\end{equation}
	Hence,
	\[
		\Big|
		\tilde{\phi}^{L_j}_1(x) -
		\big\|M_{L_j+N-1}(x) e_2 \big\|\phi_{L_j+N}\big(\eta_{L_j+N}(x), x\big) 
		\Big|
		\leq 
		c \big\|W_j(x) \big\| \frac{\big\|\vec{\tilde{p}}_{L_j+N}(x)\big\|}{\prod_{k = j_0}^{\lfloor L_j/N\rfloor} 
		|\lambda_k(x)|}.
	\]
	By \eqref{eq:119} and \eqref{eq:50} we get
	\[
		\frac{\big\|\vec{\tilde{p}}_{L_j+N}(x)\big\|}{\prod_{k = j_0}^{\lfloor L_j/N\rfloor} |\lambda_k(x)|}
		\leq c
	\]
	for all $x \in K$ and $j > j_0$. Next, we write
	\begin{align*}
		\big\|\tilde{X}^{L_j}_{L_j+N} - X_{L_j+N}\big\|
		\leq
		\big\|\tilde{X}^{L_j}_{L_j+N} - \tilde{X}_{L_j+N}\big\|
		+
		\big\|\tilde{X}_{L_j+N} - X_{L_j+N}\big\|,
	\end{align*}
	thus by Lemma \ref{lem:2}, \eqref{eq:117} and \eqref{eq:71}, we obtain
	\begin{equation}
		\label{eq:111}
		\lim_{j \to \infty} 
		\tilde{a}_{L_j + N-1} \big\|\tilde{X}^{L_j}_{L_j+N} - X_{L_j+N}\big\| = 0.
	\end{equation}
	It remains to show that
	\[
		\lim_{j \to \infty}
		\sqrt{\tilde{a}_{L_j+N-1}} \Big| \lambda_{\lfloor L_j/N \rfloor+1} - \tilde{\lambda}^{L_j}_{L_j+N} \Big| = 0,
	\]
	which can be deduced from
	\begin{equation}
		\label{eq:84}
		\lim_{j \to \infty} \sqrt{\tilde{a}_{L_j+N-1}} \sup_K \big| \tr \tilde{X}^{L_j}_{L_j+N} 
		- \tr Y_{\lfloor L_j/N \rfloor + 1}\big| = 0,
	\end{equation}
	and
	\begin{equation}
		\label{eq:89}
		\lim_{j \to \infty} \tilde{a}_{L_j+N-1} \sup_K \big| \discr \tilde{X}^{L_j}_{L_j+N}
		-
		\discr Y_{\lfloor L_j/N \rfloor + 1}\big| = 0.
	\end{equation}
	We write
	\[
		\tilde{X}^{L_j}_{L_j+N} - Y_{\lfloor L_j/N \rfloor + 1}
		=
		\big(\tilde{X}^{L_j}_{L_j+N} - X_{L_j+N}\big) + \big(X_{L_j+N} - Y_{\lfloor L_j/N \rfloor + 1}\big),
	\]
	thus \eqref{eq:84} is a consequence of \eqref{eq:111} and \eqref{eq:51a'}. Next, by \eqref{eq:114},
	\[
		\big|
		\discr \tilde{X}^{L_j}_{L_j+N} - \discr X_{L_j+N}
		\big|
		\leq
		c
		\big\|\tilde{X}^{L_j}_{L_j+N} - X_{L_j+N}\big\|
	\]
	thus \eqref{eq:89} follows by \eqref{eq:111} and \eqref{eq:52}. Summarizing, we showed that
	\[
		\lim_{j \to \infty} \sqrt{\tilde{a}_{L_j+N-1}} \sup_{x \in K} \|W_j(x) \| = 0,
	\]
	and hence \eqref{eq:78} follows. 

	Our last step is to justify the following claim. 
	\begin{claim}
		\begin{equation}
			\label{eq:76}
			\big\|M(x) e_2 \big\|^2
			|\vphi(\eta(x), x)|^2
			=
			\frac{1}{a_{j_0 N + i-1} \sinh \vartheta_{j_0}(x)} 
			\cdot
			\frac{\big| [\frakX_i(0)]_{2,1} \big| \cdot \alpha_{i-1} |\tau(x)|}{\pi \tilde{\mu}'(x)}.
		\end{equation}
	\end{claim}
	For the proof, let us observe that by \eqref{eq:70} and Claim \ref{clm:4} we have
	\[
		\big\|M(x) e_2 \big\|^2 |\vphi(\eta(x), x)|^2 = 
		\lim_{j \to \infty}
		\Big| \sqrt{\tilde{a}_{L_j+N-1}} \phi^{L_j}_1(x) \Big|^2.
	\]
	In view of \cite[formula (6.14)]{SwiderskiTrojan2019}
	\begin{align*}
		\bigg| 
		\phi^{L_j}_1(x) \cdot
		\prod_{k=j_0}^{\lfloor L_j/N \rfloor} \lambda_k(x) 
		\bigg|^2 
		&=
		\frac{1}{2 \pi \tilde{a}_{L_j+N-1} \tilde{\mu}'_{L_j}(x)} 
		\Big| \big[ \tilde{X}^{L_j}_{L_j+N}(x) \big]_{2,1} \Big|
		\sqrt{-\discr \tilde{X}^{L_j}_{L_j+N}(x)} \\
		&=
		\frac{1}{2 \pi \tilde{a}_{L_j+N-1}^{3/2} \tilde{\mu}'_{L_j}(x)} 
		\Big| \big[ \tilde{X}^{L_j}_{L_j+N}(x) \big]_{2,1} \Big|
		\sqrt{-\tilde{a}_{L_j+N-1} \discr \tilde{X}^{L_j}_{L_j+N}(x)}.
	\end{align*}
	By Lemma \ref{lem:2}
	\[
		\lim_{j \to \infty} \tilde{a}_{L_j+N-1} \discr \tilde{X}^{L_j}_{L_j+N}(x)
		=
		\lim_{j \to \infty} \tilde{a}_{L_j+N-1} \discr \tilde{X}_{L_j+N}(x).
	\]
	Using \eqref{eq:58} and \eqref{eq:71}, we can repeat the proof of Corollary \ref{cor:2}, to get
	\[
		\lim_{j \to \infty} \tilde{a}_{L_j+N-1} \discr \tilde{X}_{L_j+N}(x)
		=
		4 \alpha_{i-1} \tau(x).
	\]
	By Theorem \ref{thm:9} and \eqref{eq:119}, we can apply Corollary \ref{cor:1} to obtain
	\[
		\lim_{j \to \infty}
		\sqrt{a_{L_j+N-1}}
		\bigg| 
		\sqrt{\tilde{a}_{L_j+N-1}} 
		\phi^{L_j}_1(x) \cdot
		\prod_{k=j_0}^{\lfloor L_j/N \rfloor} \lambda_k(x) 
		\bigg|^2
		=
		\frac{\big| [ \frakX_i(0)]_{2,1} \big| \sqrt{\alpha_{i-1} |\tau(x)|}}{\pi \tilde{\mu}'(x)}.
	\]
	Finally, the claim follows by \eqref{eq:50}. 

	Now, by inserting \eqref{eq:76} into \eqref{eq:77} and using \eqref{eq:50} we conclude the proof of the theorem.
\end{proof}

Having proven asymptotic formula for orthogonal polynomials $(\tilde{p}_n : n \in \NN_0)$, we can repeat the 
proof of Theorem \ref{thm:7} to get the following result. 
\begin{theorem}
	\label{thm:12}
	Let $N$ be a positive integer and $i \in \{0, 1, \ldots, N-1\}$. Let $(\tilde{a}_n : n \in \NN_0)$ and
	$(\tilde{b}_n : n \in \NN_0)$ be Jacobi parameters such that
	\[
		\tilde{a}_n = a_n (1 + \xi_n), \qquad \tilde{b}_n = b_n (1 + \zeta_n),
	\]
	where $(a_n : n \in \NN_0)$ and $(b_n : n \in \NN_0)$ are $N$-periodically modulated Jacobi parameters so that
	$\frakX_0(0)$ is a non-trivial parabolic element, satisfying
	\[
        \bigg(\frac{\alpha_{n-1}}{\alpha_n} a_n - a_{n-1} : n \in \NN \bigg),
	    \bigg(\frac{\beta_n}{\alpha_n} a_n - b_n : n \in \NN\bigg),
	    \bigg(\frac{1}{\sqrt{a_n}} : n \in \NN\bigg) \in \calD_1^N,
	\]
	and
	\[
		\sum_{n = 0}^\infty \sqrt{a_n} (|\xi_n| + |\zeta_n|) < \infty,
	\]
	for certain real sequences $(\xi_n : n \in \NN_0)$ and $(\zeta_n : n \in \NN_0)$.
	If there is $(L_j : j \in \NN_0)$ a sequence of integers $L_j \equiv i \bmod N$, such that
	\[
		\lim_{j \to \infty} \big(\tilde{a}_{L_j+N-1} -\tilde{a}_{L_j-1}\big) = 0,
	\]
	then
	\[
		\lim_{n \to \infty} 
		\frac{1}{\tilde{\rho}_{n}} \tilde{K}_{n} \bigg(x + \frac{u}{\tilde{\rho}_n}, x + \frac{v}{\tilde{\rho}_n} \bigg)
		=
		\frac{\upsilon(x)}{\tilde{\mu}'(x)}
		\cdot
		\sinc\big((u-v) \pi \upsilon(x)\big)
	\]
	locally uniformly with respect to $(x, u, v) \in \Lambda_- \times \RR^2$, where $\upsilon$ is defined in 
	\eqref{eq:85} and
	\[
		\tilde{\rho}_n = \sum_{k=0}^n \sqrt{\frac{\alpha_k}{\tilde{a}_k}}.
	\]
\end{theorem}

\section{Examples}
\label{sec:10}
\subsection{Period $N=1$}
\label{sec:10:1}
The following corollary is an easy consequence of Theorems \ref{thm:9} and \ref{thm:10}.

\begin{corollary} 
	\label{cor:4}
	Let $(\tilde{a}_n : n \in \NN_0)$ and $(\tilde{b}_n : n \in \NN_0)$ be Jacobi parameters such that
	\begin{equation} 
		\label{eq:91a}
		\tilde{a}_n = a_n (1 + \xi_n), \qquad 
		\tilde{b}_n = b_n (1 + \zeta_n)
	\end{equation}
	where $(a_n : n \in \NN_0)$ and $(b_n : n \in \NN_0)$ are Jacobi parameters satisfying
	\begin{equation} 
		\label{eq:91b}
        \big(a_n - a_{n-1} : n \in \NN \big),
	    \big(q a_n - b_n : n \in \NN \big),
	    \bigg(\frac{1}{\sqrt{a_n}} : n \in \NN\bigg) \in \calD_1,
	\end{equation}
	and
	\begin{equation} 
		\label{eq:91c}
		\sum_{n = 0}^\infty \sqrt{a_n} (|\xi_n| + |\zeta_n|) < \infty,
	\end{equation}
	for certain real sequences $(\xi_n : n \in \NN_0)$ and $(\zeta_n : n \in \NN_0)$
	and some $q \in \{-2,2\}$. 
	Suppose that
	\begin{equation} 
		\label{eq:91d}
		\lim_{n \to \infty} \big( a_n - a_{n-1} \big) = 0, \qquad
		\lim_{n \to \infty} \big( q a_n - b_n \big) = r, \qquad
		\lim_{n \to \infty} a_n = \infty.
	\end{equation}
	Then the corresponding Jacobi matrix $\tilde{A}$ satisfies
	\[
		\sigmaAC(\tilde{A}) = \sigmaEss(\tilde{A}) = \overline{\Lambda_-} 
		\quad \text{and} \quad 
		\sigmaS(\tilde{A}) \cap \Lambda_- = \emptyset
	\]
	where
	\[
		\Lambda_- =
		\begin{cases}
			(-r, \infty) & q = 2,\\
			(-\infty, -r) & q = -2.
		\end{cases}
	\]
\end{corollary}

Let us compare Corollary~\ref{cor:4} with the results already known in the literature. In the article \cite{Motyka2014}
the author studied Jacobi parameters of the form \eqref{eq:91a} for $b_n = -2 a_n$ and the sequence $(a_n : n \in \NN_0)$ 
satisfying \eqref{eq:91c}, \eqref{eq:91d} and
\begin{equation}
	\label{eq:92a}
	\bigg( \frac{a_n - a_{n-1}}{a_n^{3/2}} : n \in \NN \bigg) \in \ell^1, 
	\qquad
	(a_n - a_{n-1} : n \in \NN) \in \calD_1.
\end{equation}
Under the above hypotheses the asymptotic formula for generalized eigenvectors of $\tilde{A}$ is obtained in 
\cite{Motyka2014}. Let us observe that
\begin{align*}
	\bigg| \frac{1}{\sqrt{a_{n+1}}} - \frac{1}{\sqrt{a_{n}}} \bigg| 
	&=
	\frac{|a_{n+1} - a_n|}{(\sqrt{a_{n+1}} + \sqrt{a_n}) \sqrt{a_{n+1} a_n}} \\
	&=
	\frac{|a_{n+1} - a_n|}{a_{n+1}^{3/2} \big( 1 + \sqrt{\tfrac{a_n}{a_{n+1}}} \big) \sqrt{\tfrac{a_n}{a_{n+1}}}} \\
	&\asymp
	\frac{|a_{n+1} - a_n|}{a_{n+1}^{3/2}},
\end{align*}
that is \eqref{eq:92a} and \eqref{eq:91b} are equivalent. Consequently, we can apply Corollary~\ref{cor:4}. Moreover,
in view of Theorem~\ref{thm:11}, we obtain the asymptotic behavior of the corresponding orthogonal polynomials
$(\tilde{p}_n : n \in \NN_0)$. 

The Jacobi parameters satisfying the hypotheses of \cite{Motyka2014} are further studied in \cite{Motyka2015}.
In particular, it is proved that $\sigmaP(\tilde{A}) \subset (0, \infty)$, and moreover,
$\sigmaAC(\tilde{A}) = (-\infty, 0]$ and $\sigmaS(\tilde{A}) \cap (-\infty, 0) = \emptyset$ provided that
\begin{equation} 
	\label{eq:92b}
	\bigg( \frac{1}{a_n} : n \in \NN \bigg),
	\bigg( \frac{a_n - a_{n-1}}{\sqrt{a_n}} : n \in \NN \bigg) \in \ell^2.
\end{equation}
Our Corollary shows that the hypothesis \eqref{eq:92b} can be dropped and at the same time it provides a stronger
conclusion that $\sigmaEss(\tilde{A}) \cap (0, \infty) = \emptyset$. It is also more flexible because we do not need
to assume that $a_n = -2b_n$. Let us emphasize that no analogue of Theorem~\ref{thm:12} was studied before.

Let us also mention two earlier articles \cite{Janas2006} and \cite{Janas2009} where the authors study Jacobi
parameters falling into the class considered in Corollary~\ref{cor:4} for 
\begin{equation}
	\label{eq:120}
	a_n = (n+1)^\gamma, \qquad
	b_n = -2(n+1)^\gamma,
\end{equation}
where $\gamma \in (\tfrac{1}{3},\tfrac{2}{3})$. The results proven there are analogues of 
\cite{Motyka2014, Motyka2015}. Recently, in \cite{Naboko2019}, a variant of Theorem~\ref{thm:11} is obtained for
Jacobi parameters \eqref{eq:120} and $\gamma \in (0,1)$.

In \cite{Janas2001}, the authors proved that for Jacobi parameters
\[
	a_n = n+1+\gamma, \qquad
	b_n = -2(n+1 + \gamma)
\]
where $\gamma \in (-1, \infty)$, we have $\sigmaAC(A) = \sigmaEss(A) = (-\infty, -1]$. This case lies on the borderline
of our methods, and it is \emph{not} covered by Corollary~\ref{cor:4}. Finally, let us mention the recent article
\cite{Yafaev2020a} studying the Jacobi parameters of the form
\[
	\tilde{a}_n = (n+1)^\gamma \big( 1 + \calO(n^{-2}) \big), \qquad
	\tilde{b}_n = q (n+1)^\gamma \big( 1 + \calO(n^{-2}) \big),
\]
for $q \in \{-2,2\}$ and $\gamma \in \big( \tfrac{3}{2}, \infty \big)$. The author describes the asymptotic formula
for $(\tilde{p}_n : n \in \NN_0)$, and shows that $\sigmaEss(A) = \emptyset$ provided that $\tilde{A}$ is self-adjoint.
This case is also \emph{not} covered by our results.

The following proposition allows to construct a large class of sequences $(a_n : n \in \NN_0)$ satisfying the hypotheses of Corollary~\ref{cor:4}.
\begin{proposition} \label{prop:5}
Let $(a_n : n \in \NN_0)$ be a positive sequence such that:
\begin{enumerate}[(a)]
\item it is eventually increasing, i.e. there exists $n_0 \geq 1$ such that
$\begin{aligned}[b]
	a_n \leq a_{n+1}
\end{aligned}$
for any $n \geq n_0$, \label{eq:prop:5a}

\item it is eventually concave, i.e. there exists $n_0 \geq 1$ such that
$\begin{aligned}[b]
	\frac{a_{n-1} + a_{n+1}}{2} \leq a_{n}
\end{aligned}$
for any $n \geq n_0$. \label{eq:prop:5b}
\end{enumerate}
Then
\begin{equation} \label{eq:123}
	\big( a_n - a_{n-1} : n \in \NN \big), 
	\bigg( \frac{1}{\sqrt{a_n}} : n \in \NN \bigg) \in \calD_1.
\end{equation}
If additionally
\begin{equation} \label{eq:124}
	\lim_{n \to \infty} \frac{a_n}{n} = 0, 
\end{equation}
then
\begin{equation} \label{eq:125}
	\lim_{n \to \infty} (a_{n+1} - a_n) = 0.
\end{equation}
\end{proposition}
\begin{proof}
Without loss of generality let $n_0 \geq 1$ be such that both \eqref{eq:prop:5a} and \eqref{eq:prop:5b} are satisfied.

By \eqref{eq:prop:5a} the following limit
\[
	a_\infty = \lim_{n \to \infty} a_n
\]
exists and $a_\infty \in (0, +\infty]$. Observe that by \eqref{eq:prop:5a}
\begin{equation} \label{eq:121}
	\sum_{n=n_0}^\infty 
	\bigg| 
		\frac{1}{\sqrt{a_{n+1}}} - \frac{1}{\sqrt{a_{n}}} 
	\bigg|
	=
	\sum_{n=n_0}^\infty 
	\bigg( 
		\frac{1}{\sqrt{a_{n}}}  - \frac{1}{\sqrt{a_{n+1}}}
	\bigg)
	=
	\frac{1}{\sqrt{a_{n_0}}} - \frac{1}{\sqrt{a_{\infty}}} < \infty.
\end{equation}
Next, set 
\begin{equation} \label{eq:126}
	d_n = a_n - a_{n-1}.
\end{equation}
Then by \eqref{eq:prop:5a} and \eqref{eq:prop:5b} we have
\[
	0 \leq d_{n+1} 
	\leq 
	d_n, 
	\quad n \geq n_0,
\]
which implies that the following limit exists
\begin{equation} \label{eq:122}
	d_\infty = \lim_{n \to \infty} d_n
\end{equation}
and $d_\infty \in [0, d_{n_0}]$.
Thus
\[
	\sum_{n=n_0}^\infty 
	| d_{n+1} - d_n |
	=
	\sum_{n=n_0}^\infty 
	( d_n - d_{n+1} )
	=
	d_{n_0} - d_\infty < \infty,
\]
which together with \eqref{eq:121} and \eqref{eq:126} implies \eqref{eq:123}. Finally, if \eqref{eq:124} is satisfied, then by Stolz--Ces\'aro theorem and \eqref{eq:122}
\[
	0 = 
	\lim_{n \to \infty} \frac{a_n}{n} =
	\lim_{n \to \infty} \big( a_{n+1} - a_{n} \big) = d_\infty,
\]
which implies \eqref{eq:125}. 
\end{proof}

By means of Proposition~\ref{prop:5} we immediately obtain
\begin{corollary} \label{cor:6}
Suppose that $f : [n_0-1, \infty) \to (0,\infty)$ for some $n_0 \geq 1$ is a twice differentiable function such that
\begin{enumerate}[(a)]
	\item for any $x \in (n_0-1,\infty)$ one has $f'(x) > 0$ and $f''(x) < 0$,
	\item $\lim_{x \to \infty} f(x) = \infty$ and $\lim_{x \to \infty} f'(x) = 0$.
\end{enumerate}
Define
\[
	a_n =
	\begin{cases}
		1 & n < n_0 - 1 \\
		f(n) & n \geq n_0-1
	\end{cases}.
\]
Then
\[
	\big( a_n - a_{n-1} : n \in \NN \big), 
	\bigg( \frac{1}{\sqrt{a_n}} : n \in \NN \bigg) \in \calD_1.
\]
Moreover,
\[
	\lim_{n \to \infty} a_n = \infty \quad \text{and} \quad
	\lim_{n \to \infty} (a_n - a_{n-1}) = 0.
\]
\end{corollary}

Below, we provide a few examples satisfying the hypotheses of Corollary~\ref{cor:6}.
\begin{example} \label{ex:8a}
Let $f(x) = \log{x}$. Then it is immediate that the hypotheses of Corollary~\ref{cor:6} are satisfied for any $n_0 \geq 3$.
\end{example}

\begin{example} \label{ex:8b}
Let $f(x) = x^\gamma$ for some $\gamma \in (0,1)$. Then it is immediate that the hypotheses of Corollary~\ref{cor:6} are satisfied for any $n_0 \geq 2$.
\end{example}

\begin{example} \label{ex:8c}
Let $f(x) = x^\gamma \log{x}$ for some $\gamma \in (0,1)$. Since
\begin{align*}
	f'(x) &= x^{\gamma-1} ( \gamma \log{x} + 1 ) \\
	f''(x) &= -x^{\gamma-2} \big( \gamma (1-\gamma) \log{x} + 1-2 \gamma \big)
\end{align*}
the hypotheses of Corollary~\ref{cor:6} are satisfied by taking large enough $n_0$.
\end{example}

\begin{example} \label{ex:8d}
Let $f(x) = \frac{x}{\log{x}}$. Since
\begin{align*}
	f'(x) &= \frac{1}{\log{x}} \bigg( 1 - \frac{1}{\log{x}} \bigg) \\
	f''(x) &= -\frac{1}{x (\log{x})^2} \bigg( 1 - \frac{2}{\log{x}} \bigg)
\end{align*}
the hypotheses of Corollary~\ref{cor:6} are satisfied for any $n_0 \geq 9$.
\end{example}

\subsubsection{Laguerre-type orthogonal polynomials} \label{sec:laguerre}
In this section we provide examples of measures $\mu$ which give rise to Jacobi parameters that satisfy
the hypotheses of Corollary~\ref{cor:4} for $\xi_n \equiv 0$ and $\zeta_n \equiv 0$.

Take $\gamma > -1$ and $\kappa \in \NN$, and consider the purely absolutely continuous probability measure $\mu$ with
the density
\[
	\mu'(x) =
	\begin{cases}
		c_{\gamma, \kappa}
		x^\gamma \exp\big(-x^\kappa\big) & \text{if } x > 0,\\
		0 & \text{otherwise.}
	\end{cases}
\]
where $c_{\gamma, \kappa}$ is the normalizing constant. The case $\kappa = 1$ corresponds to the well-known Laguerre
polynomials. According to \cite[Theorem 2.1 and Remark 2.3]{Vanlessen2007},
\begin{align*}
	a_{n-1} = d_n \bigg(\frac{1}{4} + \frac{\gamma}{8 \kappa n} + \calO \Big( \frac{1}{n^2} \Big)\bigg)
	\qquad\text{and}\qquad
	b_n = d_n \bigg(\frac{1}{2} + \frac{\gamma + 1}{4 \kappa n} + \calO \Big( \frac{1}{n^2} \Big)\bigg)
\end{align*}
where
\[
	d_n = c_0 n^{1/\kappa} \quad \text{for} \quad 
	c_0 = \bigg( \frac{2 (2 \kappa)!!}{\kappa (2 \kappa - 1)!!} \bigg)^{1/\kappa}.
\]
Observe that for $\kappa \geq 2$ it implies
\begin{equation} \label{eq:90}
	a_{n-1} = \frac{1}{4} d_n + e_n, \qquad
	b_n = \frac{1}{2} d_n + \tilde{e}_n,
\end{equation}
where both $(e_n : n \in \NN)$ and $(\tilde{e}_n : n \in \NN_0)$ belong to $\calD_1$ and tend to $0$. We notice that
\begin{align*}
	d_{n+1} - d_n 
	&= 
	c_0 \big( (n+1)^{1/\kappa} - n^{1/\kappa} \big) \\
	&= 
	c_0 n^{1/\kappa} \bigg( \Big( 1+\frac{1}{n} \Big)^{1/\kappa} - 1 \bigg) \\
	&=
	\frac{c_0}{\kappa} n^{1/\kappa - 1} + \calO ( n^{1/\kappa -2}).
\end{align*}
Hence $(d_{n+1} - d_n : n \in \NN_0)$ belongs to $\calD_1$ and tends to $0$. Since
\begin{align*}
	a_n - a_{n-1} &= 
	\frac{1}{4} \big( d_{n+1} - d_n \big) +
	e_{n+1} - e_{n} \\
	\intertext{and}
	2 a_n - b_n &= \frac{1}{2} \big( d_{n+1} - d_n \big) +
	e_{n+1} - \tilde{e}_n,
\end{align*}
we conclude that
\[
	(a_n - a_{n-1} : n \in \NN), (2 a_n - b_n : n \in \NN) \in \calD_1
\]
and
\[
	\lim_{n \to \infty} (a_n - a_{n-1}) = 0, \qquad
	\lim_{n \to \infty} (2 a_n - b_n) = 0.
\]
Furthermore, by \eqref{eq:90} the sequence $(a_n : n \in \NN_0)$ is unbounded and eventually increasing, thus
\[
	\bigg( \frac{1}{\sqrt{a_n}} : n \in \NN \bigg) \in \calD_1.
\]
Summarizing, we showed that the hypotheses of Corollary~\ref{cor:4} are satisfied with $\xi_n \equiv 0, \zeta_n \equiv 0$ for any $\kappa \geq 2$.

\subsection{Periodic modulations}
\label{sec:10:2}
The following corollary easily follows from Theorems \ref{thm:A} and \ref{thm:B}.

\begin{corollary}
	\label{cor:5}
	Let $N$ be a positive integer. Let $(\alpha_n : n \in \ZZ)$ and $(\beta_n : n \in \ZZ)$ be $N$-periodic Jacobi
	parameters such that $\frakX_0(0)$ is a non-trivial parabolic element. Set
	\begin{equation} 
		\label{eq:31}
		a_n = \alpha_n \tilde{a}_n, \qquad 
		b_n = \beta_n \tilde{a}_n
	\end{equation}
	where the sequence $(\tilde{a}_n : n \in \NN_0)$ satisfies
	\begin{equation}
		\label{eq:94a}
		\big( \tilde{a}_{n} - \tilde{a}_{n-1} : n \in \NN \big), 
		\bigg( \frac{1}{\sqrt{\tilde{a}}_n} : n \in \NN \bigg) \in \calD_1^1,
	\end{equation}
	and
	\begin{equation}
		\label{eq:94b}
		\lim_{n \to \infty} \tilde{a}_n = \infty, \qquad
		\lim_{n \to \infty} \big( \tilde{a}_{n+1} - \tilde{a}_n \big) = 0.
	\end{equation}
	Then the corresponding Jacobi matrix $A$ satisfies
	\[
		\sigmaAC(A) = \sigmaEss(A) = \overline{\Lambda_-} 
		\quad \text{and} \quad
		\sigmaS(A) \cap \Lambda_- = \emptyset.
	\]
\end{corollary}

Let us observe that by means of Corollary~\ref{cor:6} we can construct a large class of sequences $(\tilde{a}_n : n \in \NN_0)$ satisfying \eqref{eq:94a} and \eqref{eq:94b}, see in particular Examples \ref{ex:8a}--\ref{ex:8d}. In the next sections we provide a few classes of $(\alpha_n : n \in \ZZ)$ and $(\beta_n : n \in \ZZ)$ for $N=2$ such that $\frakX_0(0)$ is a non-trivial parabolic element. Consequently, it will allow us to compare Corollary~\ref{cor:5} with the results known in the literature.

\subsubsection{Modulation of the main diagonal}
Let $N = 2$, and
\[
	\alpha = (1, 1, 1, 1, \ldots), \qquad
	\beta = (\beta_0, \beta_1, \beta_0, \beta_1, \ldots)
\]
for certain $\beta_0, \beta_1 \in \RR$. Then
\[
	\frakX_0(0) =
	\begin{pmatrix}
		-1 & -\beta_0 \\
		\beta_1 & \beta_0 \beta_1 - 1
	\end{pmatrix} 
	\quad \text{and} \quad
	\frakX_1(0) =
	\begin{pmatrix}
		-1 & -\beta_1 \\
		\beta_0 & \beta_0 \beta_1 - 1
	\end{pmatrix}.
\]
Thus $\det \frakX_0(0) = 1$ and $\tr \frakX_0(0) = \pm 2$, if and only if
\[
	\beta_0 \beta_1 = 0 \quad \text{or} \quad
	\beta_0 \beta_1 = 4.
\]
\begin{example}
	\label{ex:4}
	Take $\beta_0 = q$ and $\beta_1 = 0$ for certain $q > 0$, and select any sequence $(\tilde{a}_n : n \in \NN_0)$
	satisfying \eqref{eq:94a} and \eqref{eq:94b}. Then the Jacobi matrix corresponding to \eqref{eq:31} satisfies
	\[
		\sigmaAC(A) = \sigmaEss(A) = (-\infty, 0] 
		\quad \text{and} \quad
		\sigmaS(A) \cap (-\infty, 0) = \emptyset.	
	\]
\end{example}
Sequences of a form similar to that described in Example \ref{ex:4} were studied in \cite{Damanik2007} where
it was additionally assumed that
\[
	\tilde{a}_n = (n+1)^\gamma
\]
for $\gamma \in (0,1]$. In particular, it was shown that
\begin{itemize}
	\item the Jacobi matrix $A$ is absolutely continuous on $(-\infty, 0)$ if $\gamma \in (\tfrac{2}{3}, 1]$, and
	\item $\sigmaEss(A) \subset (-\infty, 0]$ for any $\gamma \in (0,1]$.
\end{itemize}
In Example~\ref{ex:4} we recover those results for $\gamma \in (0,1)$.
\begin{example} 
	\label{ex:5}
	Take $\beta_0 = q$ and $\beta_1 = 4/q$ for certain $q > 0$, and select a sequence $(\tilde{a}_n : n \in \NN_0)$
	satisfying \eqref{eq:94a} and \eqref{eq:94b}. Then the Jacobi matrix corresponding to \eqref{eq:31} satisfies
	\[
		\sigmaAC(A) = \sigmaEss(A) = [0, \infty) 
		\quad \text{and} \quad
		\sigmaS(A) \cap (0, \infty) = \emptyset.
	\]
\end{example}
Example \ref{ex:5} extends results obtained in \cite{Pchelintseva2008} to sequences
\[
	\tilde{a}_n = (n+1)^\gamma, \quad \gamma \in (0, 1).
\]
Recall that in \cite{Pchelintseva2008} it was proved that if $\gamma = 1$ then the corresponding Jacobi matrix satisfies
\[
	\sigmaAC(A) = \sigmaEss(A) = \big[ \tfrac{4}{\beta_0 + \beta_1}, \infty \big)
	\quad \text{and} \quad
	\sigmaS(A) \cap \big( \tfrac{4}{\beta_0 + \beta_1}, \infty \big) = \emptyset.
\]

\subsubsection{Modulation of the off-diagonal}
Let us consider the following $2$-periodic Jacobi parameters
\[
	\alpha = (\alpha_0, \alpha_1, \alpha_0, \alpha_1, \ldots), \qquad
	\beta = (1, 1, 1, 1, \ldots)
\]
for certain $\alpha_0, \alpha_1 > 0$. Then
\[
	\frakX_0(0) =
	\begin{pmatrix}
		-\frac{\alpha_1}{\alpha_0}	& -\frac{1}{\alpha_0} \\
		\frac{1}{\alpha_0} & -\frac{\alpha_0}{\alpha_1} + \frac{1}{\alpha_0 \alpha_1}
	\end{pmatrix} 
	\quad \text{and} \quad
	\frakX_1(0) =
	\begin{pmatrix}
		-\frac{\alpha_0}{\alpha_1}	& -\frac{1}{\alpha_1} \\
		\frac{1}{\alpha_1} & -\frac{\alpha_1}{\alpha_0} + \frac{1}{\alpha_0 \alpha_1}
	\end{pmatrix}.
\]
The determinant of $\frakX_0(0)$ always equals $1$. For the trace, we have
\[
	\tr \frakX_0(0) = -\frac{\alpha_1}{\alpha_0} -\frac{\alpha_0}{\alpha_1} + \frac{1}{\alpha_0 \alpha_1},
\]
thus $\tr \frakX_0(0) = \pm 2$, if and only if
\[ 
	\frac{|\alpha_0^2 + \alpha_1^2-1|}{\alpha_0 \alpha_1} = 2,
\]
that is
\[
	\alpha_0 + \alpha_1 = 1 \quad \text{or} \quad
	|\alpha_0 - \alpha_1| = 1.
\]

\begin{example}
	\label{ex:6}
	Take $\alpha_0 = 1$ and $\alpha_1 = 1-q$ for certain $q \in (0, 1)$. Let $(\tilde{a}_n \in \NN_0)$ be a
	sequence satisfying \eqref{eq:94a} and \eqref{eq:94b}. Then the Jacobi matrix corresponding to \eqref{eq:31} satisfies
	\[
		\sigmaAC(A) = \sigmaEss(A) = [0, \infty) 
		\quad \text{and} \quad
		\sigmaS(A) \cap (0, \infty) = \emptyset.
	\]
\end{example}
In \cite{Simonov2007}, the author studied Jacobi parameters of the form similar to that described in Example \ref{ex:6}
for
\[
	\tilde{a}_n = n+1.
\]
He proved that 
\[
	\sigmaAC(A) = \sigmaEss(A) = [\tfrac{1}{2}, \infty) 
	\quad \text{and} \quad
	\sigmaS(A) \cap \big( \tfrac{1}{2}, \infty \big) = \emptyset.
\]
Hence, the statements in Example \ref{ex:6} extend the results of \cite{Simonov2007} to sublinear sequences
$(\tilde{a}_n : n \in \NN_0)$.

\begin{example}
	\label{ex:7}
	Take $\alpha_0 = q$ and $\alpha_1 = 1+q$, for certain $q > 0$, and select $(\tilde{a}_n : n \in \NN_0)$ 
	satisfying \eqref{eq:94a} and \eqref{eq:94b}. Then the Jacobi matrix corresponding to \eqref{eq:31} satisfies
	\[
		\sigmaAC(A) = \sigmaEss(A) = (-\infty, 0] 
		\quad \text{and} \quad
		\sigmaS(A) \cap (-\infty, 0) = \emptyset.	
	\]
\end{example}
In \cite{Naboko2009}, the authors investigated Jacobi parameters of the form similar to that described in Example
\ref{ex:7} by taking
\[
	\tilde{a}_n = (n+1)^\gamma, \quad \gamma \in (\tfrac{1}{2}, \tfrac{2}{3}).
\]
They proved that
\[
	\sigmaEss(A) \subset (-\infty, 0]
\]
which is extended and generalized in Example \ref{ex:7}.

\begin{bibliography}{jacobi}
	\bibliographystyle{amsplain}
\end{bibliography}

\end{document}